\documentclass{article}

\usepackage{amsfonts, amssymb, amsmath, amsthm, verbatim, esint,graphicx, color, mathtools}

\newcommand{\R}{\mathbb{R}}
\newcommand{\C}{\mathbb{C}}
\newcommand{\N}{\mathbb{N}}
\newcommand{\Z}{\mathbb{Z}}

\newcommand{\set}[2]{\left\{#1 \colon #2\right\}}
\newcommand{\dd}[2]{\frac{\partial #1}{\partial #2}}
\renewcommand{\Re}{\mathop{\mathrm{Re}}}
\renewcommand{\Im}{\mathop{\mathrm{Im}}}
\newcommand{\BV}{\mathrm{BV}}
\newcommand{\dirac}{\pmb{\delta}}
\newcommand{\loc}{\mathrm{loc}}
\newcommand{\blank}{{\mkern 2mu\cdot\mkern 2mu}}
\newcommand{\scp}[2]{\left\langle #1, #2 \right\rangle}

\newcommand{\sgn}{\rm sign}

\newcommand{\eps}{\epsilon}

\definecolor{darkgreen}{rgb}{0.0, 0.4, 0.3}

\newcommand{\h}{\mathcal{H}}
\newcommand{\be}{\begin{equation}}
\newcommand{\ee}{\end{equation}}
\newcommand{\f}{\varphi}

\DeclareMathOperator{\arcosh}{arcosh}
\DeclareMathOperator{\curl}{curl}

\newtheorem{theorem}{Theorem}
\newtheorem{lemma}[theorem]{Lemma}
\newtheorem{corollary}[theorem]{Corollary}
\newtheorem{proposition}[theorem]{Proposition}
\newtheorem{definition}[theorem]{Definition}
\newtheorem{remark}[theorem]{Remark}

\newtheorem*{claim}{Claim}
\newenvironment{subproof}[1]{\smallbreak \noindent\textit{#1}.}{\smallskip}

\newcounter{stepno}
\newcommand{\step}{\refstepcounter{stepno}\arabic{stepno}}
\newcounter{caseno}
\renewcommand{\thecaseno}{\Alph{caseno}}
\newcommand{\case}{\refstepcounter{caseno}\thecaseno}
\newcounter{subcaseno}
\renewcommand{\thesubcaseno}{\Alph{caseno}.\arabic{subcaseno}}
\newcommand{\subcase}{\refstepcounter{subcaseno}\thesubcaseno}

\title{Separation of domain walls with nonlocal interaction and their renormalised
energy by $\Gamma$-convergence in thin ferromagnetic films}

\author{
{\Large Radu Ignat}
\footnote{Institut de Math\'ematiques de Toulouse \& Institut Universitaire de France, UMR 5219, Universit\'e de Toulouse, CNRS, UPS
IMT, F-31062 Toulouse Cedex 9, France. Email: Radu.Ignat@math.univ-toulouse.fr} 
\and {\Large Roger Moser}\footnote{Department of Mathematical Sciences,
University of Bath,
Bath BA2 7AY,
UK. Email: r.moser@bath.ac.uk}
}

\begin{document}

\maketitle

\begin{abstract}

We analyse two variants of a nonconvex variational model from micromagnetics with a nonlocal energy functional, depending on a small
parameter $\epsilon > 0$.
The model gives rise to transition layers, called N\'eel walls, and we study their behaviour in the limit $\epsilon \to 0$.
The analysis has some similarity to the theory of Ginzburg-Landau vortices. In particular, it gives rise to
a renormalised energy that determines the interaction (attraction or repulsion) between N\'eel walls to
leading order. But while Ginzburg-Landau vortices show attraction for degrees of the same sign and repulsion
for degrees of opposite signs, the pattern is reversed in this model.

In a previous paper, we determined the renormalised energy for one of the models studied here
under the assumption that the N\'eel walls stay separated from each other.
In this paper, we present a deeper analysis that in particular removes this assumption.
The theory gives rise to an effective variational problem for the positions of
the walls, encapsulated in a $\Gamma$-convergence result.
In the second part of the paper, we turn our attention to another, more physical model,
including an anisotropy term.
We show that it permits a similar theory, but the anisotropy changes the renormalised energy
in unexpected ways and requires different methods to find it.

\medskip

\end{abstract}

\tableofcontents

\section{Introduction}

\subsection{Background}

According to the theory of micromagnetics, the magnetisation in ferromagnetic materials is given by a vector field of constant
length that obeys a variational principle. The underlying energy functional contains
several competing terms of rather different natures. In many
contexts, it gives rise to regions with approximately constant or slowly varying
magnetisation, separated by transition layers where the magnetisation varies rapidly. N\'eel walls are transition
layers of this type found in soft thin films and are one of the most important features there.

We study a model for N\'eel walls derived from the micromagnetic energy functional
(discussed in detail, e.g., by Hubert-Sch\"afer \cite{Hubert-Schaefer:98} or by De Simone-Kohn-M\"uller-Otto \cite{DKMO05}), but including some simplifications. Most significantly, our model
depends on the assumption that all the N\'eel
walls are parallel to each other, and they are one-dimensional transition layers where the magnetisation varies in-plane. 
As a result, the magnetisation is described by a map 
$$m \colon J \to \mathbb{S}^1,$$ where
$J$ is an interval. We consider two different, but closely related,
problems in this paper. The crucial difference between them is that $J= (-1, 1)$
in one case and $J = \R$ in the other case. For reasons that are explained
below, the different domains will necessitate a somewhat different set-up of the problem,
and therefore we explain the two problems separately.

\subsection{The confined problem} \label{subsect:confined}

We begin with the case where $J = (-1, 1)$, i.e., the whole magnetisation
profile is confined to the fixed interval $(-1, 1)$. Then $m \colon (-1, 1) \to \mathbb{S}^1$
represents the in-plane magnetisation vector field, renormalised to unit length.
As we have a boundary, our
problem requires boundary conditions. For this purpose, we fix the angle $$\alpha \in (0, \pi)$$
and require that
\[
m_1(-1) = m_1(1) = \cos \alpha.
\]
This will fix $m_2(\pm 1)$ up to the sign as well, but we allow either sign.
It is convenient to extend $m_1$ by $\cos \alpha$ outside of $(-1, 1)$.
The energy functional depends on a parameter $\epsilon > 0$
and is of the form
\begin{equation} \label{eqn:energy_confined}
E_\epsilon(m) = \frac{\epsilon}{2} \int_{-\infty}^\infty |m'|^2 \, dx_1 + \frac{1}{2} \|m_1 - \cos \alpha\|_{\dot{H}^{1/2}(\R)}^2,
\end{equation}
where $\dot{H}^{1/2}(\R)$ is the homogeneous Sobolev space of fractional order $1/2$ and we write
$x_1$ for the independent variable for reasons that will become clear later.
For details about the background of this model, and
how it is derived from the full micromagnetic energy, we refer to previous work
\cite{DeSimone-Kohn-Mueller-Otto:00,DeSimone-Kohn-Mueller-Otto:03,Melcher:03,Melcher:04,DeSimone-Knuepfer-Otto:06,Ignat-Otto:08,Ignat-Knuepfer:10,Chermisi-Muratov:13,Lund-Muratov:16,Muratov-Yan:16,Ignat-Moser:16,Ignat-Moser:17,Ignat-Moser:18}.
We note, however, that the first term on the right hand side in \eqref{eqn:energy_confined} is called exchange energy and the
second term is called stray field energy or magnetostatic energy in the theory of micromagnetics.

We emphasise that the second term in $E_\epsilon$ is of nonlocal nature. This makes the analysis of the problem challenging,
but the nonlocal nature is also responsible for the structure of the transition layers.
When $\epsilon$ is small,
the energy favours magnetisations such that $m_1 - \cos \alpha$ is small in the $\dot{H}^{1/2}$-sense,
but for topological reasons, it may not be able to vanish identically.
Then we expect transitions between the points $(\cos \alpha, \pm \sin \alpha)$ on $\mathbb{S}^1$,
and these are the N\'eel walls that we study.

Since any continuous transition between $(\cos \alpha, \pm \sin \alpha)$ requires that one of the
points $(\pm 1, 0)$ on $\mathbb{S}^1$ be attained, we may use the preimages of these points under $m$
(i.e., the zeros of $m_2$) as a proxy for the locations of the N\'eel walls.
There are two different types of N\'eel walls, depending on the sign of $m_1$ at such a point.
We are interested mainly in configurations with several N\'eel walls, located at
certain points $a_1, \dotsc, a_N$ with $a_1 < \dotsb < a_N$.
Thus given $N \in \N \cup \{0\}$, 
we define\footnote{We include $N = 0$ here to avoid special treatment of this case later on.
The definition should then be interpreted as follows: as $(-1, 1)^N$ is the set of maps $\{1, \dotsc, N\} \to (-1, 1)$,
for $N = 0$, we have the unique map $\emptyset \to (-1, 1)$. The additional constraints become vacuous.
So $A_0$ is \emph{not} empty.}
\[
A_N = \set{a \in (-1, 1)^N}{-1 < a_1 < \dotsb < a_N < 1}.
\]
An element $a = (a_1, \dotsc, a_N) \in A_N$ will typically be accompanied by $d \in \{\pm 1\}^N$,
which indicates the type of the N\'eel walls. That is, the pair $(a, d)$ gives rise to
the condition $m_1(a_n) = d_n$ for $n = 1, \dotsc, N$. We define
\[
\begin{split}
M(a, d) = \{m \in H^1((-1, 1); \mathbb{S}^1) \colon & m_1(a_n) = d_n \text{ for } n = 1, \dotsc, N \text{ and } 
 m_1(-1) = m_1(1) = \cos \alpha\}.
\end{split}
\]

A natural question is how much energy it takes to form N\'eel walls of given types
at given positions. In terms of the above notation, we may pose the question as follows:
what is
\be
\label{min33}
\inf_{M(a, d)} E_\epsilon
\ee
for given $a \in A_N$ and $d \in \{\pm 1\}^N$?
The question has previously been studied by DeSimone-Kohn-M\"uller-Otto \cite{DeSimone-Kohn-Mueller-Otto:03}, whose results show that
if we set
\begin{equation} \label{eqn:gamma_n}
\gamma_n = d_n - \cos \alpha
\end{equation}
for $n = 1, \ldots, N$, then
\[
\inf_{M(a, d)} E_\epsilon = \frac{\pi \sum_{n = 1}^N \gamma_n^2}{2 \log \frac{1}{\epsilon}} + O\left(\frac{\log \log \frac{1}{\epsilon}}{(\log \epsilon)^2}\right)
\]
as $\epsilon \searrow 0$. The result was improved in our previous paper \cite{Ignat-Moser:16}.
The first observation is that the asymptotic behaviour of this quantity is more easily
expressed in
\[
\delta = \epsilon \log \frac{1}{\epsilon}
\]
than in $\epsilon$ itself because $\delta$ is the appropriate size of the core of a N\'eel wall. (A N\'eel wall is a two-scale object, containing a core and two logarithmically decaying tails 
\cite{Melcher:03,Melcher:04,Ignat-Moser:16}.) For this reason, we will assume this relationship between $\epsilon$ and
$\delta$ throughout the paper and work mostly with $\delta$ henceforth. Another observation
is that the problem behaves similarly to Ginzburg-Landau vortices in some respects.
In particular, similarly to the theory of Ginzburg-Landau vortices (see the seminal book of Bethuel-Brezis-H\'elein \cite{Bethuel-Brezis-Helein:94}), we can
compute a renormalised energy that gives a lot of information about
the energy required for a certain collection of N\'eel walls \cite{Ignat-Moser:16}.
It is given by the function\footnote{For $N=0$,  equation \eqref{eqn:renormalised_energy_confined} reads $W=0$, while for $N=1$, the definition of $W(a_1, d_1)$ contains only the first term.} 

\be \label{eqn:renormalised_energy_confined}
W(a, d) = - \frac{\pi}{2} \sum_{n = 1}^N \gamma_n^2 \log(2 - 2a_n^2) 
- \frac{\pi}{2} \sum_{n = 1}^N \sum_{k \not= n} \gamma_k \gamma_n \log \left(\frac{1 + \sqrt{1 - \varrho(a_k, a_n)^2}}{\varrho(a_k, a_n)}\right),
\ee
where
\be
\label{varrho}
\varrho(b, c) = \frac{|b - c|}{1 - bc}\in [0,1)
\ee
for $b, c \in (-1, 1)$.
Besides the above interaction energy $W(a,d)$, the asymptotic behaviour of $E_\epsilon$ as $\epsilon \searrow 0$ also generates an intrinsic renormalised energy $e(d_n)$, characterising the energy of the core of a N\'eel wall of sign $d_n\in \{\pm 1\}$ (see \cite[Definition 26]{Ignat-Moser:16}).

\begin{theorem}[Renormalised energy \cite{Ignat-Moser:16}] \label{thm:renormalised_confined}
Let $N\geq 0$.  Then there exists a function $e : \{\pm 1\} \to \R$ such that for any
$a \in A_N$ and $d \in \{\pm 1\}^N$,
\[
\inf_{M(a, d)} E_\epsilon = \frac{\pi \sum_{n = 1}^N \gamma_n^2}{2 \log \frac{1}{\delta}} + \frac{W(a, d)+\sum_{n = 1}^N e(d_n)}{\left(\log \frac{1}{\delta}\right)^2} + o\left(\frac{1}{\left(\log \frac{1}{\delta}\right)^2}\right)
\, \textrm{ as } \,  \epsilon \searrow 0.
\]
\end{theorem}

While this theorem is formulated for fixed N\'eel walls, the quantity $\inf_{M(a, d)} E_\epsilon$
is continuous with respect to the positions of the walls encoded in $a$. It is not difficult to see
(and is proved nevertheless in Section \ref{sect:main_results_confined} below) that the following is true.

\begin{proposition} \label{prop:continuous_dependence_on_a}
Let $a \in A_N$ and $d \in \{\pm 1\}^N$. Then for any $c_0 > 0$ there exists $r > 0$ such that for any
$b \in A_N$, the following holds true. If $|b_n - a_n| \le r$ for $n = 1, \dotsc, N$, then for any $\epsilon \in (0, \frac{1}{4}]$,
\[
\Bigl|\inf_{M(a, d)} E_\epsilon - \inf_{M(b, d)} E_\epsilon\Bigr| \le \frac{c_0}{(\log \delta)^2}.
\]
\end{proposition}

Examining the
function $W$ can now give some information on how the energy depends on the positions
of these walls, and therefore on the interaction between them.
In particular, for every pair $a_k, a_\ell$ of N\'eel walls, we have an interaction
term involving $\varrho(a_k, a_\ell)$, and here $\varrho$ may be regarded as a distance function.
The interaction term in \eqref{eqn:renormalised_energy_confined} may be increasing or decreasing in $\varrho(a_k, a_\ell)$, depending on the sign of $\gamma_k \gamma_\ell$, and
accordingly, we can interpret the interaction as repulsive or attractive.
Examining the expression more closely,\footnote{The function $t \mapsto \log ((1 + \sqrt{1 - t^2})/t)$ is decreasing in $[0, 1)$ and behaves like $\log (1/t)$ as $t \searrow 0$. Therefore, the leading order term in  \eqref{eqn:renormalised_energy_confined} is
 $ \frac{\pi}{2} \sum_{n = 1}^N \sum_{k \not= n} \gamma_k \gamma_n \log |a_k-a_n|$.} we see that we have attraction between
N\'eel walls of the same sign and repulsion for different signs.

But such considerations are supported by the existing theory only
as long as a lower bound for their distances is assumed, because the estimates from
our paper \cite{Ignat-Moser:16} give no
control of the remainder term when this assumption is removed. This is one of the
gaps that we fill in this paper, provided that the transition angle $\alpha$ is such that $|\frac\pi 2-\alpha|$ is sufficiently
small. In order to quantify this statement, we will often assume in the following that $\alpha\in (\theta_N, \pi-\theta_N)$,
where $\theta_N\in [0,\frac \pi 2)$ is defined  as follows: $\theta_0=\theta_1=\theta_2=0$, and for $N\geq 3$,
\label{def:theta_N}
\[
\theta_N = \begin{cases}
\arccos\left(\frac{\sqrt{N + 1} - 1}{N}\right) & \text{if $N$ is odd}, \\
\arccos\left(\frac{\sqrt{N} - 1}{N - 1}\right) & \text{if $N$ is even}.
\end{cases}
\]
Obviously, $\theta_N$ is nondecreasing in $N$ and $\theta_N\to \frac\pi 2$ as $N\to \infty$.

Roughly speaking, our first new result states that families of profiles with energy
of order $O(1/|\log \delta|)$ subconverge to a limiting profile.
Moreover, if the energy is even bounded by
\[
\frac{\pi}{2|\log \delta|} \sum_{n = 1}^N \gamma_n^2 + O\left(\frac{1}{(\log \delta)^2}\right)
\]
(for the numbers $\gamma_n$ as in \eqref{eqn:gamma_n}, determined by the limiting profile), then all the N\'eel walls
stay separated from one another, provided that $\theta_N < \alpha < \pi - \theta_N$.\footnote{This condition means that the neighbouring transitions have comparable sizes. }
But in order to formulate this rigorously,
we need to introduce some additional notation first.

In the following result, we represent a continuous map $m \colon (-1, 1) \to \mathbb{S}^1$ in terms of its lifting
$\phi \colon (-1, 1) \to \R$, determined up to multiples of $2\pi$ by the condition
$m = (\cos \phi, \sin \phi)$. We abuse notation and write $E_\epsilon(\phi)$ instead of
$E_\epsilon(m)$ sometimes. In the limit $\epsilon \searrow 0$, the magnetisation is expected to satisfy $m_1 = \cos \alpha$
almost everywhere. In terms of the lifting, this means that $\phi$ takes values in $2\pi \Z \pm \alpha$
almost everywhere. We also expect that the limiting lifting will be in $\BV(-1, 1)$. We therefore define the set of piecewise constant functions
\[
\Phi = \set{\phi \in \BV(-1, 1)}{\phi(x_1) \in 2\pi \Z \pm \alpha \text{ for almost all } x_1 \in (-1, 1)}.
\]
We define the two functions $$\iota \colon \Phi \to \N \cup \{0\} \quad \textrm{and} \quad \eta \colon \Phi \to [0, \infty)$$ as follows.
For $b \in \R$, let $\dirac_b$ denote the Dirac measure
in $\R$ concentrated at $b$. Observe that for any $\phi \in \Phi$, there exists a unique
representation of the distributional derivative $\phi'$ of the form \label{nota_unconfined}
\[
\phi' = \sum_{\ell = 1}^L \sigma_\ell \, \dirac_{b_\ell},
\]
where $b_1, \dotsc, b_L \in (-1, 1)$ with $b_1 \le \dotsb \le b_L$ and
$\sigma_1, \dotsc, \sigma_L \in \{\pm 2\alpha, \pm 2(\pi - \alpha)\}$, and where
$|\sigma_{\ell - 1} + \sigma_\ell| = 2\pi$ whenever $b_{\ell - 1} = b_\ell$ for $\ell = 1, \dotsc, L$.\footnote{This condition implies that $\sigma_\ell$ and $\sigma_{\ell+1}$ have the same sign; thus, the jumps of $\phi$ are decomposed in a (strictly) monotone way
and the representation of $\phi'$ is unique.}
Given this representation, we define 
\[
\iota(\phi) = L \quad \text{and} \quad \eta(\phi) = \frac{\pi}{2} \sum_{\ell = 1}^L \left(1 - \cos \frac{\sigma_\ell}{2}\right)^2.
\]
Thus, $\iota(\phi)$ represents the number of what we think of as \emph{individual} N\'eel walls encoded by $\phi$
(each corresponding to a simple transition of angle $\pm \alpha$ or $\pm 2(\pi - \alpha)$ between the points
$(\cos \alpha, \pm \sin \alpha)$ on $\mathbb{S}^1$). The number 
$\eta(\phi)$ corresponds to the energy required for these N\'eel walls 
to leading order, as identified in Theorem \ref{thm:renormalised_confined}.

Our theory gives special relevance to functions $\phi \in \Phi$ such that every jump has size $\pm 2\alpha$ or $\pm 2(\pi - \alpha)$,
and thus every jump corresponds to exactly \emph{one} transition between the points $(\cos \alpha, \pm \sin \alpha)$ along $\mathbb{S}^1$.
This is the case exactly when $\iota(\phi)$ coincides with the number of jumps and is equivalent to the following condition.

\begin{definition} \label{def:simple}
We say that $\phi \in \Phi$ is \emph{simple} if there exist $N \in \N \cup \{0\}$, $a \in A_N$,
and $\omega \in \{\pm 2\alpha, \pm 2(\pi - \alpha)\}^N$ such that\footnote{If $N=0$, then \eqref{phi_0} reads $\phi'=0$, i.e., constant functions are simple.}
\begin{equation} \label{phi_0}
\phi' = \sum_{n = 1}^N \omega_n \dirac_{a_n}.
\end{equation}
If so, define $d \in \{\pm 1\}^N$ by
\begin{equation} \label{degres}
d_n = \begin{cases} 1 & \text{if } |\omega_n| = 2\alpha \\ -1 & \text{if } |\omega_n| = 2\pi - 2\alpha, \end{cases}
\end{equation}
provided that $\alpha \neq \pi/2$, for $n=1, \dotsc, N$. For $\alpha = \pi/2$, the jump sizes $\omega_n$ are not
sufficient to determine $d$. Instead, we then define $d_n = 1$ if either $\lim_{x_1 \nearrow a_n} \phi(x_1) \in 2\pi\Z - \pi/2$ and
$\omega_n = \pi$ or $\lim_{x_1 \nearrow a_n} \phi(x_1) \in 2\pi\Z + \pi/2$ and
$\omega_n = -\pi$, and we define $d_n = -1$ otherwise.
Then $(a, d)$ is called the \emph{transition profile} of $\phi$. 
For simple functions $\phi$, we have $N=\iota(\phi)$.
\end{definition}

We note that for any $a\in A_N$ and $d\in \{\pm 1\}^N$, the pair $(a, d)$ is the
transition profile of a simple function $\phi \in \Phi$. Moreover, this function
is unique up to a constant in $2\pi \Z$.\footnote{If $m= (\cos \phi, \sin \phi)$, then one can easily find an approximation sequence $m_\eps\in M(a,d)$ of $m$ (i.e., $m_\eps\to m$ in $L^2(-1,1)$).} If $\phi \in \Phi$ is simple with transition profile $(a, d)$, then we may define
$\gamma_n = d_n - \cos \alpha$ for $n = 1, \dotsc, N$, and we observe that
\[
\eta(\phi)=\frac\pi 2\sum_{n = 1}^N \gamma_n^2.
\]
This corresponds to the leading order term in the expansion of Theorem \ref{thm:renormalised_confined}.
In other words, we see here, as already mentioned, that $\eta$ represents the energy of the N\'eel walls
encoded in $\phi$. This remains true if $\phi$ is not simple.

We now have a compactness result for profiles with suitably bounded energy.
Under additional assumptions, we can rule out that several N\'eel walls collide as $\epsilon \searrow 0$, and
thus the limiting profile stays simple.

\begin{theorem}[Compactness and Separation] \label{thm:compactness_and_separation_confined}
Suppose that $(\phi_\epsilon)_{\epsilon > 0}$ is a family of functions in $H^1(-1, 1)$ such that
$\phi_\epsilon(-1) = \pm \alpha$ and $\phi_\epsilon(1) \in 2\pi \Z \pm \alpha$ for all $\epsilon > 0$.
Suppose that
\[
\limsup_{\epsilon \searrow 0} |\log \epsilon| E_\epsilon(\phi_\epsilon) < \infty.
\]
\begin{enumerate}
\item \label{item:compactness}
Then there exist a sequence of numbers $\epsilon_k \searrow 0$
and a function $\phi_0 \in \Phi$ such that $\phi_{\epsilon_k} \to \phi_0$ in $L^2(-1, 1)$.
\item
For the limit $\phi_0$ from statement \ref{item:compactness}, let
$N = \iota(\phi_0)$. If $\delta=\eps \log \frac 1 \eps$, $\theta_N < \alpha < \pi - \theta_N$, and
\begin{equation} \label{eqn:just_enough_energy}
E_\epsilon(\phi_\epsilon) \le \frac{\eta(\phi_0)}{\log \frac{1}{\delta}} + O\left(\frac{1}{\left(\log \frac{1}{\delta}\right)^2}\right), \quad \textrm{as } \, \eps\searrow 0,
\end{equation}
then $\phi_0$ is simple.
\end{enumerate}
\end{theorem}

In other words, under condition \eqref{eqn:just_enough_energy}, the jump points $a_1<a_2<\dots<a_N$ of $\phi_0$
are distinct and the size of each jump lies in $(-2\pi, 2\pi)$. Without condition \eqref{eqn:just_enough_energy}, the limit $\phi_0$ may have jumps
of any size in $2\pi \Z + \{0, \pm 2\alpha\}$. In general, this will of course amount
to several transitions between the points $(\cos \alpha, \pm \sin \alpha)$, so we
think of this as an accumulation of several N\'eel walls, or a composite N\'eel wall.
According to the last statement of the theorem, however, a composite N\'eel wall
requires more energy than the sum of its individual parts if $\theta_N < \alpha < \pi- \theta_N$. In contrast, if \eqref{eqn:just_enough_energy}
is satisfied, then the walls stay separated and during the transition, $\phi_\eps$ passes the set $\pi \Z$ only \emph{once}.

For simple limit configurations, the results from our previous
paper \cite{Ignat-Moser:16} then apply, and we can prove the following $\Gamma$-convergence result, which improves Theorem \ref{thm:renormalised_confined} insofar as no assumption of the position of the N\'eel walls is required. For the $\Gamma$-convergence result, there is no need of the angle constraint $\theta_N<\alpha<\pi-\theta_N$ because we restrict ourselves to simple functions $\phi_0$.
This restriction corresponds to the assumption in Ginzburg-Landau models that only vortices of
degree $\pm 1$ are present, which is a common feature in $\Gamma$-convergence results in such theories (see e.g. \cite{Ali-Pon, Ign-Jer}).

\begin{corollary}[$\Gamma$-convergence]
\label{cor:gamma_confined}
Let $\alpha\in (0, \pi)$. 
\begin{enumerate}
\item (Lower bound) Let $(\phi_\epsilon)_{\epsilon > 0}$ be a family of functions in $H^1(-1, 1)$ such that
$\phi_\epsilon(-1) = \pm \alpha$, $\phi_\epsilon(1) \in 2\pi \Z \pm \alpha$ for all $\epsilon > 0$,
and such that $\phi_\eps\to \phi_0$ in $L^2(-1,1)$
for a simple $\phi_0 \in \Phi$. Suppose that $(a, d)$ is the transition profile of $\phi_0$ and $N = \iota(\phi_0)$. Then 
\[
E_\epsilon(\phi_\epsilon) \ge \frac{\eta(\phi_0) }{\log \frac{1}{\delta}} + \frac{\sum_{n = 1}^N e(d_n) + W(a, d)}{\left(\log \frac{1}{\delta}\right)^2} - o\left(\frac{1}{\left(\log \frac{1}{\delta}\right)^2}\right) \quad \textrm{as } \, \eps\searrow 0.
\]

\item (Upper bound) If $\phi_0\in \Phi$ is simple with transition profile $(a, d)$ and if $N = \iota(\phi_0)$, then
there exists a family $(\phi_\epsilon)_{\epsilon > 0}$ in $H^1(-1, 1)$ such that
$\phi_\epsilon(-1) = \pm \alpha$, $\phi_\epsilon(1) \in 2\pi \Z \pm \alpha$ for all $\epsilon > 0$,  $\phi_\eps\to \phi_0$ in $L^2(-1,1)$, and 
$$E_\epsilon(\phi_\epsilon)  \leq \frac{\eta(\phi_0) }{\log \frac{1}{\delta}} + \frac{\sum_{n = 1}^N e(d_n) + W(a, d)}{\left(\log \frac{1}{\delta}\right)^2} + o\left(\frac{1}{\left(\log \frac{1}{\delta}\right)^2}\right) \quad \textrm{as } \, \eps\searrow 0.
$$
\end{enumerate}
\end{corollary}

This result amounts to an asymptotic expansion by $\Gamma$-convergence of the energy $E_\epsilon$
at second order (see \cite{BT} for more details about $\Gamma$-expansions).
The first order terms tell us that $\eta$ corresponds to the quantised energy of the individual
N\'eel walls and gives rise to the $\Gamma$-convergence
$|\log \delta| E_\epsilon \stackrel{\Gamma}{\rightharpoonup} \eta$ over the space of admissible liftings $\phi$ with the topology of $L^2(-1, 1)$.
The second order terms include (apart from the contribution $e(d)$ depending only on the number and types of
the N\'eel walls) the function $W(a, d)$, which encapsulates the interaction between the N\'eel walls
and governs their optimal positions. When the energy is rescaled by $(\log \delta)^2$, the second order terms
become the dominant, asymptotically finite contributions in the expansion. We call them `renormalised energy',
which is standard terminology in Ginzburg-Landau type theories. Sometimes, however, we apply this name to $W(a, d)$ alone.

A related $\Gamma$-convergence result has been proved in \cite{Ignat:09}, but it is at the first order and
weaker than Corollary \ref{cor:gamma_confined} because
it makes a statement about $\mathbb{S}^1$-valued maps rather than their liftings.

N\'eel walls may arise as the result of topological constraints. If we have a continuous function
$\phi \colon [-1, 1] \to \R$ with fixed values on the boundary,
say $\phi(-1) = \alpha$ (or $\phi(-1) = -\alpha$) and $\phi(1) = 2\pi \ell + \alpha$ (or
$\phi(1) = 2\pi \ell - \alpha$) for a fixed number $\ell \in \Z$, then the magnetisation
$m = (\cos \phi, \sin \phi) \colon [-1, 1] \to \mathbb{S}^1$ will have at least $N$ N\'eel walls
(transitions between $(\cos \alpha, \pm \sin \alpha)$), where $N = 2|\ell|$ or $N = 2|\ell| \pm 1$,
depending on the signs in the boundary data. If the number of N\'eel walls is exactly $N$,
then their signs will alternate. In this situation, we speak of \emph{topological} N\'eel walls.
Thus if we define 
$$d_N^+ = (1, -1, 1, \dotsc, \pm 1) \in \{\pm 1\}^N\quad
\textrm{and} \quad  d_N^- = (-1, 1, -1, \dotsc, \pm 1) \in \{\pm 1\}^N,$$ then we expect a renormalised energy
given by $W(a, d_N^\pm)$ for some $a \in A_N$.
Topological N\'eel walls are expected when we minimise $E_\epsilon$ subject to the above boundary data
(which amounts to prescribing the topological degree). Indeed, it can be shown,
with arguments similar to a related model \cite[Lemma 3.1]{Ignat-Moser:17}, that there are
exactly as many N\'eel walls for minimisers as the topology requires  and thus their signs
necessarily alternate.
When passing to the limit, we expect that the positions of the N\'eel walls converge to a minimiser
of $W(\blank, d_N^\pm)$ if such a minimiser exists.

The following result tells us when this is the case. The findings are consistent with
the expectation that topological N\'eel walls stay separated for $\alpha$ close to $\pi/2$,
while they tend to collide in the limit $\epsilon \searrow 0$ if $|\cos \alpha|$ is too large.
We will shortly discuss another result which makes this more precise.

\begin{proposition} \label{pro:min_W_conf}
Given $N\geq 2$ and $d = d_N^+$ or $d_N^-$, let
\begin{equation} \label{eqn:range_of_alpha}
\Theta = \begin{cases}
(\theta_N, \pi-\theta_N) & \text{ if $N$ is even} ,\\
(\theta_{N - 2}, \pi-\theta_N) & \text{ if $N$ is odd and $d_1 = 1$,}\\
(\theta_N, \pi-\theta_{N - 2}) & \text{ if $N$ is odd and $d_1 = -1$.}
\end{cases}
\end{equation}
Then 
\begin{enumerate}
\item \label{item:minimiser_exists}
$W(\blank, d)$ admits a minimizer over $A_N$ if $\alpha \in \Theta$;

\item \label{item:inf_is_-infinity} $\inf_{a\in A_N} W(a, d)=-\infty$ if $N\geq 3$ and $\alpha \not\in \overline{\Theta}$.
\end{enumerate}
\end{proposition}

In the specific situation of topological N\'eel walls, the information provided by Theorem
\ref{thm:compactness_and_separation_confined} can be improved. We formulate the following
statement for $\phi(-1) = -\alpha$ and $\phi(1) = 2\pi \ell + \alpha$ with $\ell \in \N \cup \{0\}$
only (for simplicity), but the other cases allow similar conclusions, too.
Here we denote 
$$\textrm{$e_+ := e(1)$ $\quad$ and $\quad$ $e_- := e(-1)$}$$ 
for the function $e \colon \{\pm 1\} \to \R$
from Theorem \ref{thm:renormalised_confined}. 

\begin{corollary}[Topological N\'eel walls] \label{cor:prescribed_winding_number_confined}
Let $\ell \in \N \cup \{0\}$ and $N = 2\ell + 1$. Suppose that $\theta_N < \alpha < \pi - \theta_N$.
Suppose that $(\phi_\epsilon)_{\epsilon > 0}$ is a family of functions in $H^1(-1, 1)$
such that $\phi_\epsilon(-1) = -\alpha$ and $\phi_\epsilon(1) = 2\pi \ell + \alpha$ for all $\epsilon > 0$.
Let\footnote{The quantity $\mathcal{E}_0$ corresponds to the leading order term in the energy expansion of $\phi_\eps$.
For example, a pair of N\'eel walls of signs $(+1, -1)$ generates to leading order the energy 
$$\frac\pi{2|\log \delta|} \big[(1-\cos \alpha)^2+(1+\cos \alpha)^2\big]=\frac\pi{|\log \delta|} (1+\cos^2\alpha).$$}
\[
\mathcal{E}_0 := \pi \ell (1 + \cos^2 \alpha) + \frac{\pi}{2}(1 - \cos \alpha)^2 \quad \text{and} \quad \mathcal{E}_1 := (\ell + 1) e_+ + \ell e_-.
\]
\begin{enumerate}
\item \label{item:liminf_topological} Then
\[
E_\epsilon(\phi_\epsilon) \ge \frac{\mathcal{E}_0}{\log \frac{1}{\delta}} + \frac{\mathcal{E}_1 + \inf_{\tilde{a} \in A_N} W(\tilde{a}, d_N^+)}{\left(\log \frac{1}{\delta}\right)^2} - o\left(\frac{1}{\left(\log \frac{1}{\delta}\right)^2}\right)
\, \textrm{ as } \epsilon \searrow 0.\]
\item \label{item:limsup_topological} If
\[
E_\epsilon(\phi_\epsilon) \le \frac{\mathcal{E}_0}{\log \frac{1}{\delta}} + O\left(\frac{1}{\left(\log \frac{1}{\delta}\right)^2}\right) \, \textrm{ as } \epsilon \searrow 0,
\]
then there exist $a \in A_N$, a sequence $\epsilon_k \searrow 0$,
and a simple $\phi_0 \in \Phi$ with $\phi_0(-1)=-\alpha$ and $\phi_0(1)=2\pi\ell+\alpha$
such that $\phi_{\epsilon_k} \to \phi_0$ in $L^2(-1, 1)$ and
\[
\phi_0' = \sum_{n = 1}^N \omega_n \dirac_{a_n},
\]
where $\omega_n = 2\alpha$ for $n$ odd and $\omega_n = 2\pi - 2\alpha$ for $n$ even.
\item \label{item:minimisers_topological}
If every $\phi_\epsilon$ is a minimiser of $E_\epsilon$ for its boundary data, then
all of the above holds and
\[
W(a, d_N^+) = \inf_{\tilde a \in A_N} W(\tilde a, d_N^+).
\]
\end{enumerate}
\end{corollary}

Combining Corollary \ref{cor:gamma_confined} and statements \ref{item:liminf_topological} and \ref{item:limsup_topological} of
Corollary \ref{cor:prescribed_winding_number_confined}, we deduce that
$$ \min_{{\mathcal T}_N } |\log \delta| (|\log \delta| E_\eps-\mathcal{E}_0)\rightarrow \mathcal{E}_1+\min_{A_N}W(\cdot,  d_N^+) \quad \textrm{ as } \quad \eps \to 0,$$
where ${\mathcal T}_N$ is the set of topological N\'eel walls with signs $d_N^+$. 
Statement \ref{item:minimisers_topological} in Corollary \ref{cor:prescribed_winding_number_confined}
determines (asymptotically as $\epsilon \searrow 0$) the optimal positions of N\'eel walls by minimising the renormalised energy $W(\cdot, d_N^+)$ when a winding number is prescribed by the boundary data.
For any result of this sort, we need to exclude neighbouring walls of the same sign
(which is guaranteed by the topology in Corollary \ref{cor:prescribed_winding_number_confined}), because
they will attract each other. This is a feature of the problem and not
a shortcoming of the theory. The question of attraction and repulsion
of N\'eel walls is discussed from a physical point of view by
Hubert and Sch\"afer \cite[Section 3.6.\ (C)]{Hubert-Schaefer:98}.

We also need to exclude the case of small transition angles $\alpha$. This is because
otherwise, the attraction of next-to-neighbouring walls may dominate
the (comparatively small) repulsion of neighbouring walls. This is
a phenomenon that was exploited in other papers of the authors
\cite{Ignat-Moser:17, Ignat-Moser:18} (in a somewhat different model with fixed $\epsilon$) to construct
energy minimisers comprising several N\'eel walls. For the problem studied here,
a plausible consequence (not proved here) is that for small transition angles,
several N\'eel walls may in the limit $\epsilon \searrow 0$ collapse to a
single \emph{composite} N\'eel wall, corresponding to a jump of more than $2\pi$
in the lifting (i.e., if $\alpha\in (0,\pi)$ is small, under the condition \eqref{eqn:just_enough_energy}, we may have
a jump of size $|\omega_n|\geq 2\pi$).

\subsection{The unconfined problem} \label{subsect:unconfined}

If we consider $m \colon \R \to \mathbb{S}^1$, rather than restricting the domain to a bounded 
interval, then a functional as in the preceding section will not have any nontrivial
critical points; indeed, as simple dilation of the form $\tilde{m}(x_1) = m(\lambda x_1)$
for $\lambda \in (0, 1)$ will decrease the energy unless $m$ is constant.
Any nontrivial structure, including a N\'eel wall, is therefore inherently unstable.
The situation changes, however, if the energy functional contains an additional term
modelling anisotropy.

For $\eps>0$ and $\alpha \in (0, \pi)$, we now study the functionals
\[
E_\epsilon(m) = \frac{1}{2} \int_{-\infty}^\infty \left(\epsilon |m'|^2 + (m_1 - \cos \alpha)^2\right) \, dx + \frac{1}{2} \|m_1 - \cos \alpha\|_{\dot{H}^{1/2}(\R)}^2
\]
for $m \in H_\loc^1(\R; \mathbb{S}^1)$. The additional term arises from the combination of a crystalline anisotropy
with easy axis parallel to the $m_2$-axis and an external magnetic field parallel to the $m_1$-axis.
For simplicity, we call it the anisotropy energy.

We redefine several quantities in this section, including $E_\epsilon$,
but as we will treat the two problems separately, this will not cause any problems.
Does this situation still allow results similar to our previous paper \cite{Ignat-Moser:16}
and the extension from the preceding section?
The answer is yes, and we will give some of the corresponding results here. Many of the arguments
are in fact similar to the confined case. What may seem somewhat surprising, however, is the
form of the renormalised energy for the new problem. There is now an interaction between
all three energy terms (exchange, stray field, and anisotropy energy), and while the
anisotropy energy contributes directly to the renormalised energy, it also does so indirectly
by modifying the contribution from the stray field energy. 

For $N \in \N$, we now define
\[
A_N = \set{a = (a_1, \dotsc, a_N) \in \R^N}{a_1 < \dotsb < a_N},
\]
whereas for $a \in A_N$ and $d \in \{\pm 1\}^N$, we define
\[
\begin{split}
M(a, d) = \{m \in H_\loc^1(\R; \mathbb{S}^1) \colon & E_1(m) < \infty \text{ and } 
 m_1(a_n) = d_n \text{ for } n = 1, \dotsc, N\}.
\end{split}
\]
Let
\[
I(t) = \int_0^\infty \frac{s e^{-s}}{s^2 + t^2} \, ds, \quad t>0.
\]
Also define
\begin{equation} \label{eqn:Euler-Mascheroni}
I_0 = \int_0^\infty e^{-s} \log s \, ds.
\end{equation}
This is minus the Euler-Mascheroni constant\footnote{Recalling the gamma function $\Gamma(t)=\int_0^\infty s^{t-1} e^{-s}\, ds$ for $t>0$,
we compute $I_0=\Gamma'(1)$, which is known to be $\lim_{n\to \infty} \big(\log n-\sum_{k=1}^n \frac1k\big)=-\gamma=-0.577\dotso$.} \cite{Lagarias:13}. (The latter is traditionally denoted by $\gamma$, but we use this symbol extensively
for something else.)
Then the renormalised energy turns out to be
\be
\label{eq:renorm_unconfined}
W(a, d) = -\frac{\pi}{2}I_0 \sum_{n = 1}^N \gamma_n^2-\frac{\pi}{2} \sum_{n = 1}^N \sum_{k \not= n} \gamma_k \gamma_n I(|a_k - a_n|).
\ee
Here again, given $a \in A_N$ and $d \in \{\pm 1\}^N$, we set $\gamma_n = d_n - \cos \alpha$.

It is worth pointing out that $I(t)$ is logarithmic to leading order again.
This is not immediately obvious here, but we will prove in Lemma \ref{lem:I_is_logarithmic} below that
$|I(t) + \log t - I_0| \le \pi t/2$ for all $t > 0$.
On the other hand, the function $I$ decays at the rate $1/t^2$ as $t \to \infty$. The extra term $-\frac{\pi}{2} I_0 \gamma_n^2$ for $n=1, \dots, N$ in \eqref{eq:renorm_unconfined} may be thought of as a `self-interaction'
contribution from each N\'eel wall. As $\lim_{t \searrow 0} (I(t) +\log t) = I_0$ (by Lemma \ref{lem:I_is_logarithmic}),
there is some analogy to the definition of $W$ in the confined case. Indeed,
the terms involving $\log(2 - 2a_n^2)$ in the renormalised energy in 
\eqref{eqn:renormalised_energy_confined}
may be interpreted the same way, although it is also identified as a `tail-boundary interaction' term in our previous paper
\cite{Ignat-Moser:16}. As in the confined model, next to the interaction energy $W(a,d)$
we also have the terms $e(d_n)$ for every N\'eel wall of sign $d_n\in \{\pm 1\}$.
In the following, we prove a result on the expansion of the minimal energy
in the unconfined case when the transition profile $(a, d)$ is fixed, similar to Theorem \ref{thm:renormalised_confined}
for the confined case.

\begin{theorem}[Renormalised energy] \label{thm:renormalised_unconfined}
Let $e \colon \{\pm 1\} \to \R$ be the function from Theorem \ref{thm:renormalised_confined}. For any
$a \in A_N$ and $d \in \{\pm 1\}^N$, the following expansion holds as $\epsilon \searrow 0$:
\[
\inf_{M(a, d)} E_\epsilon = \frac{\pi \sum_{n = 1}^N \gamma_n^2}{2 \log \frac{1}{\delta}} + \frac{1}{(\log \delta)^2} \left(W(a, d)+\sum_{n = 1}^N e(d_n)\right) + o\left(\frac{1}{(\log \delta)^2}\right). 
\]
\end{theorem}

Most of the new arguments in the proof of this result concern the effects of the anisotropy term on the limiting stray field potential.
When proving the corresponding result for the confined problem \cite{Ignat-Moser:16}, we
represented the stray field energy $\frac{1}{2} \|m_1 - \cos \alpha\|^2_{\dot{H}^{1/2}(\R)}$
in terms of the harmonic extension of $m_1 - \cos \alpha$ to the half-plane $\R_+^2 = \R \times (0, \infty)$
and made extensive use of the fact that a harmonic function $u \colon \R_+^2 \to \R$
remains harmonic under composition with a conformal map (especially a M\"obius transform).
In other words, we used the underlying symmetry of the stray field energy. The
anisotropy term does not share this symmetry, and therefore, this approach will not
work any more. Instead, we now use the Fourier transform in $x_1$ in order to turn a
harmonic function on $\R_+^2$ into a solution of an ordinary differential equation.
This will allow us to represent the solutions as oscillatory integrals.
Of course, some of the previous arguments can still be used here, and therefore,
parts of the proof of this result are given as a sketch only, emphasising the changes
relative to our previous work \cite{Ignat-Moser:16}.

The results from section \ref{sect:confined} (in particular, Theorem \ref{thm:compactness_and_separation_confined} and Corollaries \ref{cor:gamma_confined} and \ref{cor:prescribed_winding_number_confined})
will have almost identical counterparts for the new situation; see Theorem \ref{thm:unconfined_gamma_conv} below.
But this now requires nothing more than a combination
of the arguments from the confined case and from the proof of 
Theorem~\ref{thm:renormalised_unconfined}. For the convenience of the
reader, we give a very brief sketch of the proof in Section \ref{sec:last}.

Having found a new renormalised energy $W$, we may now wish to use it to determine
the (asymptotically) optimal positions of a set of N\'eel walls. As before, neighbouring
walls of the same sign will attract each other, and therefore, no minimisers of $W(\blank, d)$ in $A_N$
are to be expected unless $d_{n + 1} = - d_n$ for $n = 1, \dotsc, N - 1$. But in contrast
to the confined model, we now have to consider the possibility of arbitrarily large distances
between the N\'eel walls as well. We find that in some cases, no minimum exists even
though the infimum is finite. More precisely, we know the following.

\begin{proposition} \label{pro:min_W_unconf}
Let $N\geq 2$ and $d = d_N^+$ or $d_N^-$. Suppose that $\Theta$ is defined as in \eqref{eqn:range_of_alpha}. 
\begin{enumerate}
\item If $\alpha \in \Theta$, then $\inf_{A_N}W(\blank, d)>-\infty$. 

\item \label{item:inf_-infty} If $N\geq 3$ and $\alpha \not\in \overline{\Theta}$, then $\inf_{A_N} W(\blank, d)=-\infty$. 

\item \label{item:N=3}
Suppose that $N = 3$. If $d_1  \cos \alpha \in (-\frac{7}{9}, -\frac{1}{3})$,
then $W(\blank, d)$ has a unique critical point, which is not a minimiser. If $d_1 \cos \alpha \not\in (-\frac{7}{9}, -\frac{1}{3})$,
then $W(\blank, d)$ has no critical point.

\item \label{item:small_wall_outermost}
Suppose that $N \ge 2$ and $d_1 \cos \alpha\geq 0$ or $d_N \cos \alpha\geq 0$.
Then $W(\blank, d)$ has no critical point (and thus no minimisers). In particular, this is always the case if $N$ is even.

\end{enumerate}
\end{proposition}

The proposition gives only partial information for $N \ge 5$ odd. The question of existence/non\-ex\-is\-tence of minimisers
is open here, except for the cases covered in statements \ref{item:inf_-infty} and \ref{item:small_wall_outermost}.

Some of these observations are obviously consistent with our results \cite{Ignat-Moser:17, Ignat-Moser:18}
about the existence/nonexistence of minimisers of the functional $E_1$ for a prescribed winding number.
In particular, for winding numbers giving rise to the situation of statement \ref{item:small_wall_outermost},
we find that $E_1$ has no minimiser. Statement \ref{item:N=3}, however, shows an apparent contrast.
In this situation, the functional $E_1$ \emph{has} a minimiser, provided that $\alpha$ is sufficiently small
(and by symmetry of the model, this also applies when $\pi - \alpha$ is small). Although our previous papers only consider the
case $\epsilon = 1$, we expect that similar results apply to every fixed $\epsilon > 0$, but no
analysis has been carried out studying how the threshold for $\alpha$ or the
shape of the minimisers depend on $\epsilon$.

\subsection{Comparison between the confined and the unconfined model}

As we have seen, by and large, our two models permit the same sort of results. There are, however,
some subtle 
differences both on a technical level and in the consequences.

As mentioned previously, the confined and the unconfined model have different mechanisms to stabilise
the N\'eel walls. The same mechanisms also determine the transition angle $\alpha$.
In the confined model, this is the result of the steric interaction with the sample edges (represented by the boundary of the interval $J$).
In the unconfined model, we have a combination of the anisotropy effect and an external
magnetic field instead.

The two models lead to renormalised energies with a lot of similarities but also some qualitative differences.
At close range, the interaction between any two
N\'eel walls is essentially the same and logarithmic at leading order for both models. But
in the unconfined model, the walls can be arbitrarily far apart, and then the long-range interaction
decays quadratically to $0$. The confined model prevents this by the set-up, but the confinement is also
visible in the renormalised energy in the form of a boundary interaction term (also logarithmic).

As a consequence, the global energy landscapes of the two renormalised energies look quite different.
This is most easily demonstrated in the context of Corollary \ref{cor:prescribed_winding_number_confined},
where we consider $m = (\cos \phi, \sin \phi)$ with a prescribed winding number for $\phi$. In this
situation, the number and signs of the N\'eel walls are essentially given. In the confined model,
according to Proposition \ref{pro:min_W_conf}, the renormalised energy $W$ attains its minimum
among all admissible configurations for some values of $\alpha$ and is unbounded below otherwise.
(Roughly speaking, we have minimisers when $|\pi/2 - \alpha|$ is small.) For the unconfined model, the question is less clear, but
is discussed in Proposition \ref{pro:min_W_unconf}. While the result does not cover all possible cases, it
does give a complete overview for arrays of up to four N\'eel walls, and it turns out that $W$ \emph{never}
has a minimiser in all cases where we know the answer.

For both models, given $a \in A_N$ and $d \in \{\pm 1\}^N$, we identify a limiting stray field potential
$u_{a, d}^* \colon \R_+^2 \to \R$ in the proofs below, the understanding of which is crucial for all the results discussed in
this paper. Near every N\'eel wall, this function behaves like the phase of a vortex (which is why the
theory has some connections to Ginzburg-Landau vortices), and at leading order, this is the same for both models. The behaviour at $\infty$ is similar, too, but only in the unconfined model does it have immediate
consequences. This is what determines the decay of $W$
in the unconfined model as the distance between the N\'eel walls increases. Furthermore, the way that
the limiting stray field potential arises from $E_\epsilon$ is somewhat different for the two models.
The exchange energy has only an indirect effect on it, and thus in the confined model, where $E_\epsilon$
consists only of the exchange energy and the stray field energy, we mostly need to study the latter.
In the unconfined model, we also have the anisotropy energy, which has a fundamental effect on $u_{a, d}^*$.
Indeed, in many of the proofs in Section \ref{sect:unconfined}, we have to consider $u_{a, d}^*$ and
a limiting anisotropy energy jointly. This is somewhat surprising, as a variant of the Pohozaev
identity \cite[Proposition 1.1]{Ignat-Moser:17} implies that for critical points of $E_\epsilon$, the exchange energy and the anisotropy
energy are of the same magnitude, while the stray field energy is larger by a factor of order $|\log \epsilon|$.

\subsection{Comparison with other Ginzburg-Landau models}

The connection between our theory and Ginzburg-Landau vortices is not obvious, but can be
seen when studying the stray field energy in terms of a stray field potential. This is a function
on $\R_+^2 = \R \times (0, \infty)$ that is obtained from the magnetisation $m$ by means of a
boundary value problem (see \eqref{eqn:stray_field_potential} and \eqref{eqn:boundary_condition} below) and in the limit $\epsilon \searrow 0$ gives rise to the aforementioned function $u_{a, d}^*$. Its physical
relevance is that its gradient corresponds to the magnetic stray field induced by $m$. Near a N\'eel
wall, the stray field potential behaves like the phase of a vortex in $\R_+^2$. 
As the energy carried by the stray field potential turns out to be the dominant term in the limit,
the interaction between these `vortices' matters a lot.

In the context of Ginzburg-Landau type models, computing the renormalised energy between topological singularities has become a topic  extensively studied in last three decades, see, e.g.,
the seminal book of Bethuel-Brezis-H\'elein \cite{Bethuel-Brezis-Helein:94} and the further developments discussed by 
Sandier-Serfaty \cite{SS_book}. In two-dimensional Ginzburg-Landau  models, one of the key tools is the Jacobian of the order parameter, which
detects the point singularities in the shape of vortices.
Many techniques have been developed based on the Jacobian, yielding compactness results, lower bounds, and $\Gamma$-convergence (at the first and second order), see e.g. \cite{Ali-Pon, Col-Jer, Ign-Jer_cras, Ign-Jer, Jer_low, Jer_son, Jer_spi, Lin, San_low, SS_flow}.

However, in our models, there is no topological invariant 
playing the role of the Jacobian, even though we have a stray field potential that behaves like
a vortex angle around a domain wall.
In fact, our problem is phenomenologically different from 
the standard Ginzburg-Landau model. First, a N\'eel wall is a two length-scale topological object,
while the vortex has only one scale. Second, while the renormalised energy between
 Ginzburg-Landau vortices is generated solely
by the tail-tail interaction of neighbouring vortices, in our problem,
 there is a core-tail interaction as well. In fact, the latter even 
dominates the tail-tail interaction between N\'eel walls, which is a new feature  in the context of Ginzburg-Landau problems.
 It also gives rise to a new phenomenon: we have attraction for walls of the same sign and repulsion for walls of different signs, which is exactly the opposite in the context of Ginzburg-Landau vortices.

As a result of all of this, we can use the well-known ideas for Ginzburg-Landau
vortices as a motivation for much of our theory, but they are generally not sufficient and
can occasionally be misleading. The different attraction/repulsion pattern also gives rise
to an energy landscape in the renormalised energy that is quite different from Ginzburg-Landau vortices.
Some of this energy landscape is explored in Propositions \ref{pro:min_W_conf} and \ref{pro:min_W_unconf} above,
but more work is required here.

\subsection{Representations of the stray field energy}

A good understanding of the nonlocal term in the energy functional, the stray field
energy, is essential for the analysis of both problems. In both cases, it is given
in terms of the seminorm $\|\blank\|_{\dot{H}^{1/2}(\R)}$, and there are several
ways to characterise this quantity. Possibly the best-known representation involves
the Fourier transform $\hat{f}(\xi)=\int_{\R} e^{-i\xi x_1} f(x_1)\, dx_1$ for $\xi\in \R$. Then
\[
\|f\|_{\dot{H}^{1/2}(\R)}^2 =\frac1{2\pi} \int_{-\infty}^\infty |\xi| |\hat{f}|^2 \, d\xi.
\]
Our arguments mostly rely on different representations, however, one of which has
a physical interpretation as well. (This is the reason why $\dot{H}^{1/2}(\R)$
appears in the problem in the first place.)

Given $f \in \dot{H}^{1/2}(\R)$, consider the boundary value problem
\begin{alignat}{2}
\Delta u & = 0 & \quad & \text{in $\R_+^2$}, \label{eqn:u_harmonic} \\
\dd{u}{x_2} & = -f' && \text{on $\R \times \{0\}$}. \label{eqn:u_boundary}
\end{alignat}
Let $C_0^\infty(\overline{\R_+^2})$ denote the set of smooth functions in $\R^2_+$ with compact support in $\R\times [0, \infty)$. 
If $\dot{H}^1(\R_+^2)$ denotes the completion of $C_0^\infty(\overline{\R_+^2})$ with respect
to the inner product
\[
\scp{u}{v}_{\dot{H}^1(\R_+^2)} = \int_{\R_+^2} \nabla u \cdot \nabla v \, dx,
\]
then there exists a unique solution $u$ of \eqref{eqn:u_harmonic}, \eqref{eqn:u_boundary}
in $\dot{H}^1(\R_+^2)$. (If interpreted as a function, however, then $u$ is unique only
up to a constant.) 
If $f = m_1 - \cos \alpha$ for a given magnetisation $m$, then
$u$ should be regarded as a potential for the magnetic stray field induced by $m$. Then
\[
\|f\|_{\dot{H}^{1/2}(\R)}^2 = \int_{\R_+^2} |\nabla u|^2 \, dx.
\]

There exists a dual representation, which may be more familiar but has no physical interpretation.
As $\Delta u = 0$, the vector field $\nabla^\perp u = (-\dd{u}{x_2}, \dd{u}{x_1})$ is curl free in $\R_+^2$.
Hence there exists $v \colon \R_+^2 \to \R$ such that $\nabla v = \nabla^\perp u$.
Because $\dd{v}{x_1} = - \dd{u}{x_2} = f'$ on $\R \times \{0\}$, we may further choose
$v$ such that $v = f$ on $\R \times \{0\}$. We then compute $\Delta v = 0$ in $\R_+^2$, so
$v$ is a harmonic extension of $f$. Indeed it it the unique harmonic extension with
finite Dirichlet energy. Now we may also write
\[
\|f\|_{\dot{H}^{1/2}(\R)}^2 = \int_{\R_+^2} |\nabla v|^2 \, dx.
\]

Finally, we mention one more representation of the stray field energy that does
not play a significant role here, but has been used extensively for other questions
about the model \cite{Ignat-Moser:17,Ignat-Moser:18}. It can be shown that \cite[Chapter 7]{Lieb-Loss:01}
\[
\|f\|_{\dot{H}^{1/2}(\R)}^2 = \frac{1}{2\pi} \int_{-\infty}^\infty \int_{-\infty}^\infty \frac{(f(s) - f(t))^2}{(s - t)^2} \, ds \, dt.
\] 

When studying the functional $E_\epsilon$, we apply these formulas to the
function $m_1 - \cos \alpha$. For the problem described in Section \ref{subsect:unconfined},
this is a function defined on $\R$, but in Section \ref{subsect:confined}, it is initially
defined in the interval $(-1, 1)$. We therefore extend $m_1$ by $\cos \alpha$ everywhere
else. Then $\|m_1 - \cos \alpha\|_{\dot{H}^{1/2}(\R)}$ can be represented in terms
of the boundary value problem
\begin{alignat}{2}
\Delta u & = 0 & \quad & \text{in $\R_+^2$}, \label{eqn:stray_field_potential} \\
\dd{u}{x_2} & = -m_1' && \text{on $\R \times \{0\}$}. \label{eqn:boundary_condition}
\end{alignat}
for both problems. This problem has a unique solution $u \in \dot{H}^1(\R_+^2)$,
which is called the \emph{stray field potential} induced by the magnetisation $m$.

\section{Analysis for the confined problem} \label{sect:confined}

The purpose of this section is to prove Theorem \ref{thm:compactness_and_separation_confined},
Corollaries \ref{cor:gamma_confined} and \ref{cor:prescribed_winding_number_confined}, and Propositions
\ref{prop:continuous_dependence_on_a} and \ref{pro:min_W_conf}.
Thus we consider the energy
functional $E_\epsilon$ defined in \eqref{eqn:energy_confined} and the renormalised energy
defined in \eqref{eqn:renormalised_energy_confined}. Furthermore, both $A_N$ and $M(a, d)$ (for
$a \in A_N$ and $d \in \{\pm 1\}^N$) are defined as in Section \ref{subsect:confined}.

\subsection{Compactness}

We first address the compactness result, i.e., statement \ref{item:compactness}
in Theorem \ref{thm:compactness_and_separation_confined}.
For comparison, we mention that there are other compactness results for $\mathbb{S}^1$-valued magnetisations
in terms of  their liftings in various ferromagnetic thin-film regimes in \cite{RS01, K06, IK19}.

Even though we obtain accumulation points in $\Phi \subset \BV(-1, 1)$ here,
there is no compactness with respect to the $\BV$-topology \cite[Theorem~3]{Ignat-Otto:08}.
We use the $L^2$-topology for convenience, but obviously it can be replaced by any other
$L^p$-topology with $p < \infty$.

\begin{proposition} \label{prop:compactness_confined}
Consider a sequence $(\epsilon_k)_{k\in \N}$ of positive numbers with
$\epsilon_k \searrow 0$ as $k \to \infty$. Let $(\phi_k)_{k \in \N}$
be a sequence in $H^1(-1, 1)$ such that every $\phi_k$ satisfies the
boundary conditions $\phi_k(-1) = \pm \alpha$ and $\phi_k(1) \in 2\pi \Z \pm \alpha$.
Suppose that
\begin{equation} \label{eqn:bounded_energy_confined}
\limsup_{k \to \infty} |\log \epsilon_k| E_{\epsilon_k}(\phi_k) < \infty.
\end{equation}
Then $\set{\phi_k}{k \in \N}$ is uniformly bounded in $L^\infty(-1,1)$ and relatively compact in $L^2(-1, 1)$. Moreover,
if $\phi_0$ is any accumulation point,
then $\phi_0 \in \Phi$.
\end{proposition}

\begin{proof}
This statement is an improvement of a result of the first author \cite[Theorem 1]{Ignat:09}. In fact, only the compactness of the $\mathbb{S}^1$-valued maps $m_\eps=(\cos \phi_\eps, \sin \phi_\eps)$ is proved in \cite{Ignat:09}, while here we give a more precise statement in terms of the lifting $\phi_\eps$. 
Nevertheless, we can follow some of the same steps here, referring to the earlier paper for some of the details.

First recall that for $m \in H^1((-1, 1); \mathbb{S}^1)$, the $\dot{H}^{1/2}$-seminorm of $m_1 - \cos \alpha$
is controlled by the stray field energy. More precisely,
\[
\|m_1 - \cos \alpha\|_{\dot{H}^{1/2}(\R)}^2 \le 2E_\epsilon(m).
\]
Consider $m_k = (\cos \phi_k, \sin \phi_k)$.
Since $m_{k1}(-1) = \cos \alpha$, inequality \eqref{eqn:bounded_energy_confined}
is therefore enough to conclude that $m_{k1} \to \cos \alpha$
in $L^2(-1, 1)$.

Next we localise the large variations of $m_{k2}$. Since each $m_k$ belongs to $H^1((-1, 1); \mathbb{S}^1)$, $m_k$ is a uniformly continuous function (for fixed $k$).
Therefore, the set
\[
m_{k2}^{-1}\left((-\textstyle \frac{1}{2} \sin \alpha, \frac{1}{2} \sin \alpha)\right)
\]
is open and every connected component is an open interval. If $(a, b)$ is one of these
connected components, then $|m_{k2}(a)| = |m_{k2}(b)| = \frac{1}{2} \sin \alpha$.
If they have opposite signs, then
\[
b - a \ge \frac{\sin^2 \alpha}{\|m_k'\|^2_{L^2(-1, 1)}}
\]
by the fundamental theorem of calculus and the Cauchy-Schwarz inequality.
Thus, discarding all the other connected components, we get 
a finite number $N_k$ of intervals $(a^k_n, b^k_n)$, $1\leq n \leq  N_k$,
such that
\begin{enumerate}
\item $-1 < a_1^k < b_1^k \le a_2^k < b_2^k \le \dotsb \le a_{N_k}^k < b_{N_k}^k < +1$,
\item $m_{k2}(a_n^k) = \pm \frac{1}{2} \sin \alpha$ and $m_{k2}(b_n^k) = \mp \frac{1}{2} \sin \alpha$
(with opposite sign),
\item \label{item:behaviour_in_intervals}
$- \frac{1}{2} \sin \alpha < m_{k2} < \frac{1}{2} \sin \alpha$ in $(a_n^k, b_n^k)$
for $n = 1, \dotsc, N_k$, and
\item \label{item:behaviour_between_intervals}
$m_{k2} < \frac{1}{2} \sin \alpha$ or $m_{k2} > -\frac{1}{2} \sin \alpha$ in $(b_n^k, a_{n + 1}^k)$
for $n = 0, \dotsc, N_k$, where $b_0^k = -1$ and $a_{N_k + 1}^k = 1$
\end{enumerate}
(see Figure \ref{fig:compactness}).
\begin{figure}[htb!]
\center
\includegraphics[width=0.5\textwidth]{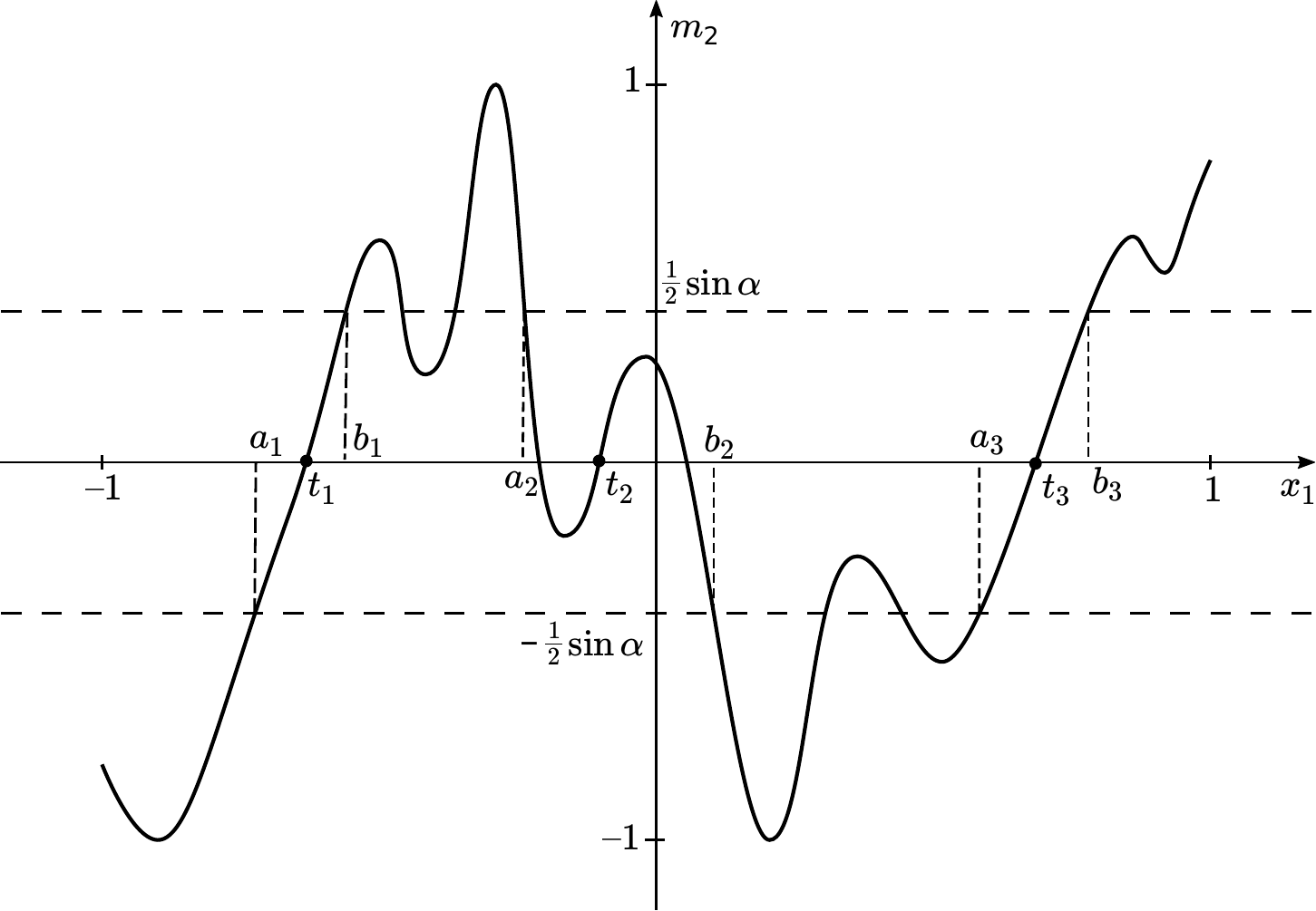} \caption{The variations of
$m_{2}$.} \label{fig:compactness}
\end{figure}
Then, using a duality argument and an interpolation inequality for a well-chosen cut-off function, 
we can prove that
\[
\sup_{k \in \N} N_k < \infty.
\]
The details for these arguments can be found in the aforementioned paper \cite{Ignat:09}.

In order to prove the relative compactness of $(\phi_k)_{k \in \N}$, we construct step functions $\psi_k \colon \R \to 2\pi \Z \pm \alpha$
approximating $\phi_k$. To this end, if $N_k > 0$, choose $t_n^k \in (a_n^k, b_n^k)$
such that $\phi_k(t_n^k) \in \pi\Z$
(so $m_{k2}(t_n^k) = 0$) for $n = 1, \dotsc, N_k$. Furthermore, set $t_0^k = -1$ and $t_{N_k + 1}^k = 1$.
By properties \ref{item:behaviour_in_intervals} and \ref{item:behaviour_between_intervals} 
of the intervals $(a_1^k, b_1^k), \dotsc, (a_{N_k}^k, b_{N_k}^k)$, the values of $m_k$ are
restricted to one of the arcs $\mathcal{C}^\pm = \set{m \in \mathbb{S}^1}{\pm m_2 \ge -\frac{1}{2} \sin \alpha}$ in $[t_{n - 1}^k, t_n^k]$
for $n = 1, \dotsc, N_k + 1$. Hence if we choose $\phi_n^k \in 2\pi \Z \pm \alpha$ such that
$(\cos \phi_n^k, \sin \phi_n^k) \in \mathcal{C}^\pm$ as well and $|\phi_n^k - \phi_k(t_n^k)| < \pi$, then
\begin{equation} \label{eqn:approximating_step_function1}
|\phi_k - \phi_n^k| < \pi
\end{equation}
throughout $[t_{n - 1}^k, t_n^k]$.
Let $\chi_n^k$ be the characteristic function of $(t_{n - 1}^k, t_n^k]$ and set
\[
\psi_k = \sum_{n = 1}^{N_k + 1} \phi_n^k \chi_n^k.
\]
If $N_k = 0$, then $m_k(x_1) \in \mathcal{C}^\pm$ for all $x_1 \in (-1, 1)$
for one of these arcs. It follows that $\phi_k(-1) = \phi_k(1)$
and we set $\psi_k \equiv \phi_k(1)$.
Since $\phi_k(-1) = \pm \alpha$, it follows that $|\phi_k| \le \pi (N_k + 2)$ in $(-1, 1)$.
Since $(N_k)_{k \in \N}$ is bounded, we deduce that $(\phi_k)$ is uniformly bounded in $L^\infty(-1,1)$. 
Moreover, inequality \eqref{eqn:approximating_step_function1} implies that there exists a
universal constant $C > 0$ with
\begin{equation} \label{eqn:approximating_step_function2}
(\phi_k - \psi_k)^2 \le C |e^{i\phi_k} - e^{i\psi_k}|^2
\end{equation}
in $(-1, 1)$.
As a consequence of properties \ref{item:behaviour_in_intervals} and \ref{item:behaviour_between_intervals},
we further have the inequality
\[
|\sin \phi_k + \sin \psi_k| = |m_{k2} + \sin \psi_k| \ge \frac{1}{2} \sin \alpha
\]
everywhere in $(-1, 1)$. Thus \eqref{eqn:approximating_step_function2} implies that
\begin{align*}
\int_{-1}^1 (\phi_k - \psi_k)^2 \, dx_1
& \le C \int_{-1}^1 (\cos \phi_k - \cos \psi_k)^2 \, dx_1 +  C \int_{-1}^1 (\sin \phi_k - \sin \psi_k)^2 \, dx_1 \\
& \le  C \int_{-1}^1 (\cos \phi_k - \cos \alpha)^2 \, dx_1 + \frac{4C}{\sin^2 \alpha} \int_{-1}^1 (\sin^2 \phi_k - \sin^2 \psi_k)^2 \, dx_1 \\
& \le  C \left(1 + \frac{16}{\sin^2 \alpha}\right) \int_{-1}^1 (\cos \phi_k - \cos \alpha)^2 \, dx_1 \to 0
\end{align*}
as $k \to \infty$. The last step is due to the convergence $m_{k1} \to \cos \alpha$
in $L^2(-1, 1)$ established earlier in the proof. It now follows that any
accumulation point of $\set{\phi_k}{k \in \N}$ will also be an accumulation point of
$\set{\psi_k}{k \in \N}$ and vice versa.

The uniform bound for $N_k$ implies that the sequence $(\psi_k)_{k \in \N}$ is uniformly bounded in $\BV(-1, 1)$
and in $L^\infty(-1, 1)$.
Therefore, it is relatively compact in $L^2(-1, 1)$ and
any accumulation point belongs to $\BV(-1, 1)$ and takes values in
$2\pi \Z \pm \alpha$ almost everywhere.
\end{proof}

\subsection{Blow-up of the renormalised energy}

In this section we examine the behaviour of the renormalised energy $W$
defined in \eqref{eqn:renormalised_energy_confined} when two
or more N\'eel walls approach each other or the boundary.
This question is more subtle than it may appear, because there are attractive and
repulsive terms in the definition of $W$.

We will work with the quantity 
\be
\label{def_rho}
\rho(a) = \frac{1}{2} \min\{2a_1 + 2, a_2 - a_1, \dotsc, a_N - a_{N - 1}, 2-2a_N \},
\ee
defined for $N \ge 1$ (although for the moment we assume that $N \ge 2$).

\begin{lemma} \label{lem:blow_up_of_logarithmic_NEW} 
Let $N\geq 2$, $B_1, \dots, B_N<0$, and $A_{k\ell} \in \R$ for $1 \le k < \ell \le N$.
Let $f \colon A_N \to \R$ be defined by the formula
\[
f(a) =\sum_{k=1}^N B_k \log (1-a_k^2)+\sum_{1\leq k < \ell\leq N } A_{k\ell} \log \varrho (a_\ell, a_k), \quad \textrm{for } \, a\in A_N,
\]
where $\varrho$ is given in \eqref{varrho}. Then the following holds true.
\begin{enumerate}
\item \label{item:divergence_to_infintiy}
If $\sum_{K \le k < \ell \le L} A_{k\ell} < 0$ for all $K, L \in \{1, \dotsc, N\}$ with $K < L$, then 
 $f(a) \to \infty$ as $\rho(a) \to 0$. As a consequence, $f$ is bounded below. 
  
\item \label{item:unbounded_below}
If there exist $K < L$ such that
$\sum_{K \le k < \ell \le L} A_{k \ell} > 0$,
then $f$ is unbounded below, i.e., there exists a sequence $(a^{(i)})_{i \in \N}$ in $A_N$ such that $\rho(a^{(i)})\to 0$ and $f(a^{(i)})\to -\infty$ as $i\to \infty$.
\end{enumerate}
\end{lemma}

We will apply this lemma for numbers of the form $A_{k\ell} = \gamma_k \gamma_\ell$, where
$\gamma_n$ is as in \eqref{eqn:gamma_n}. In order to verify the hypothesis in some cases,
we use the following result.

\begin{lemma} \label{lem:coefficients_for_alternating_signs}
Let $N\geq 2$ and $d = d_N^+$ or $d_N^-$, and let $\gamma_n=d_n-\cos \alpha$ for $n = 1, \dotsc, N$.
 Suppose that $\Theta \subset (0, \pi)$ is defined as in \eqref{eqn:range_of_alpha}.
\begin{enumerate}
\item \label{item:negative} If $\alpha \in \Theta$,
then  for any $K, L \in \{1, \dotsc, N\}$ with $K < L$:
\[
\sum_{K \le k < \ell \le L} \gamma_k \gamma_\ell < 0.
\]
\item \label{item:positive} If $N\geq 3$ and $\alpha \not\in \overline{\Theta}$,  
then there exist $K<L$ such that $\sum_{K \le k < \ell \le L}\gamma_k \gamma_\ell >0$.
\end{enumerate}
\end{lemma}

We prove the two statements in Lemma \ref{lem:blow_up_of_logarithmic_NEW} independently.

\begin{proof}[Proof of Lemma \ref{lem:blow_up_of_logarithmic_NEW}, statement \ref{item:divergence_to_infintiy}]
We argue by induction over $N$. If $N = 2$ and $A_{12} < 0$, then the three terms in $f(a)$
are positive because $a_1, a_2\in (-1,1)$ and $\varrho (a_1, a_2)\in (0,1)$.
 Assuming that $\rho(a)\to 0$, we conclude that either one of the points
$a_1$ or $a_2$ approaches the boundary, or $\varrho  (a_1, a_2) \to 0$. In both cases, the statement is obvious.

Now suppose that the statement is true for
all integers between $2$ and $N - 1$.
Consider a sequence $(a^{(i)})_{i \in \N}$ in $A_N$ with $\rho(a^{(i)}) \to 0$. We may assume without loss of generality that $a_n^{(i)} \to a_n$
for some $a_n \in [-1, 1]$ for every $n = 1, \ldots, N$.
(If not, we choose a subsequence with this property.)
By the assumption that $\rho(a^{(i)}) \to 0$, either one of the limit points $a_1$ or $a_N$ is
on the boundary  (i.e., $a_1 = -1$ or $a_N = 1$), or at least two of the limit points $a_1, \dotsc, a_N$ are equal.

We consider several cases.

\setcounter{caseno}{0}
\begin{subproof}{Case \case\label{case:outermost_walls_sepaparted2}}
If $a_1 < a_N$, choose a partition  $\Lambda_1, \dotsc, \Lambda_J$ of $\{1, \dotsc, N\}$ such that
$a_k = a_\ell$ whenever $k, \ell \in \Lambda_j$, but $a_k < a_\ell$ if $k \in \Lambda_{j_1}$ and $\ell \in \Lambda_{j_2}$ for $j_1 < j_2$. Note that $|\Lambda_j| < N$ for every $1\leq j\leq J$.

\setcounter{subcaseno}{0}
\begin{subproof}{Case \subcase\label{subcase:J=N}}
Suppose first that $J=N$ (i.e., $a_1<\dots<a_N$) and either $a_1=-1$ or $a_N=1$. In this case, the terms in $f$ involving $\log \varrho(a^{(i)}_\ell, a^{(i)}_k)$ are uniformly bounded, while at least one boundary interaction term in $f$ 
(involving $\log (1-(a^{(i)}_1)^2)$ or $\log (1-(a^{(i)}_N)^2)$) blows up to $+\infty$ as $i\to \infty$.
\end{subproof}

\begin{subproof}{Case \subcase\label{subcase:J<N}}
Now suppose that $J<N$. Then $|\Lambda_j| \ge 2$ for at least one value of $j$.
But for every $j$ such that $|\Lambda_j| \ge 2$, the induction assumption implies that
\[
 \sum_{k\in \Lambda_j} B_k \log (1-(a^{(i)}_k)^2)+\sum_{\substack{k < \ell \\ k, \ell \in \Lambda_j}} A_{k \ell} \log \varrho 
(a_\ell^{(i)}, a_k^{(i)}) \to \infty.
\]
All the other terms in the definition of $f$ remain uniformly bounded from below (some boundary interaction term in $f$ might blow up to $+\infty$), and thus, it follows that $f(a^{(i)}) \to \infty$ as $i \to \infty$.
\end{subproof}
\end{subproof}

\begin{subproof}{Case \case\label{case:everything_collapses2}}
If $a_1 = a_N$, then we distinguish two cases again.

\setcounter{subcaseno}{0}
\begin{subproof}{Case \subcase\label{subcase:away_from_boundary}}
Suppose that $-1<a_1=a_N<1$. We define $\sigma_i = a_N^{(i)} - a_1^{(i)}$ and note that $\sigma_i \to 0$ as $i \to \infty$.
We set $\tilde{a}^{(i)} = (\tilde{a}_1^{(i)}, \dotsc, \tilde{a}_N^{(i)})$, where
\[
\tilde{a}_n^{(i)} = \frac{a_n^{(i)} - a_1^{(i)}}{2\sigma_i}, \quad 1\leq n\leq N.
\]
Then $\tilde{a}^{(i)} \in A_N$ and $\tilde{a}_1^{(i)}=0<\frac12=\tilde{a}_N^{(i)}$, so $\tilde{a}^{(i)}$ is
consistent with Case \ref{case:outermost_walls_sepaparted2}.
Since all of the points in $a^{(i)}$ and $\tilde{a}^{(i)}$ stay in an interval of the form
$[c - 1, 1 - c]$ for some $c \in (0, 1)$, then all the boundary interaction terms in the formulas for $f({a}^{(i)})$ and $f(\tilde{a}^{(i)})$ are uniformly bounded.
Furthermore,
\[
\log(\tilde a_\ell^{(i)} - \tilde a_k^{(i)}) - \log 2 \le \log \varrho(\tilde a^{(i)}_\ell, \tilde a^{(i)}_k) \le \log(\tilde a_\ell^{(i)} - \tilde a_k^{(i)}) - \log(2c - c^2)
\]
for every $k< \ell$ and every $i \in \N$, and the same inequalities hold for $a$.
Therefore, there exists $C>0$ such that for every $i$,
\[
f({a}^{(i)}) \geq f(\tilde a^{(i)}) + \sum_{k < \ell} A_{k\ell} \log(2 \sigma_i)-C.
\] 
If ${\rho}(\tilde{a}^{(i)}) \to 0$, then the arguments of Case \ref{case:outermost_walls_sepaparted2}
show that $f(\tilde{a}^{(i)}) \to \infty$ as
$i \to \infty$. Otherwise, the values $f(\tilde{a}^{(i)})$ will remain bounded. In both cases, we know that
$\log(2 \sigma_i) \to - \infty$, while $\sum_{k < \ell} A_{k \ell} < 0$ by the above assumption.
It follows that $f(a^{(i)}) \to \infty$ as $i \to \infty$.
\end{subproof}

\begin{subproof}{Case \subcase\label{subcase:convergence_to_1}} Finally, assume that $a_1=a_N\in \{\pm 1\}$.
It suffices to consider the case where $a_1=a_N=1$ (as the other case is similar).
For $b \in (-1, 1)$, the M\"obius transform $\Phi_b \colon \R_+^2 \to \R_+^2$ is defined by
\be
\label{mobius}
\Phi_b(z) = \frac{z + b}{1 + bz}, \quad z \in \C.
\ee
It is readily checked that $\varrho$ is invariant under $\Phi_b$.

We set $\tilde a_k^{(i)}=\Phi_{-a_1^{(i)}}(a_k^{(i)})$ for $1\leq k\leq N$. So $\tilde a^{(i)}\in A_N$ and $\tilde a_1^{(i)}=0$.
If $i$ is sufficiently large, then also $1-(\tilde a_k^{(i)})^2\geq 1-(a_k^{(i)})^2$ for $1\leq k\leq N$,
because $a_k^{(i)}\to 1$ as $i\to \infty$. As $\varrho$ is invariant under the M\"obius transform, we know that
$\varrho(\tilde a^{(i)}_\ell, \tilde a^{(i)}_k)=\varrho(a^{(i)}_\ell, a^{(i)}_k)$. Therefore, we deduce that
$$f({a}^{(i)}) \geq f(\tilde a^{(i)}) +B_1\log (1-(a_1^{(i)})^2).$$
If ${\rho}(\tilde{a}^{(i)}) \to 0$, then, as  $\tilde a_1=0$, 
the arguments of Cases \ref{case:outermost_walls_sepaparted2} and \ref{subcase:away_from_boundary} apply. Hence
$f(\tilde{a}^{(i)}) \to \infty$ as
$i \to \infty$. Otherwise, the values $f(\tilde{a}^{(i)})$ will remain bounded. In both cases, we know that
$B_1 < 0$ and $a_1^{(i)}\to 1$. Thus $f(a^{(i)}) \to \infty$ as $i \to \infty$.
\end{subproof}
\end{subproof}

To prove that $f$ is bounded below, we consider a minimising sequence $(a^{(i)})_{i \in \N}$.
Then by what we have just proved, $\rho(a^{(i)})$ stays uniformly away from $0$. Since $f$ is bounded for $\rho(a) \geq C>0$, the conclusion follows.
\end{proof}

\begin{proof}[Proof of Lemma \ref{lem:blow_up_of_logarithmic_NEW}, statement \ref{item:unbounded_below}]
Fix $a \in A_N$ such
that $a_K < 0$ and $a_L > 0$. For $\eta \in (0, 1]$, define $a^{(\eta)} = (a_1^{(\eta)}, \dotsc, a_N^{(\eta)}) \in A_N$ with
\[
a_k^{(\eta)} = \begin{cases}
a_k & \text{if $k < K$ or $k> L$} \\
\eta a_k & \text{if $K \le k \le L$}.
\end{cases}
\]
Note that $1-(a_k^{(\eta)})^2$ stays away from $0$ uniformly and
$|\log \varrho(a^{(\eta)}_\ell, a^{(\eta)}_k) - \log(a^{(\eta)}_\ell-a^{(\eta)}_k)|$ is uniformly
bounded for every $k<\ell$ and every $\eta\in (0, 1]$. Thus, there exists a constant $C>0$ such that
\[
\begin{split}
f(a^{(\eta)}) & \leq C+\sum_{K \le k < \ell \le L} A_{k\ell} \left(\log(a_\ell - a_k) + \log \eta\right) + \sum_{\substack{k < \ell < K \text{ or } \\ L < k < \ell \text{ or } \\ k < K < L < \ell}} A_{k \ell} \log (a_\ell - a_k) \\
& \quad + \sum_{k < K \le \ell \le L} A_{k\ell} \log(\eta a_\ell - a_k) + \sum_{K \le k \le L < \ell} A_{k\ell} \log(a_\ell - \eta a_k).
\end{split}
\]
Now we let $\eta \searrow 0$. We note that $\log (\eta a_\ell - a_k) \to \log |a_k|$ when $k < K$ and $\log(a_\ell - \eta a_k) \to \log a_\ell$ when $\ell > L$.
Thus
\[
\limsup_{\eta \searrow 0} \left(f(a^{(\eta)}) - \log \eta \sum_{K \le k < \ell \le L} A_{k\ell}\right) < \infty.
\]
By the assumption $\sum_{K \le k < \ell \le L} A_{k\ell} >0$, we find that
$f(a^{(\eta)}) \to -\infty$ as $\eta \searrow 0$.
\end{proof}

\begin{proof}[Proof of Lemma \ref{lem:coefficients_for_alternating_signs}]
In order to prove statement \ref{item:negative}, we need to show that
\be
\label{rrr}
\sum_{K \le k < \ell \le L} (d_k - \cos \alpha)(d_\ell - \cos \alpha) < 0
\ee
for all $K, L \in \{1, \dotsc, N\}$ with $K < L$, provided that $\alpha \in \Theta$. 

If $N=2$, then \eqref{rrr} holds true for every $\alpha\in (0,\pi)$. For $N\geq 3$,
we consider the sums 
\[
S_0^K = \sum_{k = 0}^{K - 1} \sum_{\ell = k + 1}^{K - 1} ((-1)^k - \cos \alpha)((-1)^\ell - \cos \alpha)
\]
and
\[
S_1^K = \sum_{k = 1}^K \sum_{\ell = k + 1}^K ((-1)^k - \cos \alpha)((-1)^\ell - \cos \alpha)
\]
for any $K = 2, \dotsc, N$. We compute
\[
\begin{split}
S_0^K & = \frac{1}{2}\left(\sum_{k = 0}^{K - 1} \sum_{\ell = 0}^{K - 1} ((-1)^k - \cos \alpha)((-1)^\ell - \cos \alpha) - \sum_{k = 0}^{K - 1} ((-1)^k - \cos \alpha)^2\right) \\
& = \frac{1}{2} \sum_{k = 0}^{K - 1} \sum_{\ell = 0}^{K - 1} ((-1)^{k + \ell} -2(-1)^\ell \cos \alpha + \cos^2 \alpha) 
 - \frac{1}{2} \sum_{k = 0}^{K - 1} (1 -2(-1)^k \cos \alpha + \cos^2 \alpha)\\
& = \frac{1}{2} \sum_{k = 0}^{K - 1} \sum_{\ell = 0}^{K - 1} (-1)^{k + \ell} - K \cos \alpha \sum_{\ell = 0}^{K - 1} (-1)^\ell + \frac{K^2}{2} \cos^2 \alpha - \frac{K}{2} + \cos \alpha \sum_{k = 0}^{K - 1} (-1)^k - \frac{K}{2} \cos^2 \alpha \\
& = \frac{1}{2} \left(\sum_{k = 0}^{K - 1} (-1)^k\right)^2 + (1 - K) \cos \alpha \sum_{k = 0}^{K - 1} (-1)^k + \frac{K}{2}\left((K - 1) \cos^2 \alpha - 1\right).
\end{split}
\]
Similarly,
\[
S_1^K =  \frac{1}{2} \left(\sum_{k = 1}^K (-1)^k\right)^2 + (1 - K) \cos \alpha \sum_{k = 1}^K (-1)^k + \frac{K}{2}\left((K - 1) \cos^2 \alpha - 1\right).
\]
If $K$ is even, then
\[
\sum_{k = 0}^{K - 1} (-1)^k = \sum_{k = 1}^K (-1)^k = 0,
\]
and thus
\[
S_0^K = S_1^K = \frac{K}{2}((K - 1) \cos^2 \alpha - 1).
\]
If $K$ is odd, then
\[
\sum_{k = 0}^{K - 1} (-1)^k = 1 \quad \text{and} \quad \sum_{k = 1}^K (-1)^k = -1,
\]
and thus
\[
S_0^K = \frac{K - 1}{2} (K \cos^2 \alpha - 2\cos \alpha - 1)
\]
and
\[
S_1^K = \frac{K - 1}{2} (K \cos^2 \alpha + 2\cos \alpha - 1).
\]

Hence for  any given $N\geq 3$, we conclude that 
\eqref{rrr} is satisfied under the following conditions.
\begin{itemize}
\item If $N$ is even, then \eqref{rrr} reduces to the condition that
$S_0^{N-1} < 0$ and $S_1^{N-1} < 0$. This amounts to $(N - 1) \cos^2 \alpha + 2|\cos \alpha| - 1 < 0$, i.e.,
to the inequalities $\theta_N < \alpha < \pi - \theta_N$.
\item If $N$ is odd and $d_1=1$, we have the condition that $S_0^N < 0$ and $S_1^{N-2}<0$, i.e.,
\begin{align*}
N \cos^2 \alpha - 2 \cos \alpha -1 & < 0, \\
(N - 2) \cos^2 \alpha + 2\cos \alpha - 1 & < 0.
\end{align*}
If $\cos \alpha \ge 0$, then the second inequality is the strongest, otherwise it is the first.
Thus this case amounts to $\theta_{N-2}<\alpha<\pi-\theta_N$.

\item if $N$ is odd and $d_1=-1$, we have the condition that $S_1^N < 0$ and $S_0^{N-2}<0$, i.e., 
\begin{align*}
N \cos^2 \alpha + 2 \cos \alpha -1 & < 0, \\
(N - 2) \cos^2 \alpha - 2\cos \alpha - 1 & < 0.
\end{align*}
This leads to $\theta_{N}<\alpha<\pi-\theta_{N-2}$.
\end{itemize}

The proof of statement \ref{item:positive} is similar. Indeed, if $N$ is even,
then the sum can be made positive if, and only if, there exists $K \le N$ such that $S_0^K > 0$ or $S_1^K > 0$.
By the above computations, this is the case if, and only if, $\alpha \not\in [\theta_N, \pi - \theta_N]$.

If $N$ is odd, then it is convenient to consider the cases $\cos \alpha \ge 0$ and $\cos \alpha < 0$ separately.
Suppose that $\cos \alpha \ge 0$. If $d_1 = 1$ (and thus $d_N = 1$ as well), then we require the existence of $K \le N$ such that
$S_0^K > 0$ or $S_1^{K - 1} > 0$. It is readily checked that any of these inequalities
will imply that in particular $S_1^{N - 2} > 0$.
This leads to the condition $\alpha < \theta_{N - 2}$.
If $d_1 = -1$, then we require $K \le N$ such that $S_0^{K - 1} > 0$ or $S_1^K > 0$. In this case, the
term
$
S_1^N 
$
is the greatest. It is positive when $\alpha < \theta_N$.
The situation for $\cos \alpha < 0$ is similar.
\end{proof}

We can now answer the question at the beginning of this section.
If $\theta_N < \alpha < \pi - \theta_N$, then the repulsion between neighbouring walls
\emph{of different signs} will dominate and the renormalised energy will blow up when
two such walls approach each other or a wall (of any sign) approaches the boundary.
As discussed previously, no such conclusion can be expected when two walls
of the same sign approach each other. In fact, by \eqref{eqn:renormalised_energy_confined}, the renormalised energy tends to $-\infty$ when two neighbouring walls of the same sign approach one another. This is consistent with the energy landscape: after the `collision' of two such walls, the number of walls decreases. Thus the total energy, normalised by
$(\log \delta)^2$, decreases by $O(|\log \delta|)$.
As $\delta \searrow 0$, this should be interpreted as `$-\infty$' in terms of $W(a,d)$.

The above observation can be formulated as follows. 
Recall that for $a \in A_N$, the quantity $\rho(a)$ is defined in \eqref{def_rho}.

\begin{proposition}[Repulsion of N\'eel walls] \label{prop:blow-up_of_renormalised_energy}
Let $N \in \N$ such that $\theta_N < \alpha < \pi - \theta_N$. Suppose that $(a^{(i)})_{i \in \N}$ is a sequence in $A_N$ such that
$\rho(a^{(i)}) \to 0$ as $i \to \infty$ and $(d^{(i)})_{i \in \N}$
is a sequence in $\{\pm 1\}^N$. Suppose further that
\[
\liminf_{i \to \infty} \min \set{a_{k+1}^{(i)} - a_k^{(i)}}{1 \le k \le N-1 \text{ and } d_k^{(i)} = d_{k+1}^{(i)}} > 0.
\]
Then $W(a^{(i)}, d^{(i)}) \to \infty$ as $i \to \infty$.
\end{proposition}

\begin{proof}
By the definition of $W(a,d)$ and $\varrho(a_k, a_\ell)$, there exists a constant $C>0$ (depending on $N$) such that
\be
\label{prtz}
W(a, d) \ge -\frac{\pi}{2}\sum_{k=1}^N\gamma_k^2 \log(1 - a_k^2) + \pi\sum_{k < \ell} \gamma_k \gamma_\ell \log \varrho(a_\ell, a_k) - C
\ee
for any $a \in A_N$ and $d \in \{\pm 1\}^N$, where $\gamma_n = d_n - \cos \alpha$.
Under the extra assumption that $d_n^{(i)} = - d_{n + 1}^{(i)}$ for $n = 1, \dotsc, N - 1$
and for every $i \in \N$ (i.e., the signs $d_n^{(i)}$ alternate in $n$),  then in view of \eqref{prtz} and Lemma \ref{lem:blow_up_of_logarithmic_NEW} (applied to $A_{k\ell}=\pi \gamma_k\gamma_\ell$ and $B_k=-\frac\pi 2 \gamma_k^2$)
and Lemma \ref{lem:coefficients_for_alternating_signs},
the conclusion follows.

Next we consider the general case. We may assume without loss of generality that there exists a partition
$\Lambda_1, \dotsc, \Lambda_J$ of $\{1, \dotsc, N\}$ with the property that
\[
\liminf_{i \to \infty} \min \set{a_\ell^{(i)} - a_k^{(i)}}{1 \le k<\ell \le N \text{ and } \{k, \ell\} \not\subseteq \Lambda_j \text{ for } j = 1, \dotsc J} > 0,
\]
while
\[
\lim_{i \to \infty} \max\set{a_\ell^{(i)} - a_k^{(i)}}{1 \le k < \ell \le N \text{ and } k, \ell \in \Lambda_j} = 0
\]
for any $j = 1, \ldots, J$ such that $|\Lambda_j| \ge 2$.
Otherwise, we pass to a subsequence with this property. (We think of each $\Lambda_j$ as the set of indices of a cluster
of points approaching one another as $i \to \infty$.) We may further assume that $d_k^{(i)}$ is independent of $i$.

Fix $j \in \{1, \dotsc, J\}$. Then for any sufficiently large value of $i$, it is clear that $\Lambda_j$ comprises a consecutive
set of numbers, i.e., there exist $K_j, L_j \in \{1, \dotsc, N\}$ such that $\Lambda_j = \{K_j, \dotsc, L_j\}$, since the
points of each $a^{(i)}$ are ordered. The hypothesis of the proposition then implies that for any fixed $j$, either $d_k^{(i)} = (-1)^k$ for all $k \in \Lambda_j$
or $d_k^{(i)} = -(-1)^k$ for all $k \in \Lambda_j$. Since $\theta_N$ is increasing in $N$,
Lemma \ref{lem:blow_up_of_logarithmic_NEW} and Lemma \ref{lem:coefficients_for_alternating_signs} show that
\[
-\frac{\pi}{2}\sum_{k=K_j}^{L_j}(\gamma_k^{(i)})^2 \log(1 - (a_k^{(i)})^2) +\pi \sum_{K_j \le k < \ell \le L_j} \gamma_k^{(i)} \gamma_\ell^{(i)} \log \varrho(a_\ell^{(i)}, a_k^{(i)}) \to \infty
\]
as $i \to \infty$ for every $j = 1, \dotsc, J$ such that $|\Lambda_j| \ge 2$.
(Here $\gamma_k^{(i)} = d_k^{(i)} - \cos \alpha$.) Finally, we observe that the condition $\rho(a^{(i)}) \to 0$ implies that either $a_1^{(i)} \to -1$ or $a_N^{(i)} \to 1$
or there exists $j \in \{1, \dotsc, N\}$ with $|\Lambda_j| \ge 2$. As we have inequality \eqref{prtz}, the claim then follows.
\end{proof}

Now we have all the tools for the proof of Proposition \ref{pro:min_W_conf}.

\begin{proof}[Proof of Proposition \ref{pro:min_W_conf}]
For the proof of statement \ref{item:minimiser_exists}, let $(a^{(i)})_{i \in \N}$ be a minimising sequence
of $W(\cdot, d)$ over $A_N$. (By what we know so far, this might mean that $W(a^{(i)}, d) \to -\infty$ as $i\to \infty$.)
Assume by contradiction that for a subsequence (still denoted by $(a^{(i)})_{i \in \N}$) we have $\rho(a^{(i)})\to 0$ as $i\to \infty$.
Then we use \eqref{prtz} to estimate $W$ from below in terms of quantities
that can be controlled with the help of
Lemmas \ref{lem:blow_up_of_logarithmic_NEW} and \ref{lem:coefficients_for_alternating_signs}.
We conclude that $W(a^{(i)}, d)\to +\infty$, which contradicts the minimising character of $(a^{(i)})_{i \in \N}$. Thus, $\liminf_{i\to \infty} \rho(a^{(i)})>0$. Passing to a subsequence if necessary, we may assume that $a^{(i)}\to a$ as $i\to \infty$ with $a\in A_N$. By definition \eqref{eqn:renormalised_energy_confined} of $W(\blank, d)$, we deduce that $W(a^{(i)}, d)\to W(a,d)$ as $i\to \infty$, i.e., the infimum of $W(\blank, d)$ over $A_N$ is finite and achieved by $a$. 

For statement \ref{item:inf_-infty}, we first use Lemma \ref{lem:coefficients_for_alternating_signs},
which tells us that for $N \ge 3$ and $\alpha \not\in \overline{\Theta}$, there exist $K<L$
such that $\sum_{K\leq k<\ell\leq L} \gamma_k \gamma_\ell>0$. Therefore,
by Lemma \ref{lem:blow_up_of_logarithmic_NEW}, there exists a sequence $(a^{(i)})_{i \in \N}$ in $A_N$
such that $\rho(a^{(i)})\to 0$ and
\[
-\frac{\pi}{2}\sum_{k=1}^{N}\gamma_k^2 \log(1 - (a_k^{(i)})^2) +\pi \sum_{1 \le k < \ell \le N} \gamma_k \gamma_\ell \log \varrho(a_\ell^{(i)}, a_k^{(i)}) \to -\infty
\]
as $i\to \infty$. Comparing with definition \eqref{eqn:renormalised_energy_confined}, we see that there exists
a constant $C$ such that
$$W(a^{(i)}, d)\leq C-\frac{\pi}{2}\sum_{k=1}^{N}\gamma_k^2 \log(1 - (a_k^{(i)})^2) +\pi \sum_{1 \le k < \ell \le N} \gamma_k \gamma_\ell \log \varrho(a_\ell^{(i)}, a_k^{(i)}).$$
Hence $W(a^{(i)}, d) \to -\infty$ as well.
\end{proof}

\subsection{Energy estimates}
\label{sec:en_est_confined}

In this section we improve some of the energy estimates from our previous paper \cite{Ignat-Moser:16}.
In particular, this will remove the need to bound the distance between two N\'eel walls from
below.

We use several of the tools from the previous paper here. Therefore, we discuss them briefly
before giving some improved estimates. This includes in particular the construction of
a limiting stray field potential, given in terms of a function on the upper half-plane
$\R_+^2 = \R \times (0, \infty)$ satisfying a certain boundary value problem.

We use the following notation.
For $r > 0$, we set
\[
\Omega_r(a) = \R_+^2 \setminus \bigcup_{n = 1}^N B_r(a_n, 0),
\]
where $B_r(x)$ denotes the open disk centred at $x$ of radius $r>0$. 
Apart from the half-plane $\R_+^2$, which is sometimes regarded as
a subset of $\C$ by the usual identification, we consider
\[
S = \set{x_1 + i x_2 \in \C}{x_1 > 0, \ 0 < x_2 < \pi}
\]
and the map $F \colon S \to \R_+^2$ given by
\[
F(w) = - \frac{1}{\cosh w}=-\frac2{e^w+e^{-w}}, \quad  w\in S.
\]
Furthermore, we consider the function $\hat{u} \colon S \to \R$ with $$\hat{u}(w) = \frac{\pi}{2} - \Im w$$ and
$u \colon \R_+^2 \to \R$ given by $$u = \hat{u} \circ F^{-1}.$$ This function solves the boundary value problem
\label{confined_limiting_stray}
\begin{alignat*}{2}
\Delta u & = 0 & \quad & \text{in $\R_+^2$}, \\
u & = \frac{\pi}{2} && \text{on $\{0\} \times (-1, 0)$}, \\
u & = -\frac{\pi}{2} && \text{on $\{0\} \times (0, 1)$}, \\
\dd{u}{x_2} & = 0 && \text{on $\{0\} \times (-\infty, -1)$ and on $\{0\} \times (1, \infty)$}.
\end{alignat*}
Note that $|u|\leq \frac\pi 2$ in $\R^2_+$ and $u(z) \to 0$ as $|z|\to \infty$. 
For $b \in (-1, 1)$, we recall the M\"obius transform $\Phi_b$ defined in \eqref{mobius}.
Observe that $\Phi_b^{-1} = \Phi_{-b}$. We define $$u_b = u \circ \Phi_{-b}.$$
For $a \in A_N$ and $d \in \{\pm 1\}^N$, set $\gamma_n = d_n - \cos \alpha$ and
\[
u_{a, d}^* = \sum_{n = 1}^N \gamma_n u_{a_n}.
\]
This function plays the role of a limiting stray field potential for an array of
N\'eel walls at the points $a_1, \dotsc, a_N$ of signs $d_1, \dotsc, d_N$.
The renormalised energy is related to
\be
\label{minusW}
\frac{1}{2} \lim_{r \searrow 0} \left(\int_{\Omega_r(a)} |\nabla u_{a, d}^*|^2 \, dx - \pi \sum_{n = 1}^N \gamma_n^2 \log \frac{1}{r}\right),
\ee
which happens to equal $-W(a,d)$. However, 
the full renormalised energy
contains another term, which turns out to be $2W(a, d)$, giving $W(a, d)$ as the sum \cite[page 442]{Ignat-Moser:16}.
The quantity in \eqref{minusW} can be identified as the contribution of the interaction between all the
logarithmically decaying tails of the N\'eel walls, whereas the other term corresponds to the interaction
between pairs of a tail from one wall and the core of another. This relation between the two terms
can surely be no coincidence, and indeed we discover something similar for the unconfined problem below,
but the reason is unclear.

In order to improve the results from the previous paper \cite{Ignat-Moser:16}, we
need above all to refine some estimates
for the Dirichlet energy in \eqref{minusW}. We begin with a result
similar to \cite[Lemma 9]{Ignat-Moser:16}, but we prove an estimate in the half-disk $B_R^+(b, 0) = B_R(b, 0) \cap \R_+^2$
instead of the half-space $\R^2_+$. This is the natural estimate in the context of Ginzburg-Landau theory,
because $u_b$ behaves like the phase of a vortex of degree $1$. 

\begin{lemma} \label{lem:energy_in_annulus}
Let $b \in (-1, 1)$ and $0 < r < R \le 1 - |b|$. Then
\[
\int_{B_R^+(b, 0) \setminus B_r(b, 0)} |\nabla u_b|^2 \, dx \le \pi \log \frac{R}{r} + \frac{\pi}{2} \log 76.
\]
\end{lemma}

\begin{proof}
By the symmetry, it suffices to consider the case $b \in [0, 1)$. The first step is to examine the set
$\Phi_{-b}(B_R^+(b, 0) \setminus B_r(b, 0))$. To this end, we fix $\rho \in (0, 1 - b]$ and we observe
that $\Phi_{-b}(\partial^+ B_\rho(b, 0))$ is a semicircle centred on $\R \times \{0\}$ by the standard
properties of the M\"obius transform (where $\partial^+ B_\rho(b, 0) = \partial B_\rho(b, 0) \cap \R_+^2$).
In order to determine this semicircle, it suffices to compute
\[
\Phi_{-b}(b + \rho) = \frac{\rho}{1 - b^2 - b\rho} \quad \text{and} \quad \Phi_{-b}(b - \rho) = -\frac{\rho}{1 - b^2 + b\rho}.
\]
Observing that $\frac{\rho}{1 - b^2 - b\rho} > \frac{\rho}{1 - b^2 + b\rho}$, we then see that
\[
\Phi_{-b}(B_\rho^+(b, 0)) \subseteq B_{\rho/(1 - b^2 - b\rho)}^+(0), \quad \Phi_{-b}(\R_+^2 \setminus B_\rho(b, 0)) \subseteq \R_+^2 \setminus B_{\rho/(1 - b^2 + b\rho)}^+(0).
\]
Thus
\[
\Phi_{-b}(B_R^+(b, 0) \setminus B_r(b, 0)) \subseteq B_{R/(1 - b^2 - bR)}^+(0) \setminus B_{r/(1 - b^2 + br)}(0).
\]

Set $\tilde{R} = R/(1 - b^2 - bR)$ and $\tilde{r} = r/(1 - b^2 + br)$.
Using the identity
\[
|\cosh w|^2 = \frac{1}{2} \cosh(2 \Re w) + \frac{1}{2} \cos(2 \Im w), \quad w \in \C,
\]
we see that for any $\rho > 0$,
\begin{align*}
F^{-1}(\R_+^2 \setminus B_\rho(0)) \subseteq \set{w \in S}{\Re w \leq \frac{1}{2} \arcosh \left(\frac{2}{\rho^2} + 1\right)}, \\
F^{-1}(B_\rho^+(0)) \subseteq \set{w \in S}{\Re w > \frac{1}{2} \arcosh \left(\frac{2}{\rho^2} - 1\right)}.
\end{align*}
As conformal maps leave the Dirichlet energy invariant, it follows that
\begin{align*}
\int_{B_R^+(b, 0) \setminus B_r(b, 0)} |\nabla u_b|^2 \, dx
& \le \int_{B_{\tilde{R}}^+(0) \setminus B_{\tilde{r}}(0)} |\nabla u|^2 \, dx \\
& \le \int_{\set{w \in S}{\arcosh(2/\tilde{R}^2 - 1) < 2\Re w < \arcosh(2/\tilde{r}^2 + 1)}} |\nabla \hat{u}|^2 \, dx \\
& = \frac{\pi}{2} \left(\arcosh\left(\frac{2(1 - b^2 + br)^2}{r^2} + 1\right) - \arcosh\left(\frac{2(1 - b^2 - bR)^2}{R^2} - 1\right)\right).
\end{align*}

Note that for any $y \ge 1$,
\[
\begin{split}
|\arcosh y - \log y| & = \left|\log\left(y + \sqrt{y^2 - 1}\right) - \log y\right| = \log\left(1 + \sqrt{1 - \frac{1}{y^2}}\right)  \le \log 2.
\end{split}
\]
Hence
\[
\int_{B_R^+(b, 0) \setminus B_r(b, 0)} |\nabla u_b|^2 \, dx \le \frac{\pi}{2} \log \left(\frac{R^2(2(1 - b^2 + br)^2 + r^2)}{r^2(2(1 - b^2 - bR)^2 - R^2)}\right) + \pi \log 2.
\]
Using the inequality $R \le 1 - b$, we obtain $1 - b^2 - bR \ge 1 - b^2 - b(1 - b) = 1 - b$. Similarly,
$1 - b^2 + br \le 1 - b^2 + b(1 - b) = (1 + 2b)(1 - b)$. Therefore, using $r \le 1 - b$ and $R \le 1 - b$ also to
estimate $r^2$ and $R^2$, we conclude that
\begin{align*}
\int_{B_R^+(b, 0) \setminus B_r(b, 0)} |\nabla u_b|^2 \, dx &
 \le \pi \log \frac{R}{r} + \frac{\pi}{2} \log \left(\frac{2(1 + 2b)^2 (1 - b)^2 + (1 - b)^2}{2(1 - b)^2 - (1 - b)^2}\right) 
 + \pi \log 2 \\
& = \pi \log \frac{R}{r} + \frac{\pi}{2} \log \left(2(1 + 2b)^2 + 1\right) + \pi \log 2 \\
& \le \pi \log \frac{R}{r} + \frac{\pi}{2} \log 19 + \pi \log 2 \\
& \le \pi \log \frac{R}{r} + \frac{\pi}{2} \log 76,
\end{align*}
as required.
\end{proof}

Note that $u_{a, d}^*$ does not belong to $\dot{H}^1(\R_+^2)$
because of the singularities at $(a_n, 0)$. As a consequence of the preceding
inequality, however, we can regularise it near the singular points and at the same
time obtain good estimates.

\begin{lemma} \label{lem:modification_of_u_{a,d}^*}
Let $c > 0$ and $N \in \N$. Then there exists $C > 0$ such that for any $a \in A_N$ and $d \in \{\pm 1\}^N$ and for any
$\epsilon \in (0, \frac{1}{2}]$ with $\rho(a) \ge c\delta= c \eps|\log \eps|$, there exists $\xi \in \dot{H}^1(\R_+^2)$ such that
$\xi = u_{a, d}^*$ in $\Omega_{c\delta}(a)$ and such that the inequalities
\[
\sup_{x_1 \in \R} |u_{a, d}^*(x_1, 0) - \xi(x_1, 0)| \le \pi
\]
and\footnote{This inequality provides an upper bound for the quantity \eqref{minusW},
because $\xi=u_{a, d}^*$ in $\Omega_{c\delta}(a)$.}
\[
\int_{\R_+^2} |\nabla \xi|^2 \, dx \le \pi \log \frac{1}{\delta} \sum_{n = 1}^N \gamma_n^2 - 2W(a, d) + C
\]
are satisfied, where $\gamma_n = d_n - \cos \alpha$ for $n = 1, \dotsc, N$.
\end{lemma}

\begin{proof}
Fix $n \in \{1, \dotsc, N\}$. Then Lemma \ref{lem:energy_in_annulus} implies that
\[
\int_{B_{c\delta}^+(a_n, 0) \setminus B_{c\delta/2}(a_n, 0)} |\nabla u_{a_n}|^2 \, dx \le \frac{\pi}{2} \log 304.
\]
Note that the function
\[
f(x) = u_{a_n}(x_1 + a_n, x_2) + \arctan \frac{x_1}{x_2}
\]
is harmonic in $\R_+^2$ and constant on $(-c\delta, c\delta) \times \{0\}$. Moreover, as $|\nabla f|^2(x)\leq 
2(|\nabla u_{a_n}|^2(x_1 + a_n, x_2) + 1/|x|^2)$, we estimate
\begin{equation} \label{eqn:L^2-estimate_for_grad_f}
\int_{B_{c\delta}^+(0) \setminus B_{c\delta/2}(0)} |\nabla f|^2 \, dx \le \pi \log 1216.
\end{equation}
We may extend the function $\tilde{f}(x) = f(x) - f(0)$ to $B_{c\delta}(0)$ by the odd reflection
$\tilde{f}(x_1, -x_2) = - \tilde{f}(x_1, x_2)$ for $(x_1, x_2) \in B_{c\delta}^+(0)$.
Then $\tilde{f}$ is harmonic in $B_{c\delta}(0)$ and so are its derivatives. The mean
value formula then gives
\[
\biggl|\dd{\tilde{f}}{x_i}(x)\biggr| \le \frac{16}{\pi c^2 \delta^2} \int_{B_{c\delta/4}(x)} \biggl|\dd{\tilde{f}}{x_i}(y)\biggr| \, dy
\]
for any $x \in \partial B_{3c\delta/4}(0)$. The maximum principle allows us to extend this
inequality to all $x \in B_{3c\delta/4}(0)$. Combining the resulting estimate with \eqref{eqn:L^2-estimate_for_grad_f}
and H\"older's inequality, we obtain 
a universal constant $C_1$ such that
\[
\|\nabla f\|_{L^\infty(B_{c\delta/2}^+(0))} \le \frac{C_1}{c\delta}.
\]
That is,
\begin{equation} \label{eqn:gradient_estimate_near_a_n}
\left|\nabla u_{a_n}(x) + \frac{(x_2, a_n - x_1)}{(x_1 - a_n)^2 + x_2^2}\right| \le \frac{C_1}{c\delta}
\end{equation}
for every $x \in B_{c\delta/2}^+(a_n, 0)$.

In addition, for $k \not= n$, we obtain the inequality
\[
\begin{split}
\int_{B_{c\delta}^+(a_k, 0)} |\nabla u_{a_n}|^2 \, dx & \le \int_{B_{|a_n - a_k|+ c\delta}^+(a_n, 0) \setminus B_{|a_n - a_k| - c\delta}(a_n, 0)} |\nabla u_{a_n}|^2 \, dx \\
& \le \pi \log 3 + \frac{\pi}{2} \log 76 = \frac{\pi}{2} \log 684.
\end{split}
\]
(Here we have used Lemma \ref{lem:energy_in_annulus} and the fact that $|a_n - a_k| \ge 2c\delta$ for
$k \not= n$.) From this we conclude, as above, that
\begin{equation} \label{eqn:gradient_estimate_near_a_k}
|\nabla u_{a_n}(x)| \le \frac{C_2}{c\delta}
\end{equation}
in $B_{c\delta/2}^+(a_k, 0)$ for a universal constant $C_2$.

Now choose $\eta \in C_0^\infty(B_{c\delta/2}(0))$ with $\eta \equiv 1$ in $B_{c\delta/4}(0)$
and with $0 \le \eta \le 1$ and $|\nabla \eta| \le \frac{8}{c\delta}$ everywhere. Set
\[
\tilde{u}_{a_n}(x) = (1 - \eta(x_1 - a_n, x_2)) u_{a_n}(x) + \eta(x_1 - a_n, x_2) \fint_{B_{c\delta/2}^+(a_n, 0) \setminus B_{c\delta/4}(a_n, 0)} u_{a_n} \, dy.
\]
As $|u|\leq \frac \pi 2$, we have $-\frac{\pi}{2} \le \tilde{u}_{a_n} \le \frac{\pi}{2}$. Moreover, it is clear that $\tilde{u}_{a_n}$
coincides with $u_{a_n}$ in $\R_+^2 \setminus B_{c\delta/2}(a_n, 0)$. Thus, by Lemma 9 in \cite{Ignat-Moser:16},
we find universal constants $C_3, C_4$ such that
\begin{align*}
\int_{\R_+^2} |\nabla \tilde{u}_{a_n}|^2\, dx & \le \int_{\R_+^2 \setminus B_{c\delta}(a_n, 0)} |\nabla u_{a_n}|^2\, dx + C_3 \\
& \le \pi \log \frac{1}{\delta} + \pi \log \frac{1}{c} + \pi \log (2 - 2a_n^2) 
+ \frac{C_4 c\delta |a_n|}{1 - a_n^2} + \frac{C_4 c^2 \delta^2}{(1 - a_n^2 - c\delta |a_n|)^2} + C_3.
\end{align*}
Since $1 - a_n^2 = (1 - a_n) (1 + a_n) \ge c\delta$ and
$1 - a_n^2 - c\delta |a_n| = (1 - |a_n|)(1 + |a_n|) - c\delta |a_n| \ge c\delta (1 + |a_n|) - c\delta |a_n| = c\delta$,
we conclude that
\begin{equation} \label{eqn:modified_stray_field_potential1}
\int_{\R_+^2} |\nabla \tilde{u}_{a_n}|^2\, dx \le \pi \log \frac{1}{\delta} + \pi \log (2 - 2a_n^2) + C_5
\end{equation}
for a constant $C_5$ that depends only on $c$.

Furthermore, for $k \not= n$,
\begin{multline*}
\int_{\R_+^2} \nabla \tilde{u}_{a_k} \cdot \nabla \tilde{u}_{a_n} \, dx = \int_{\R_+^2} \nabla u_{a_k} \cdot \nabla u_{a_n} \, dx + \int_{\R_+^2} \nabla \tilde{u}_{a_k} \cdot (\nabla \tilde{u}_{a_n} - \nabla u_{a_n}) \, dx \\
+ \int_{\R_+^2} (\nabla \tilde{u}_{a_k} - \nabla u_{a_k}) \cdot \nabla u_{a_n} \, dx.
\end{multline*}
Define $\zeta_n(x) = \arctan \frac{x_1 - a_n}{x_2}$. Then we estimate
\begin{align*}
\int_{\R_+^2} \nabla \tilde{u}_{a_k} \cdot (\nabla \tilde{u}_{a_n} - \nabla u_{a_n}) \, dx 
 & = \int_{B_{c\delta/2}^+(a_n, 0)} \nabla \tilde{u}_{a_k} \cdot (\nabla \tilde{u}_{a_n} - \nabla u_{a_n}) \, dx \\
& \le \frac{\pi c^2 \delta^2}{8} \|\nabla \tilde{u}_{a_k}\|_{L^\infty(B_{c\delta/2}^+(a_n, 0))} \|\nabla \tilde{u}_{a_n}\|_{L^\infty(B_{c\delta/2}^+(a_n, 0))} \\
& \quad + \frac{\pi c^2 \delta^2}{8} \|\nabla \tilde{u}_{a_k}\|_{L^\infty(B_{c\delta/2}^+(a_n, 0))} \|\nabla u_{a_n} + \nabla \zeta_n\|_{L^\infty(B_{c\delta/2}^+(a_n, 0))} \\
& \quad + \|\nabla \tilde{u}_{a_k}\|_{L^\infty(B_{c\delta/2}^+(a_n, 0))} \|\nabla \zeta_n\|_{L^1(B_{c\delta/2}^+(a_n, 0))}.
\end{align*}
Using \eqref{eqn:gradient_estimate_near_a_k} and \eqref{eqn:gradient_estimate_near_a_n} and observing that
\[
\|\nabla \zeta\|_{L^1(B_{c\delta/2}^+(a_n, 0))} = \frac{c\delta \pi}{2},
\]
we therefore find that there exists a constant $C_6 = C_6(c)$ such that
\[
\int_{\R_+^2} \nabla \tilde{u}_{a_k} \cdot (\nabla \tilde{u}_{a_n} - \nabla u_{a_n}) \, dx \le C_6.
\]
Similarly, we prove that
\[
\int_{\R_+^2} (\nabla \tilde{u}_{a_k} - \nabla u_{a_k}) \cdot \nabla u_{a_n} \, dx=\int_{B^+_{c\delta/2}(a_k,0)} (\nabla \tilde{u}_{a_k} - \nabla u_{a_k}) \cdot \nabla \tilde u_{a_n} \, dx \le C_6,
\]
because $\tilde u_{a_n}= u_{a_n}$ on $\R^2_+\setminus B^+_{c\delta/2}(a_n,0)$ and $|a_n-a_k|\geq 2c\delta$. 
By \cite[Lemma 8]{Ignat-Moser:16}, we obtain
\begin{equation} \label{eqn:modified_stray_field_potential2}
\int_{\R_+^2} \nabla \tilde{u}_{a_k} \cdot \nabla \tilde{u}_{a_n} \, dx  \le \int_{\R_+^2} \nabla u_{a_k} \cdot \nabla u_{a_n} \, dx + 2C_6  
 = \pi \log \left(\frac{1 + \sqrt{1 - \varrho(a_k, a_n)^2}}{\varrho(a_k, a_n)}\right) + 2C_6. 
\end{equation}

We now set
\[
\xi = \sum_{n = 1}^N \gamma_n \tilde{u}_{a_n},
\]
then the desired inequality for the Dirichlet energy follows from inequalities \eqref{eqn:modified_stray_field_potential1}
and \eqref{eqn:modified_stray_field_potential2}.
The inequality for $u_{a, d}^*(x_1, 0) - \xi(x_1, 0)$ is a consequence of the construction.
\end{proof}

Under assumptions similar to Proposition \ref{prop:blow-up_of_renormalised_energy},
where only neighbouring walls of different sign can approach one another, we prove the following a priori
lower bound for the energy $E_\eps$.

\begin{proposition}[A priori lower energy bound] \label{prop:lower_energy_bound_confined}
Let $N \in \N$ and $\sigma > 0$. Suppose that $\theta_N < \alpha < \pi - \theta_N$. Then there exists $C_0$ with the following
property. Suppose that $a \in A_N$ and $d \in \{\pm 1\}^N$ such that for $n = 1, \dotsc N - 1$, either $a_{n + 1} - a_n \ge \sigma$ or
$d_{n + 1} = - d_n$. Then the inequality
\[
E_\epsilon(m) \ge \frac{\pi}{2\log \frac{1}{\delta}} \sum_{n = 1}^N \gamma_n^2 + \frac{W(a, d) - C_0}{(\log \delta)^2}
\]
holds true for all $m \in M(a, d)$ and all $\epsilon \in (0, \frac{1}{4}]$.
\end{proposition}

The proof depends in part on the following uniform bound on the renormalised energy $W(a,d)$,
which holds if $\rho(a)\geq c\delta$.

\begin{lemma} \label{claimW}
Let $c>0$. There exist $\tilde{C}_0>0$ and $\hat{C}_0>0$ such that for every $a\in A_N$ with $\rho(a) \ge c\delta$
and for every $d\in \{\pm 1\}^N$, if $m \in M(a, d)$ satisfies 
\[
E_\epsilon(m) \le \frac{\pi}{2\log \frac{1}{\delta}} \sum_{n = 1}^N \gamma_n^2 + \frac{W(a, d) - \tilde{C}_0}{(\log \delta)^2},
\]
then $|W(a,d)|\leq \hat C_0$.
\end{lemma}

\begin{proof}
 We first note that $E_\epsilon$ attains its minimum in $M(a, d)$, as observed in
our previous paper \cite[Proposition 1]{Ignat-Moser:16}.
We may therefore assume that $m$ minimises $E_\epsilon$ in $M(a, d)$, as otherwise, we may simply
replace it with a minimiser. 

If $\rho(a) \ge c\delta$, then $W(a,d)$ is of order $|\log \delta|$.
In particular, there exists a constant $C_1 = C_1(c, N)$ such that
$E_\eps(m)\leq \frac{C_1}{|\log \delta|}$. Thus, writing $m = (\cos \phi, \sin \phi)$, by \cite[Lemma 13]{Ignat-Moser:16} and the Cauchy-Schwarz inequality,
we obtain
\[
\sum_{n = 1}^N \int_{a_n - c\delta}^{a_n + c\delta} |m_1'| \, dx_1 \le 
\sum_{n = 1}^N  \sqrt{\int_{a_n - c\delta}^{a_n + c\delta}\sin^2 \f \, dx_1 \int_{a_n - c\delta}^{a_n + c\delta} (\f')^2 \, dx_1}    \leq \frac{C_2}{\log \frac{1}{\delta}}
\]
for a constant $C_2$ that depends only on $c$ and $N$. Set 
$$\Gamma = \sum_{n = 1}^N \gamma_n^2$$ and let $\xi$ be the function from
Lemma \ref{lem:modification_of_u_{a,d}^*}. Note that $u_{a, d}^*(\blank, 0)$ is constant in $(a_{n - 1}, a_n)$
for $n = 1, \dotsc, N + 1$ (where $a_0 = -1$ and $a_{N + 1} = 1$) with a jump of size $-\gamma_n \pi$ at
$a_n$ for $n = 1, \dotsc, N$. The function $m_1$, on the other hand, satisfies $m_1(a_n) = d_n$.
We use the fundamental theorem of calculus in each of the intervals $(a_{n - 1}, a_n)$, obtaining
\[
\begin{split}
\pi \Gamma & = \int_{-1}^1 u_{a, d}^*(x_1, 0) m_1'(x_1) \, dx_1 \\
& = \int_{-1}^1 (u_{a, d}^*(x_1, 0) - \xi(x_1, 0)) m_1'(x_1) \, dx_1 + \int_{-1}^1 \xi(x_1, 0) m_1'(x_1) \, dx_1.
\end{split}
\]
Moreover,
\[
\int_{-1}^1 (u_{a, d}^*(x_1, 0) - \xi(x_1, 0)) m_1'(x_1) \, dx_1 \le \pi \sum_{n = 1}^N \int_{a_n - c\delta}^{a_n + c\delta} |m_1'| \, dx_1 \le \frac{C_2 \pi}{\log \frac{1}{\delta}},
\]
while
\[
\int_{-1}^1 \xi(x_1, 0) m_1'(x_1) \, dx_1 = \int_{\R_+^2} \nabla \xi \cdot \nabla u \, dx \le \|\nabla \xi\|_{L^2(\R_+^2)} \sqrt{2E_\epsilon(m)}.
\]
The inequality of Lemma \ref{lem:modification_of_u_{a,d}^*} then gives a constant $C_3 = C_3(N, c)> 0$ such that
$$
\pi^2\left(\Gamma - \frac{C_2}{\log \frac{1}{\delta}}\right)^2 
\le \left(\pi\Gamma \log \frac{1}{\delta} - 2W(a, d) + C_3\right) \left(\frac{\pi\Gamma}{\log \frac{1}{\delta}} + \frac{2W(a, d) - 2\tilde{C}_0}{(\log \delta)^2}\right),
$$
which implies that
\begin{multline*}
\pi^2 \Gamma^2 - \frac{2C_2 \pi^2\Gamma}{\log \frac{1}{\delta}} + \frac{C_2^2 \pi^2}{(\log \delta)^2} \\
\le \pi^2 \Gamma^2 + \frac{\pi \Gamma (C_3 - 2\tilde{C}_0)}{\log \frac{1}{\delta}} - \frac{4 (W(a, d))^2}{(\log \delta)^2} + \frac{2W(a, d)(C_3 + 2\tilde{C}_0)}{(\log \delta)^2} - \frac{2\tilde{C}_0C_3}{(\log \delta)^2}.
\end{multline*}
If we choose $\tilde{C}_0 \ge C_2 \pi - \frac{C_3}{2}$, then it follows that
\[
4(W(a, d))^2 \le 2W(a, d)(C_3 + 2\tilde{C}_0).
\]
Therefore, the quantity $W(a, d)$ is bounded by a constant that depends only on $N$ and $c$.
\end{proof}

\begin{proof}[Proof of Proposition \ref{prop:lower_energy_bound_confined}]
 We use different arguments depending on the magnitude of $\rho(a)/\delta$.
\setcounter{stepno}{0}
\begin{subproof}{Step \step: small distances}
For a fixed number $c\leq 1$, depending only on $N$ (but to be determined later), we first show that the desired inequality holds
when $\rho(a) < c\delta$.

To this end, let $h = \rho(a)$ and set $q = 1 - \cos \theta_N$. Since $m \in M(a, d)$,
there exist $b_1, b_2 \in [-1, 1]$ with $b_1 < b_2$ and $b_2 - b_1 \le h$ such that $m_1(b_1) = \pm 1$ and $m_1(b_2) = \cos \alpha$
or vice versa. Hence
\[
q \le \left|\int_{b_1}^{b_2} m_1' \, dx\right| \le \sqrt{\frac{2h}{\epsilon} E_\epsilon(m)}
\]
by the fundamental theorem of calculus and the Cauchy-Schwarz inequality.
It follows that
\[
E_\epsilon(m) \ge \frac{q^2 \epsilon}{2h}.
\]
On the other hand, by the definition of $W$, it is clear that there exists a constant $C_1 = C_1(N)$ such that the inequality $h < \delta$ will imply that
\[
\frac{\pi}{2\log \frac{1}{\delta}} \sum_{n = 1}^N \gamma_n^2 + \frac{W(a, d)}{(\log \delta)^2} \le \frac{C_1 \log \frac{1}{h}}{(\log \delta)^2}.
\]
Thus if we choose $c \le 1$, then it suffices to show that
\[
\frac{q^2 \epsilon}{2h} \ge \frac{C_1 \log \frac{1}{h}}{(\log \delta)^2}.
\]
As $\epsilon |\log \delta| = \delta - \epsilon \log |\log \epsilon|$, this inequality is equivalent to
\[
q^2\left(\delta \log \frac{1}{\delta} - \epsilon \log \frac{1}{\delta} \log \log \frac{1}{\epsilon}\right) \ge 2C_1 h \log \frac{1}{h}.
\]
It is easy to see that $\epsilon \log \log \frac{1}{\epsilon} \le \delta/2$ for all $\epsilon \in (0, 1)$. Thus it suffices to show that
\[
q^2 \delta \log \frac{1}{\delta} \ge 4C_1  h \log \frac{1}{h}
\]
for $h \le c \delta$. But this is clear if $c > 0$ is chosen sufficiently small. 
\end{subproof}

Before we proceed to the second step, we point out that so far, we have not used the hypothesis that
$a_{n + 1} - a_n \ge \sigma$ or $d_{n + 1} = - d_n$ for all $n$. 
This is necessary only when $\rho(a) \ge c\delta$.

\begin{subproof}{Step \step: large distances}
Now we assume that $\rho(a) \ge c\delta$. If $m$ does \emph{not} satisfy the inequality in the hypothesis of Lemma \ref{claimW}, then there is
nothing to prove. Thus, by Lemma \ref{claimW}, we may assume that $|W(a, d)|$ is bounded by a
constant independent of $a$ or $d$.

Due to the hypothesis that $\alpha\in (\theta_N, \pi-\theta_N)$ and that either $a_{n + 1} - a_n \ge \sigma$ or
$d_{n + 1} = - d_n$, Proposition~\ref{prop:blow-up_of_renormalised_energy} applies.
It follows that the N\'eel walls are uniformly separated in the sense that
$\rho(a) \ge C_2$ for a constant $C_2 > 0$ that depends only on $N$.
Now we can use the theory of our previous paper \cite{Ignat-Moser:16}, and in particular
Theorem 28, which gives a constant $C_3 = C_3(N, \sigma)$ such that
\[
E_\epsilon(m) \ge \frac{\pi}{2\log \frac{1}{\delta}} \sum_{n = 1}^N \gamma_n^2 - \frac{C_3}{(\log \delta)^2}.
\]
As we already have a bound for $|W(a, d)|$, it now suffices to choose $C_0$ sufficiently large.
\end{subproof}
\end{proof}

\subsection{Proofs of the main results} \label{sect:main_results_confined}

\begin{proof}[Proof of Proposition \ref{prop:continuous_dependence_on_a}]
Let $a \in A_N$ and set $R = \frac{1}{2}\rho(a)$. Fix $d \in \{\pm 1\}^N$.
For $\epsilon \in (0, \frac{1}{4}]$, choose minimisers $m_\epsilon$ of
$E_\epsilon$ in $M(a, d)$. These exist and are smooth by \cite[Proposition 1]{Ignat-Moser:16}.

Given $b \in A_N$ with $|b_n - a_n| \le r \le R/2$ for $n = 1, \dotsc, N$, we construct a map
$G \colon \overline{\R_+^2} \to \overline{\R_+^2}$ as follows. For $n = 1, \dotsc, N$, if
$x \in B_R^+(a_n, 0)$, then we set $G(x) = x + (b_n - a_n, 0)$. Furthermore, we set $G(x) = x$
for all $x \in \Omega_{2R}(a)$. Then we extend $G$ to $\overline{\R_+^2}$ such that it becomes bijective and
a bi-Lipschitz map, and such that $G(x_1, 0) \in \R \times \{0\}$ for all $x_1 \in \R$. We set $g(x_1) = G(x_1, 0)$ for $x_1 \in (-1, 1)$. Now consider $\tilde{m}_\epsilon = m_\epsilon \circ g^{-1} \in M(b, d)$. Using a change of variable, we compute
\[
\int_{-1}^1 |\tilde{m}_\epsilon'|^2 \, dx_1 = \int_{-1}^1 \frac{|m_\epsilon'|^2}{|g'|} \, dx_1 \le \|1/g'\|_{L^\infty(-1, 1)} \int_{-1}^1 |m_\epsilon'|^2 \, dx_1.
\]
In order to compare the stray field energies of $m_\epsilon$ and $\tilde{m}_\epsilon$, we consider the
harmonic extension $v_\epsilon \in \dot{H}^1(\R_+^2)$ of $m_{\epsilon 1} - \cos \alpha$ and set
$\tilde{v}_\epsilon = v_\epsilon \circ G^{-1}$. Then
\[
\begin{split}
\int_{\R_+^2} |\nabla \tilde{v}_\epsilon|^2 \, dx & \le \int_{\R_+^2} |\nabla v_\epsilon|^2 |(DG)^{-1}|^2 |\det DG| \, dx \\
& \le \int_{\Omega_{2R}(a) \cup \bigcup_{n = 1}^N B_R^+(a_n, 0)} |\nabla v_\epsilon|^2 \, dx + \||(DG)^{-1}|^2 \det DG \|_{L^\infty(\R_+^2)} \int_{\Omega_R(a) \setminus \Omega_{2R}(a)} |\nabla v_\epsilon|^2 \, dx.
\end{split}
\]
Hence
\[
\begin{split}
\inf_{M(b, d)} E_\epsilon - \inf_{M(a, d)} E_\epsilon & \le E_\epsilon(\tilde{m}_\epsilon) - E_\epsilon(m_\epsilon) \\
& \le \frac\epsilon2 \left(\|1/g'\|_{L^\infty(-1, 1)} - 1\right) \int_{-1}^1 |m_\epsilon'|^2 \, dx_1 \\
& \quad + \left(\||(DG)^{-1}|^2 \det DG \|_{L^\infty(\R_+^2)} - 1\right) \int_{\Omega_R(a) \setminus \Omega_{2R}(a)} \frac{|\nabla v_\epsilon|^2}{2} \, dx.
\end{split}
\]
By Theorem \ref{thm:renormalised_confined}, we know that $m$ satisfies the assumptions of \cite[Theorem 28]{Ignat-Moser:16}
for a constant $C_0$ that is independent of $\epsilon$. Hence, by Remark 30 and (62) in \cite{Ignat-Moser:16}, together with Lemma \ref{lem:energy_in_annulus}, there exists another constant $C$, also independent of
$\epsilon$ (but depending on $R$), such that
\[
\epsilon \int_{-1}^1 |m_\epsilon'|^2 \, dx_1 + \int_{\Omega_R(a) \setminus \Omega_{2R}(a)} |\nabla v_\epsilon|^2 \, dx \le \frac{C}{(\log \delta)^2}.
\]
The quantities
\[
\|1/g'\|_{L^\infty(-1, 1)} - 1 \quad \text{and} \quad \||(DG)^{-1}|^2 \det DG \|_{L^\infty(\R_+^2)} - 1
\]
can be made arbitrarily small when $r$ is small enough. Thus we obtain the inequality
\[
\inf_{M(b, d)} E_\epsilon - \inf_{M(a, d)} E_\epsilon \le \frac{c_0}{(\log \delta)^2}
\]
for an arbitrarily small $ c_0> 0$.

 The reverse inequality is proved with essentially the same arguments, but we
exchange the roles of $a_n$ and $b_n$. This has the consequence that instead of working with
one minimiser for a given $\epsilon$, we have a family depending on the position of $b_n$.
We can check that all the relevant quantities in the resulting inequalities are
uniformly bounded, and we obtain the desired estimate.
\end{proof}

\begin{proof}[Proof of Theorem \ref{thm:compactness_and_separation_confined}]
The first statement of the theorem is an obvious consequence of 
Proposition~\ref{prop:compactness_confined}. Thus
only the second statement remains to be proved. Suppose, therefore, that
we have $\epsilon_k \searrow 0$ and $\phi_0 \in \Phi$ such that $\phi_{\epsilon_k} \to \phi_0$ in $L^2(-1, 1)$.
We may further assume that the convergence holds pointwise almost everywhere.
Since $\phi_0 \in \Phi$,
there exist $\tilde{N} \in \N \cup \{0\}$, $a \in A_{\tilde{N}}$, and $\omega \in ((2\pi \Z + \{0, \pm 2\alpha\}) \setminus \{0\})^{\tilde{N}}$ such that
\[
\phi_0' = \sum_{n = 1}^{\tilde{N}} \omega_n \dirac_{a_n}.
\]

Now suppose that
\begin{equation} \label{eqn:energy_bounded_from_above}
E_{\epsilon_k}(\phi_{\epsilon_k}) \le \frac{\eta(\phi_0)}{\log \frac{1}{\delta_k}} + \frac{C_0}{\left(\log \delta_k\right)^2}
\end{equation}
for all $k \in \N$, where $\delta_k = \epsilon_k \log \frac{1}{\epsilon_k}$.

Fix $n \in \{1, \dotsc, \tilde{N}\}$ for the moment. Then there exists $r \in (0, \frac{\rho(a)}2]$ such that
$\phi_{\epsilon_k}(a_n \pm r) \to \phi_0(a_n \pm r)$ as $k \to \infty$ and $\phi_0$ is
continuous at $a_n \pm r$.
We know that $\phi_0(a_n + r) - \phi_0(a_n - r) = \omega_n$ and $\phi_{\epsilon_k}$ is
continuous. If $\omega_n > 0$, then we choose
\[
L_n^- = \min\set{\ell \in \Z}{\pi \ell > \phi_0(a_n - r)}
\]
and
\[
L_n^+ = \max\set{\ell \in \Z}{\pi \ell < \phi_0(a_n + r)}.
\]
For sufficiently large value of $k$, we can choose
$b_{nL_n^-}^k, \dotsc, b_{nL_n^+}^k \in (a_n - r, a_n + r)$ such that
$\phi_{\epsilon_k}(b_{n\ell}^k) = \pi \ell$ for all $\ell = L_n^-, \dotsc, L_n^+$ and such that
$b_{nL_n^-}^k < \dotsb < b_{nL_n^+}^k$.
If $\omega_n < 0$, we choose $b_{n\ell}^k$ similarly.

Let $N = \iota(\phi_0)$ and let $a^k \in A_N$ comprise
all these points, i.e.,
\[
a^k = (b_{1L_1^-}^k, \dotsc, b_{1L_1^+}^k, \dotsc, b_{\tilde{N}L_{\tilde{N}}^-}^k, \dotsc, b_{\tilde{N}L_{\tilde{N}}^+}^k).
\]
Then $(\cos \phi_{\epsilon_k}, \sin \phi_{\epsilon_k}) \in M(a^k, d)$ for some $d \in \{\pm 1\}^N$,
which will satisfy $d_{n - 1} = - d_n$ or $|a_{n - 1}^k - a_n^k| \ge \rho(a)$ for
$n = 1, \dotsc, N$ whenever $k$ is sufficiently large. If $\theta_N < \alpha < \pi - \theta_N$, then Proposition \ref{prop:lower_energy_bound_confined} applies.
Thus we obtain a constant $C_1$ such that
\begin{equation} \label{eqn:energy_bounded_from_below}
E_{\epsilon_k}(\phi_{\epsilon_k}) \ge \frac{\pi}{2\log \frac{1}{\delta_k}} \sum_{n = 1}^N \gamma_n^2 + \frac{W(a^k, d) - C_1}{(\log \delta_k)^2}
\end{equation}
for every sufficiently large $k$, where, as usual, $\gamma_n = d_n - \cos \alpha$.
It is readily checked that
\[
\eta(\phi_0) = \frac{\pi}{2} \sum_{n = 1}^N \gamma_n^2
\]
(indeed, the function $\eta$ is defined with this identity in mind). Hence
\eqref{eqn:energy_bounded_from_above} and \eqref{eqn:energy_bounded_from_below} imply that
\[
\limsup_{k \to \infty} W(a^k, d) < \infty.
\]
According to Proposition \ref{prop:blow-up_of_renormalised_energy}, this means that
all the points of $a^k$ remain separated from one another when $k \to \infty$.
By construction, this is only possible when $N = \tilde{N}$ and $|\omega_n| = 2\alpha$ or $2\pi - 2\alpha$
for every $n = 1, \dotsc, N$.
Hence $\phi_0$ is simple.
\end{proof}

For the $\Gamma$-convergence result of Corollary \ref{cor:gamma_confined}, we have matching
lower and upper bounds of $E_\eps$ only in the case of a limiting magnetisation $m_0 = (\cos \phi_0, \sin \phi_0)$ for $\phi_0$ simple, i.e., when
all jumps come from individual domain walls of sign $\pm 1$.
This is a common feature in $\Gamma$-convergence results
for Ginzburg-Landau models where the vortex points carry winding numbers $\pm 1$.

\begin{proof}[Proof of Corollary \ref{cor:gamma_confined}]
For the lower bound, we consider a sequence $\epsilon_k \searrow 0$ and then write $m_k=(\cos \phi_{\eps_k}, \sin \phi_{\eps_k})$. Extracting
a subsequence if necessary, we may assume that $\phi_k\to \phi_0$ almost everywhere in $(-1,1)$.
 We represent $\phi'$ as in \eqref{phi_0}. Then $\omega \in \{\pm 2\alpha, \pm 2(\pi-\alpha)\}^N$,
since $\phi$ is simple. As in the proof of Theorem \ref{thm:compactness_and_separation_confined}, 
we construct the points $a^k\in A_N$ such that 
$|a_{n-1}^k-a_n^k|\geq \rho(a)$ and 
$m_k\in M(a^k, d)$ for $k$ sufficiently large.
(Note that for this construction there is no need of the angle restriction $\theta_N<\alpha<\pi-\theta_N$, which is imposed in Theorem \ref{thm:compactness_and_separation_confined} for a different reason.)
Since $\eta(\phi_0)=\frac\pi 2 \sum_{n=1}^N \gamma_n^2$ and $a^k\to a$ as $k\to \infty$, 
by Theorem \ref{thm:renormalised_confined} and Proposition \ref{prop:continuous_dependence_on_a}, we deduce the desired lower bound:
\[
\begin{split}
E_{\epsilon_k}(\phi_{\epsilon_k}) & \ge \inf_{M(a^k, d)} E_{\epsilon_k} \ge \inf_{M(a, d)} E_{\epsilon_k} - o\left(\frac{1}{(\log \delta_k)^2}\right) \\
& = \frac{\eta(\phi_0)}{ \log \frac{1}{\delta_k}} + \frac{\sum_{n = 1}^N e(d_n) + W(a, d)}{(\log \delta_k)^2} -  o\left(\frac{1}{(\log \delta_k)^2}\right).
\end{split}
\]

The upper bound follows from Theorem \ref{thm:renormalised_confined} as follows.
Given a simple $\phi_0$ with transition profile $(a, d)$, Theorem \ref{thm:renormalised_confined}
gives a family $(m_\epsilon)_{\epsilon > 0}$ in $M(a, d)$ such that
\[
E_\epsilon(m_\epsilon) \le 
\frac{\eta(\phi_0) }{\log \frac{1}{\delta}} + \frac{\sum_{n = 1}^N e(d_n) + W(a, d)}{\left(\log \frac{1}{\delta}\right)^2} + o\left(\frac{1}{\left(\log \frac{1}{\delta}\right)^2}\right).
\]
Next we note that $m_\epsilon$ can always be modified, without changing the energy, such that
between $a_n$ and $a_{n + 1}$ (for $n = 1, \dots, N - 1$), as well as between $-1$ and $a_1$ and
between $a_N$ and $1$, the sign of the second component $m_{\epsilon 2}$ is the same as the sign of $\sin \phi_0$.
To this end, we merely replace $m_\epsilon$ by $(m_{\epsilon 1}, \pm |m_{\epsilon 2}|)$, with the
sign chosen appropriately in each of these intervals. Thus we may assume that each $m_\epsilon$ has this property.
Now let $\phi_\epsilon$ denote the phase of $m_\epsilon$ (i.e., such that $m_\epsilon = (\cos \phi_\epsilon, \sin \phi_\epsilon)$)
with $\phi_\epsilon(-1) = \phi_0(-1)$. Then automatically
\[
\phi_\epsilon(a_n) = \frac{1}{2} \left(\lim_{x_1 \nearrow a_n} \phi_0(x_1) + \lim_{x_1 \searrow a_n} \phi_0(x_1)\right)
\]
for $n = 1, \dots, N$, and $\phi_\epsilon(1) = \phi_0(1)$.
Hence the only possible accumulation point for $(\phi_\epsilon)_{\epsilon > 0}$ in $\Phi$ is $\phi_0$.
The compactness of Theorem \ref{thm:compactness_and_separation_confined} then yields
 $\phi_\epsilon \to \phi_0$ in $L^2(-1, 1)$ as $\epsilon \searrow 0$.
\end{proof}

\begin{proof}[Proof of Corollary \ref{cor:prescribed_winding_number_confined}]
For the first statement, assume that the functions $\phi_\epsilon \in H^1(-1, 1)$
satisfy $\phi_\epsilon(-1) = \alpha$ and $\phi_\epsilon(1) = 2\pi \ell + \alpha$ for every $\epsilon > 0$,
where $N = 2\ell + 1$. 
It suffices to show that any subsequence $(\phi_{\epsilon_k})_{k \in \N}$ with $\epsilon_k \searrow 0$
contains another subsequence that satisfies the desired inequality. In order to keep the
notation simple, we suppress the subsequences in the following.
By continuity of $\phi_\eps$, we can choose $a^\eps\in A_N$ such that $m_\eps=(\cos \phi_\eps, \sin \phi_\eps)\in M(a^\eps, d_N^+)$. Since $\alpha \in (\theta_N, \pi-\theta_N)$, Proposition \ref{prop:lower_energy_bound_confined}
implies that for some $C>0$,
$$E_{\epsilon}(\phi_{\epsilon}) \geq \frac{{\mathcal E}_0}{ \log \frac{1}{\delta}} + \frac{W(a^\eps, d_N^+)-C}{\left(\log \frac{1}{\delta}\right)^2} \quad \textrm{for every } \eps.$$
If $\rho(a^\eps)\to 0$ as $\eps\to 0$ (for some subsequence),
then we use Proposition \ref{prop:blow-up_of_renormalised_energy}, which implies that $W(a^\eps, d_N^+)\to \infty$ as $\eps\to 0$.
This immediately gives the desired inequality.
Otherwise, the points of $a^\eps$ stay uniformly separated from one another and uniformly away from the boundary as
$\eps\to 0$. By Proposition \ref{prop:continuous_dependence_on_a} and a compactness
argument, we find $a \in A_N$ such that $a_\eps\to a$ (for a subsequence) and $E_\epsilon(\phi_\epsilon) \ge \inf_{M(a, d_N^+)} E_\epsilon - o(1/(\log \delta)^2)$
as $\epsilon \searrow 0$. Theorem \ref{thm:renormalised_confined} now gives the conclusion.

For the proof of statement \ref{item:limsup_topological}, assume that
\be
\label{assum12}
E_\epsilon(\phi_\epsilon) \le \frac{\mathcal{E}_0}{\log \frac{1}{\delta}} + O\left(\frac{1}{(\log \delta)^2}\right).
\ee
By Theorem \ref{thm:compactness_and_separation_confined}, there exist a sequence $\epsilon_k \searrow 0$ and a limit $\phi_0 \in \Phi$ such that
$\phi_{\epsilon_k} \to \phi_0$ as $k \to \infty$ in $L^2(-1, 1)$. We may further assume that we have pointwise convergence
almost everywhere. We claim that $\phi_0$ is simple (i.e., has jumps of size $2\alpha$ and $2(\pi - \alpha)$ only),
as required for the statement. In order to show this, 
consider a jump of $\phi_0$, of any size, at a point $b \in (-1, 1)$. Then there exist $\psi_-, \psi_+ \in 2\pi\Z \pm \alpha$
and $r > 0$ such that $\phi_0 = \psi_-$ in $(b - r, b)$ and $\phi_0 = \psi_+$ in $(b, b + r)$. Furthermore, we may choose
$r$ such that $\phi_{\epsilon_k}(b - r) \to \psi_-$ and $\phi_{\epsilon_k}(b + r) \to \psi_+$ as $k \to \infty$.
If, say, $\psi_+ - \psi_- = 2\pi j$ for some $j \in \Z \setminus \{0\}$, then by the continuity of $\phi_{\epsilon_k}$,
the set
\[
\set{x_1 \in (b - r, b + r)}{\phi_{\epsilon_k}(x_1) \in \pi \Z}
\]
has at least $2|j|$ points whenever $k$ is sufficiently large. We may select $2|j|$ of them,
say $\{t_1^{(k)}, \dotsc, t_{2|j|}^{(k)}\}$,
such that $\cos \phi_{\epsilon_k}(t_i^{(k)}) = (-1)^i$ for $i = 1, \dotsc, 2|j|$ or
$\cos \phi_{\epsilon_k}(t_i^{(k)}) = (-1)^{i + 1}$ for $i = 1, \dotsc, 2|j|$.
Similar statements hold if $\psi_+ - \psi_- = 2\pi j \pm 2\alpha$ (but now we have
$2|j|\pm1$ points). 
Near the boundary, the function $\phi_0$ is constant, too. More precisely, there exists $r > 0$ such that
$\phi_0 = \chi_-$ in $(-1, r - 1)$ and $\phi_0 = \chi_+$ in $(1 - r, 1)$ for two numbers $\chi_-, \chi_+ \in 2\pi \Z \pm \alpha$.
If $\chi_- \neq -\alpha$, say $\chi_- = 2\pi j - \alpha$ for $j \in \Z \setminus \{0\}$, then the set
\[
\set{x_1 \in (-1, r - 1)}{\phi_{\epsilon_k}(x_1) \in \pi \Z}
\]
has at least $2|j|$ points for $k$ large enough. (We may think of this situation as a jump at the boundary.) 
Again we may select $2|j|$ of them
such that the sign of $\cos \phi_{\epsilon_k}$ oscillates between $\pm 1$.
Similar statements hold if $\chi_- = 2\pi j + \alpha$ and for the other boundary point.
The prescribed boundary conditions entail that the number of points of $\phi_{\epsilon_k}^{-1}(\pi \Z)$ covered by the
above discussion is at least $N$. 

Suppose first that $\ell \ge 1$ (and thus $N \ge 3$). In this case,
if $\phi_0$ were \emph{not} simple or did \emph{not} match the given boundary data, then we could construct $a^{(k)} \in A_N$
 (from points chosen among the above $t_i^{(k)}$ for all the jumps, including jumps at the boundary) and $\sigma > 0$
such that for $n = 1, \dotsc, N - 1$, either $a_{n + 1}^{(k)} - a_n^{(k)} \ge \sigma$ or
$\cos \phi_{\epsilon_k}(a_{n + 1}^{(k)}) = - \cos \phi_{\epsilon_k}(a_n^{(k)})$,
and such that there are exactly $\ell + 1$ positive and $\ell$ negative signs, but
$\rho(a^{(k)}) \to 0$ as $k \to \infty$. 
Proposition \ref{prop:lower_energy_bound_confined} and Proposition \ref{prop:blow-up_of_renormalised_energy}
would then give an estimate for the energy
incompatible with the assumption \eqref{assum12}. Therefore, $\phi_0$ is simple and there is no
jump at the boundary, which means that $\phi_0(-1)=-\alpha$, $\phi_0(1)=2\pi\ell+\alpha$,
and there must be at least $N$ jumps (at least $\ell + 1$ 
of which are of the size $2\alpha$ and at least $\ell$ of the size $2(\pi - \alpha)$).
In particular, we conclude that $\eta(\phi_0)\geq {\mathcal E}_0$.
But Corollary \ref{cor:gamma_confined}, together with \eqref{assum12}, implies that ${\mathcal E}_0\geq \eta(\phi_0)$,
so we have equality. Therefore, $\phi_0$ has exactly $N$ jumps. It also follows that
$\phi_0$ is of the form as described in the statement.

If $\ell = 0$ and $N = 1$, then we take advantage of the fact that $\theta_2 = \theta_1 = 0$.
In this case, if $\phi_0$ did \emph{not} match the given boundary data,
we would be able to construct $a^{(k)} \in A_1$ with
properties as above. If $\phi_0$ were \emph{not} simple, in order to achieve that $\rho(a^{(k)})\to 0$, we could construct $a^{(k)} \in A_2$ instead. We would then find a contradiction with the same arguments.

Statement \ref{item:minimisers_topological} is a standard consequence of the $\Gamma$-convergence result in 
Corollary~\ref{cor:gamma_confined}, which means that minimisers of 
$|\log \delta|(|\log \delta| E_\eps-\mathcal{E}_0)-\mathcal{E}_1$ converge to 
minimisers of $W(\cdot, d^+_N)$. Indeed, if $\tilde a$ is any point in $A_N$,
consider the (unique) simple function $\tilde \phi_0\in \Phi$ with the jump points $\tilde a$
 and the structure described in statement 2, and satisfying the boundary conditions
$\tilde \phi_0(-1)=-\alpha$, $\tilde \phi_0(1)=2\pi \ell+\alpha$.
By Corollary~\ref{cor:gamma_confined}, there exists a family $(\tilde \phi_\eps)_{\eps > 0}$ with $\tilde \phi_\eps(-1)=-\alpha$, $\tilde \phi_\eps(1)=2\pi \ell+\alpha$, such that
$$E_\eps(\phi_\eps)\leq E_\eps(\tilde \phi_\eps)\leq \frac{\mathcal{E}_0}{\log \frac{1}{\delta}} + \frac{\mathcal{E}_1 + W(\tilde{a}, d_N^+)}{\left(\log \frac{1}{\delta}\right)^2} + o\left(\frac{1}{\left(\log \frac{1}{\delta}\right)^2}\right)$$
as $\epsilon \searrow 0$, since $\phi_\eps$ are minimisers of $E_\eps$ for  their boundary data. Then statement \ref{item:limsup_topological} applies to the given family of minimisers 
$(\phi_\eps)_{\eps > 0}$, so for a subsequence,  we have the convergence
$\phi_{\eps_k} \to \phi_0$ for a simple function $\phi_0$
with
jump points $a\in A_N$ as in statement \ref{item:limsup_topological}. Then the lower bound in Corollary~\ref{cor:gamma_confined} combined with the above upper bound for $E_\eps(\phi_\eps)$, yield $W(a,d_N^+)\leq W(\tilde{a}, d_N^+)$ 
in 
the limit $\eps\searrow 0$. That is, $a$ is a minimizer of $W(\cdot, d^+_N)$ over $A_N$.
\end{proof}

\section{Analysis for the unconfined problem} \label{sect:unconfined}

This section analyses the problem described in Subsection \ref{subsect:unconfined}.
We first derive some properties of
the function $I$ appearing in the renormalised energy, then we derive and analyse the limiting stray field potential
for this problem, and finally we explain how the arguments for the confined case
\cite{Ignat-Moser:16} need to be adapted for the proof of Theorem \ref{thm:renormalised_unconfined}. We also
  prove a $\Gamma$-convergence result in Theorem~\ref{thm:unconfined_gamma_conv} below,
which adapts the statements of Theorem \ref{thm:compactness_and_separation_confined} and Corollary~\ref{cor:gamma_confined} to the unconfined problem.

\subsection{The function $I$}

Here we prove a few statements about the function
\[
I(t) = \int_0^\infty \frac{s e^{-s}}{s^2 + t^2} \, ds, \quad t>0,
\]
defined in the introduction. First we have an alternative representation.

\begin{lemma} \label{lem:alternative_rep_of_I}
For any $t > 0$,
\[
I(t) = \int_0^\infty \frac{\cos s}{s + t} \, ds.
\]
\end{lemma}

\begin{remark}
The integral in Lemma \ref{lem:alternative_rep_of_I} does not converge in the $L^1$-sense, but
the Leibniz criterion for alternating series implies that it converges as an
improper Riemann integral. We always use this interpretation for oscillating integrals
of this type.
\end{remark}

\begin{proof}
Let $t, R > 0$. Note that the functions $z \mapsto \frac{e^{\pm iz}}{z + t}$ are holomorphic in
$\set{z \in \C}{\Re z > -t}$. Hence using contour integrals along the boundaries of
the quarter disks
\[
\set{z \in \C}{\Re z > 0,\  \pm \Im z > 0, \ |z| < R},
\]
we find that
\[
0 = \int_0^R \frac{e^{is}}{s + t} \, ds - i \int_0^R \frac{e^{-s}}{is + t} \, ds + iR \int_0^{\pi/2} e^{i\phi} \frac{\exp(iR e^{i\phi})}{Re^{i\phi} + t} \, d\phi
\]
and
\[
0 = \int_0^R \frac{e^{-is}}{s + t} \, ds + i \int_0^R \frac{e^{-s}}{-is + t} \, ds - iR \int_0^{\pi/2} e^{-i\phi} \frac{\exp(-iR e^{-i\phi})}{Re^{-i\phi} + t} \, d\phi.
\]
Since
\[
|\exp(iR e^{i\phi})| = e^{-R\sin \phi} \quad \text{and} \quad |Re^{i\phi} + t| \ge R
\]
for $\phi \in [0, \pi/2]$, and since there exists a constant $c > 0$ such that $c\phi\leq \sin \phi$ for every $\phi\in [0, \frac\pi2]$, we deduce that
\[
\left|iR\int_0^{\pi/2} e^{i\phi} \frac{\exp(iR e^{i\phi})}{Re^{i\phi} + t} \, d\phi\right| \le  \int_0^{\pi/2} e^{-cR\phi} \, d\phi \le \frac{1}{cR}.
\]
Similarly,
\[
\left|iR \int_0^{\pi/2} e^{-i\phi} \frac{\exp(-iR e^{-i\phi})}{Re^{-i\phi} + t} \, d\phi\right| \le \frac{1}{cR}.
\]
Hence, by adding the above equalities, we obtain
\begin{align*}
\int_0^\infty \frac{\cos s}{s + t} \, ds & = \frac{1}{2} \lim_{R \to \infty} \int_0^R \frac{e^{is} + e^{-is}}{s + t} \, ds 
= \frac{i}{2} \lim_{R \to \infty} \int_0^R e^{-s} \left(\frac{1}{t + is} - \frac{1}{t - is}\right)\, ds 
 = \int_0^\infty \frac{s e^{-s}}{s^2 + t^2} \, ds.
\end{align*}
The last integral is of course identical to $I(t)$.
\end{proof}

The next lemma shows that $I(t)$ is logarithmic to leading order when $t \searrow 0$ and decays quadratically as $t\to \infty$.
We will see later that $I$ can also be used to describe the tail profile of a N\'eel wall
(see \eqref{v_x1} below). In this respect, our decay estimate is  consistent
with previous estimates for the decay of a tail of a N\'eel wall in \cite{Chermisi-Muratov:13} and \cite[Theorem 5.2]{Ignat-Moser:17}. 

\begin{lemma} \label{lem:I_is_logarithmic}
The function $I$ is positive, decreasing, and convex with $I(t) \le 1/t^2$ for all $t > 0$.
If $I_0$ is the number defined in \eqref{eqn:Euler-Mascheroni}, then
 \[
0 \le I(t) + \log t - I_0 \le \frac{\pi t}{2}
\]
for all $t > 0$ as well. In particular, $I(t)=\log \frac1t+I_0+o(1)$ as $t\searrow 0$. 
\end{lemma}

\begin{proof}
By a change of variable, we 
write $I(t)=\int_0^\infty \frac{s e^{-st}}{s^2 + 1} \, ds>0$ for $t>0$, and then we compute
\be
\label{zez}
I'(t)=-\int_0^\infty \frac{s^2 e^{-st}}{s^2 + 1} \, ds<0 \quad \textrm{and}\quad 
I''(t)=\int_0^\infty \frac{s^3 e^{-st}}{s^2 + 1} \, ds>0.
\ee
Hence $I$ is decreasing and convex. Moreover, using 
$I(t) = \int_0^\infty \frac{se^{-s}}{s^2 + t^2} \, ds$, we see that
$$I(t)\le \frac{1}{t^2} \int_0^\infty s e^{-s} \, ds = \frac{1}{t^2}, \quad t>0.$$
Furthermore, integration by parts gives
\[
I(t) = -\log t + \int_0^\infty e^{-s} \log \sqrt{s^2 + t^2} \, ds.
\]
Note that
$$
\frac{d}{dt} \int_0^\infty e^{-s} \log \sqrt{s^2 + t^2} \, ds = \int_0^\infty \frac{te^{-s}}{s^2 + t^2} \, ds = \int_0^\infty \frac{e^{-st}}{s^2 + 1} \, ds$$
and 
$$0\leq \int_0^\infty \frac{e^{-st}}{s^2 + 1} \, ds\leq \int_0^\infty \frac{ds}{s^2 + 1} = \frac{\pi}{2}$$ for every $t>0$. 
Thus,
\[
I_0 = \left.\int_0^\infty e^{-s} \log \sqrt{s^2 + t^2} \, ds\right|_{t=0} \leq  \int_0^\infty e^{-s} \log \sqrt{s^2 + t^2} \, ds\leq I_0+\frac{\pi t}{2}, \quad t>0,
\]
which leads to the desired conclusion.
\end{proof}

\subsection{The limiting stray field potential: construction} \label{sect:limiting_stray_field}

In this section we will redefine and compute the function $u_{a, d}^* \colon \R_+^2 \to \R$, the limiting stray
field potential for N\'eel walls of sign $d_1, \dotsc, d_N$ at the points $a_1, \dotsc, a_N$,
for the problem with anisotropy term. Simultaneously, we will obtain a limiting profile $\mu_{a, d}^* \colon \R \to \R$
for the tails of the N\'eel walls. The two functions are determined through the boundary value problem\footnote{Note that the limiting stray-field potential $u$ corresponding to a $180^\circ$ N\'eel wall in the confined case (defined on page \pageref{confined_limiting_stray}) satisfies 
$\dd{u}{x_1}=-\pi  \dirac_0$ in $(-1,1)\times \{0\}$. In the unconfined case, an additional term $\mu_{a, d}^*$ appears in the equation for
$\dd{u_{a, d}^*}{x_1}$, due to the presence of anisotropy. In fact, this term makes the limiting stray-field energy finite in $\R^2_+\setminus B_1(0)$, see Proposition \ref{prop:conjugate_harmonic_functions}.}
\begin{alignat*}{2}
\Delta u_{a, d}^* & = 0 & \quad & \text{in $\R_+^2$}, \\
\dd{u_{a, d}^*}{x_1} & = \mu_{a, d}^* - \pi \sum_{n = 1}^N \gamma_n \dirac_{a_n} && \text{on $\R \times \{0\}$}, \\
\dd{u_{a, d}^*}{x_2} & = - (\mu_{a, d}^*)' && \text{on $\R \times \{0\}$}.
\end{alignat*}

In order to see that this is a well-posed boundary value problem (and
also in order to compute the solution later), we consider the harmonic conjugate
of $u_{a, d}^*$, i.e., the function $v_{a, d}^* \colon \R_+^2 \to \R$, unique
up to a constant, that satisfies $\nabla v_{a, d}^* = \nabla^\perp u_{a, d}^*$.
The second boundary condition for $u_{a, d}^*$ implies that $v_{a, d}^*(\blank, 0) - \mu_{a, d}^*$
is constant. Thus we may assume without loss of generality that
$v_{a, d}^*(\blank, 0) = \mu_{a, d}^*$ in $\R$; this will then also determine $v_{a, d}^*$ completely.
We finally obtain the following boundary value problem for the conjugate harmonic function:
\begin{alignat*}{2}
\Delta v_{a, d}^* & = 0 & \quad & \text{in $\R_+^2$}, \\
\dd{v_{a, d}^*}{x_2} & = v_{a, d}^* - \pi \sum_{n = 1}^N \gamma_n \dirac_{a_n} && \text{on $\R \times \{0\}$}.
\end{alignat*}
That is, we have a harmonic function satisfying a Robin type boundary condition here.
We now give some arguments depending in part on formal calculations, but they will be justified later.

As in our previous paper \cite{Ignat-Moser:16}, we can solve this problem by
superimposing the solutions of simpler problems. We therefore consider the following:
\begin{alignat}{2} 
\Delta v & = 0 & \quad & \text{in $\R_+^2$}, \label{eqn:equation_for_v_unconfined} \\
\dd{v}{x_2} & = v - \pi \pmb{\delta}_0 && \text{on $\R \times \{0\}$}. \label{eqn:boundary_condition_for_v_unconfined}
\end{alignat}
Let $\hat{v}$ denote the Fourier transform of $v$ with respect to $x_1$, i.e.,
\[
\hat{v}(\xi, x_2) = \int_{-\infty}^\infty e^{-i\xi x_1} v(x_1, x_2) \, dx_1, \quad \xi \in \R, \ x_2 \ge 0.
\]
Then we find that
\begin{alignat*}{2}
\frac{\partial^2 \hat{v}}{\partial x_2^2} - \xi^2 \hat{v} & = 0 & \quad & \text{for $x_2 > 0$}, \\
\dd{\hat{v}}{x_2} & = \hat{v} - \pi && \text{for $x_2 = 0$}.
\end{alignat*}
We want a function $v$ with finite Dirichlet energy at $|x|\to \infty$; thus we rule out solutions with exponential growth
as $x_2 \to \infty$ and find that
\begin{equation} \label{eqn:Fourier_rep_v}
\hat{v}(\xi, x_2) = \frac{\pi e^{-|\xi| x_2}}{1 + |\xi|}.
\end{equation}

Before justifying these formal calculations, we derive some properties of the function
$v$ implicitly defined here. From here on, the arguments are fully rigorous again.
As $\hat v(\blank, 0)\in L^2(\R)$, we deduce that $v(\blank, 0)\in L^2(\R)$ and 
\be
\label{planch_v^2}
\int_{\R} v^2\, dx_1=\frac1{2\pi}\int_{\R} \hat v^2\, d\xi=\frac\pi 2 \int_{\R} \frac{d\xi}{(1+|\xi|)^2}=
\pi \int_{0}^\infty \frac{d\xi}{(1+\xi)^2}=\pi.
\ee
Applying the inverse Fourier transform for every fixed $x_2>0$, we obtain
\[
v(x_1, x_2) = \frac{1}{2} \int_{-\infty}^\infty \frac{e^{i\xi x_1 - |\xi|x_2}}{1 + |\xi|} \, d\xi.
\]

In order to find a more convenient representation, we use a contour integral in $\C$.
Consider the contour consisting of the intervals $[0, R]$ and $[0, iR]$ (the second with reverse orientation)
and the quarter circle $\Gamma_R$ parametrised by $g(t) = Re^{it}$ for $0 \le t \le \frac{\pi}{2}$.
Suppose first that $x_1 > 0$ and fix $x_2 > 0$ as well. The function $z \mapsto e^{izx_1 - zx_2}/(1 + z)$
is holomorphic away from $z = -1$. Hence
$$
\int_0^R \frac{e^{i\xi x_1 - |\xi|x_2}}{1 + |\xi|} \, d\xi  = \int_0^R \frac{e^{iz x_1 - z x_2}}{1 + z} \, dz 
 = \int_0^{iR} \frac{e^{iz x_1 - z x_2}}{1 + z} \, dz - \int_{\Gamma_R} \frac{e^{iz x_1 - z x_2}}{1 + z} \, dz.
$$
We observe that
\[
\begin{split}
\left|\int_{\Gamma_R} \frac{e^{iz x_1 - z x_2}}{1 + z} \, dz\right| & = R\left|\int_0^{\pi/2} \frac{\exp(R(ix_1 - x_2)(\cos t + i \sin t))}{1 + Re^{it}} e^{it} \, dt\right| \\
& \le R \int_0^{\pi/2} \exp(R(-x_2 \cos t - x_1 \sin t)) \, dt.
\end{split}
\]
Since $x_2 \cos t + x_1 \sin t \ge \min\{x_1, x_2\} > 0$ for $t \in [0, \frac{\pi}{2}]$, we conclude that
\[
\lim_{R \to \infty} \int_{\Gamma_R} \frac{e^{iz x_1 - z x_2}}{1 + z} \, dz = 0
\]
and
\[
\int_0^\infty \frac{e^{i\xi x_1 - |\xi|x_2}}{1 + |\xi|} \, d\xi = i\int_0^\infty \frac{e^{-tx_1 - it x_2}}{1 + it} \, dt.
\]
Similarly, but integrating over the boundary of
\[
\set{z \in \C}{\Re z > 0, \ \Im z < 0, \ |z| < R},
\]
we see that
\[
\int_{-\infty}^0 \frac{e^{i\xi x_1 - |\xi|x_2}}{1 + |\xi|} \, d\xi = -i\int_0^\infty \frac{e^{-t x_1 + it x_2}}{1 - it} \, dt
\]
and thus
\[
v(x_1, x_2) = \frac{i}{2} \int_0^\infty e^{-tx_1} \left(\frac{e^{-itx_2}}{1 + it} - \frac{e^{itx_2}}{1 - it}\right) \, dt.
\]
We further compute
\[
\frac{e^{-itx_2}}{1 + it} - \frac{e^{itx_2}}{1 - it} = \frac{e^{-itx_2} - e^{itx_2}}{1 + t^2} - it \frac{e^{-itx_2} + e^{itx_2}}{1 + t^2} = -2i \frac{\sin(tx_2) + t\cos(tx_2)}{1 + t^2}.
\]
Hence
\[
v(x_1, x_2) = \int_0^\infty e^{-tx_1} \frac{\sin(tx_2) + t\cos(tx_2)}{1 + t^2} \, dt
\]
for $x_1 > 0$. Similar computations for $x_1 < 0$ show that
\begin{equation} \label{eqn:integral_rep_v}
v(x_1, x_2) = \int_0^\infty e^{-t|x_1|} \frac{\sin(tx_2) + t\cos(tx_2)}{1 + t^2} \, dt
\end{equation}
for all $x_1 \not= 0$ and $x_2 > 0$.
(This is consistent with the expectation that $v$ is an even function in $x_1$,
which comes from the symmetry of the boundary value problem \eqref{eqn:equation_for_v_unconfined}, \eqref{eqn:boundary_condition_for_v_unconfined}.)
Furthermore, by the dominated convergence theorem, there is a continuous
extension to $x_2 = 0$ (as long as $x_1 \neq 0$), given by the obvious integral.
We then check that
\begin{equation} \label{eqn:v_boundary}
\dd{v}{x_2}(x_1, 0) = v(x_1, 0) \quad \textrm{ for every $x_1 \neq 0$}.
\end{equation}

We now want to find the corresponding conjugate harmonic function $u \colon \R_+^2 \to \R$. The condition $\nabla v = \nabla^\perp u$ suggests that\footnote{\label{foot11}Thus $u$ is harmonic in $\R^2_+$ and 
$\dd{u}{x_1}=v-\pi\dirac_{0}$ in $\R\times \{0\}$, so that the Fourier transform of $u$ in $x_1$ 
is given by $\hat u(\xi, x_2)=\frac{i\pi \xi}{|\xi|(1+|\xi|)}e^{-|\xi| x_2}$ for $\xi\neq 0$ and $x_2>0$.}
there exist two constants $c_1, c_2 \in \R$ such that
\[
u(x_1, x_2) = - \int_0^\infty e^{-tx_1} \frac{\cos(tx_2) - t\sin(tx_2)}{1 + t^2} \, dt + c_1
\]
for $x_1 > 0$ and
\[
u(x_1, x_2) = \int_0^\infty e^{tx_1} \frac{\cos(tx_2) - t\sin(tx_2)}{1 + t^2} \, dt + c_2
\]
for $x_1 < 0$. We expect, however, that $u$ has a limit as $|x| \to \infty$, and we may set this limit $0$. Letting $x_1 \to \infty$,
we see that this would imply that $c_1 = c_2 = 0$. Thus
\begin{equation} \label{eqn:integral_rep_u}
u(x_1, x_2) = -\frac{x_1}{|x_1|} \int_0^\infty e^{-t|x_1|} \frac{\cos(tx_2) - t\sin(tx_2)}{1 + t^2} \, dt.
\end{equation}
In particular, $u$ is odd in $x_1$ and $u(0,x_2)=0$ for every $x_2>0$. Here again, we have a continuous extension to $x_2 = 0$ when $x_1 \neq 0$.

We now justify these formal computations. Moreover, we prove in Propositions~\ref{prop:conjugate_harmonic_functions} and \ref{prop:pointwise_estimates}  below that $u(x)=-\arctan \frac{x_1}{x_2}+o(1)$ and $v(x)=\log \frac1{|x|}+I_0+o(1)$ as $x\to 0$ and $u,v\to 0$ as $|x|\to \infty$.

\begin{proposition} \label{prop:conjugate_harmonic_functions}
Formulas \eqref{eqn:integral_rep_v} and \eqref{eqn:integral_rep_u} define a pair of conjugate harmonic functions $u, v \colon \R_+^2 \to \R$. There exists a universal constant $C$ with
\[
\pi \log \frac{1}{r} - C \le \int_{\R_+^2 \setminus B_r(0)} |\nabla v|^2 \, dx \le \pi \log \frac{1}{r} + C
\]
for all $r \in (0, 1]$. Moreover, $|u|\leq \frac\pi 2$ in $\R^2_+$ and  $u, v\to 0$ as $|x|\to \infty$.
\end{proposition}

\begin{proof}
\setcounter{stepno}{0}
\begin{subproof} {Step \step:  harmonicity and limit of $v$}
We have seen that for every fixed $x_2 > 0$, the Fourier transform of $v$ with respect to $x_1$ is given by
\eqref{eqn:Fourier_rep_v}.
It is thus clear that the function $x_1 \mapsto v(x_1, x_2)$ belongs to $H^k(\R)$ for every $k \in \N$
 for all $x_2 > 0$. Moreover, we see that $v \in L_\loc^2((0, \infty), H^k(\R))$.
Integrating against a test function and using Plancherel's theorem, we further see that the distributional
derivative $\frac{\partial^2 v}{\partial x_2^2}$ satisfies
\[
\widehat{\frac{\partial^2 v}{\partial x_2^2}} = \frac{\pi \xi^2 e^{-|\xi| x_2}}{1 + |\xi|} = \xi^2 \hat{v} = - \widehat{\frac{\partial^2 v}{\partial x_1^2}}.
\]
Hence $v$ is harmonic in $\R_+^2$.

Furthermore, we see that the norms $\|v(\blank, x_2)\|_{H^k(\R)}$ tend to $0$ as $x_2 \to \infty$. In particular, by the embedding $H^1(\R)\subset L^\infty(\R)$, we deduce that 
$v(x_1, x_2) \to 0$ as $x_2 \to \infty$ uniformly in $x_1$.
With the formula \eqref{eqn:integral_rep_v}, it is easy to see that
$v(x_1, x_2) \to 0$ as $x_1 \to \pm \infty$ uniformly in $x_2$.
Hence $\lim_{|x| \to \infty} v(x) = 0$.
\end{subproof}

\begin{subproof}{Step \step: estimate of $\int_{\R_+^2 \setminus B_r(0)} |\nabla v|^2 \, dx$}
We consider the Dirichlet energy of $v$ first
in $\R\times (s, +\infty)$ and then separately in the infinite strips $(s, +\infty)\times (0,2s)$ for $s\in (0,1]$.

Fix $s \in (0, 1]$. Then in $\R\times (s, +\infty)$, we compute
\[
\begin{split}
\lefteqn{\int_s^\infty \int_{-\infty}^\infty \left(\xi^2 (\hat{v}(\xi, x_2))^2 + \left(\dd{\hat{v}}{x_2}(\xi, x_2)\right)^2\right) \, d\xi \, dx_2
 = 2\pi^2 \int_s^\infty \int_{-\infty}^\infty \frac{\xi^2 e^{-2|\xi| x_2}}{(1 + |\xi|)^2} \, d\xi \, dx_2}
  \qquad\qquad \\
& = 2\pi^2 \int_{-\infty}^\infty \frac{\xi^2}{(1 + |\xi|)^2} \int_s^\infty  e^{-2|\xi| x_2} \, dx_2 \, d\xi 
 = \pi^2 \int_{-\infty}^\infty \frac{|\xi| e^{-2|\xi|s}}{(1 + |\xi|)^2} \, d\xi \\
 & = 2\pi^2 \int_0^\infty \frac{\xi e^{-2\xi s}}{(1 + \xi)^2} \, d\xi 
= 2\pi^2 \int_0^\infty \frac{t e^{-2t}}{(s + t)^2} \, dt.
\end{split}
\]
We note that
\begin{equation} \label{eqn:formula_for_integration_by_parts}
\frac{t}{(s + t)^2} = \frac{1}{s + t} - \frac{s}{(s + t)^2} = \frac{d}{dt} \left(\log(s + t) + \frac{s}{s + t}\right).
\end{equation}
Hence an integration by parts gives
\[
\int_0^\infty \frac{t e^{-2t}}{(s + t)^2} \, dt = \log \frac{1}{s} - 1 + 2\int_0^\infty e^{-2t} \left(\log (s + t) + \frac{s}{s + t}\right) \, dt.
\]
Finally, the concavity of the logarithm yields the inequality $\log(s + t) \le \log t + \frac{s}{t}$,  and thus,
\[
\begin{split}
\int_0^\infty e^{-2t} \log (s + t) \, dt & \le \int_0^1 e^{-2t} \log 2 \, dt + \int_1^\infty e^{-2t} \left(\log t + \frac{s}{t}\right) \, dt \\
& \le \log 2 + \int_1^\infty e^{-2t} \log t \, dt + s \int_1^\infty \frac{e^{-2t}}{t} \, dt
\end{split}
\]
for $s \le 1$.
Noting that the right hand side is bounded by a constant, we conclude that there exists a number $C_1$ satisfying
\[
\pi \log \frac{1}{s} - C_1 \le \int_s^\infty \int_{-\infty}^\infty |\nabla v|^2 \, dx_1 \, dx_2 \le \pi \log \frac{1}{s} + C_1 \quad \textrm{for all $s \le 1$.}
\]

Next we consider the strip $(s, +\infty)\times (0,2s)$. 
For $x_1 > 0$, we find that
\[
\begin{split}
\left|\dd{v}{x_1}(x_1, x_2)\right| & = \left| \int_0^\infty te^{-tx_1} \frac{\sin(tx_2) + t\cos(tx_2)}{1 + t^2} \, dt\right| \\
& \le \int_0^\infty \frac{t + t^2}{1 + t^2} e^{-tx_1} \, dt  = \frac{1}{x_1} \int_0^\infty \frac{sx_1 + s^2}{x_1^2 + s^2} e^{-s} \, ds.
\end{split}
\]
Young's inequality gives
\[
s x_1 + s^2 \le \frac{1}{2(\sqrt{2} - 1)}x_1^2 + \left(1 + \frac{\sqrt{2} - 1}{2}\right)s^2 = \frac{\sqrt{2} + 1}{2}(x_1^2 + s^2).
\]
Hence
\[
\left|\dd{v}{x_1}(x_1, x_2)\right| \le \frac{\sqrt{2} + 1}{2x_1}.
\]
Similarly, we find that
\[
\left|\dd{v}{x_2}(x_1, x_2)\right| \le \frac{\sqrt{2} + 1}{2x_1}
\]
when $x_1 > 0$. We therefore compute
\[
\int_0^{2s} \int_s^\infty |\nabla v(x_1, x_2)|^2 \, dx_1 \, dx_2 \le \left(\frac{3}{2} + \sqrt{2}\right) \int_0^{2s} \int_s^\infty \frac{dx_1}{x_1^2} \, dx_2 = 3 + 2\sqrt{2}.
\]
Similar estimates hold for $x_1 < 0$.

Finally, we can conclude {Step \thestepno} and estimate the energy outside the disk $B_r(0)$ as follows.
Given $r \in (0, 1]$, we now apply the above inequalities for $s = r$ and for $s = r/2$ and we conclude that
\begin{equation} \label{eqn:energy_upper_estimate}
\pi \log \frac{1}{r} - C_1 \le \int_{\R_+^2 \setminus B_r(0)} |\nabla v|^2 \, dx \le \pi \log \frac{2}{r} + C_1 + 6 + 4\sqrt{2}.
\end{equation}
This proves the inequalities for the Dirichlet energy.
\end{subproof}

\begin{subproof}{Step \step:  $u$ and $v$ are conjugate harmonic functions}  
Since $\curl \nabla^\perp v = 0$, there exists a function $\tilde{u} \colon \R_+^2 \to \R$ with
$\nabla v = \nabla^\perp \tilde{u}$. This function satisfies
\begin{alignat*}{2}
\Delta \tilde{u} & = 0 & \quad & \text{in $\R_+^2$}, \\
\dd{\tilde{u}}{x_1} & = v && \text{on $(-\infty, 0)$ and $(0, \infty)$},
\end{alignat*}
owing to \eqref{eqn:v_boundary}. Moreover, we know that
\[
\int_{\R_+^2 \setminus B_r(0)} |\nabla \tilde{u}|^2 \, dx < \infty
\]
for any $r > 0$. Now define $w(x) = \tilde{u}(x/|x|^2)$. Then $w$ is again harmonic in $\R_+^2$
and belongs to $H^1(B_R^+(0))$ for any $R > 0$. We compute, for $x_1\neq 0$:
\[
\dd{w}{x_1}(x_1, 0) = -\frac{1}{x_1^2} \dd{\tilde{u}}{x_1}\left(\frac{1}{x_1}, 0\right) = -\frac{1}{x_1^2} v\left(\frac{1}{x_1}, 0\right).
\]
Note that by Lemma \ref{lem:I_is_logarithmic},
\be
\label{v_x1}
v(x_1, 0) = \int_0^\infty \frac{t e^{-t|x_1|}}{1 + t^2} \, dt = \int_0^\infty \frac{se^{-s}}{x_1^2 + s^2} \, ds =I(|x_1|)\le \frac{1}{x_1^2}.
\ee
Thus $\dd{w}{x_1}(x_1, 0)$ is bounded near $0$.
The fact that $w \in H^1(B_1^+(0))$ prevents jumps of $w(\blank, 0)$.
Hence, $w(\blank, 0)$ is Lipschitz
continuous near the origin. It follows that $w$ (as a harmonic function in $\R^2_+$) is continuous at $x=0$ (see \cite[Lemma 2.13]{Gil-Tru}), and hence
$c\coloneqq \lim_{|x| \to \infty} \tilde{u}(x)$ exists. It is easy to check that
$\nabla u = -\nabla^\perp v = \nabla \tilde{u}$ pointwise in $(0, \infty)^2$ as well as in
$(-\infty, 0) \times (0, \infty)$.
Since $\lim_{x_1 \to \pm \infty} u(x_1, 0) = 0$ (by \eqref{eqn:integral_rep_u}), we conclude that $\tilde{u} - c$ must coincide with the
function $u$ determined in formula \eqref{eqn:integral_rep_u}.
\end{subproof}

\begin{subproof} {Step \step: Other properties of $u$}
An immediate consequence of the above is that $u\to 0$ as $|x|\to \infty$.
Also, note that \eqref{eqn:integral_rep_u} implies that
$$|u(x_1,0)|\leq \int_0^\infty \frac1{1+t^2} \, dt=\frac \pi 2, \quad \textrm{for every } x_1\neq 0.$$
Since $u\to 0$ at infinity and $u$ is harmonic in $\R^2_+$, the maximum principle yields $|u|\leq \frac \pi 2$ in $\R^2_+$.
\end{subproof}
\end{proof}

Having the functions $u$ and $v$ given by \eqref{eqn:integral_rep_u} and \eqref{eqn:integral_rep_v}, we also define,
for a given $b \in \R$, the translated functions
\be
\label{def_ub_vb_unconfined}
u_b(x) = u(x_1 - b, x_2) \quad \textrm{and} \quad v_b(x) = v(x_1 - b, x_2).
\ee
Recall that for the unconfined problem, we define
\[
A_N = \set{a = (a_1, \dotsc, a_N) \in \R^N}{a_1 < \dotsb < a_N}.
\]
Then, given $a \in A_N$ and $d \in \{\pm 1\}^N$, we define
\[
u_{a, d}^* = \sum_{n = 1}^N \gamma_n u_{a_n} \quad \text{and} \quad v_{a, d}^* = \sum_{n = 1}^N \gamma_n v_{a_n}.
\]
Here $\gamma_n = d_n - \cos \alpha$, as defined in the introduction.
We also write
\[
\mu_{a, d}^*(x_1) = v_{a, d}^*(x_1, 0).
\]

\subsection{The limiting stray field potential: further properties} \label{sect:further_properties}

The above functions $u_{a, d}^*$ and $v_{a, d}^*$ have infinite Dirichlet energy due to the singularities
at $(a_n, 0)$ for $n = 1, \dotsc, N$. It follows from the inequality of Proposition \ref{prop:conjugate_harmonic_functions},
however, that $u_{a, d}^*, v_{a, d}^* \in W_\loc^{1, p}(\R_+^2)$ for any $p \in [1, 2)$ (by the dyadic decomposition
argument of Struwe \cite{Struwe}).
Indeed, given $p \in [1, 2)$, $R > 0$, $N \in \N$, and a compact set $K \subset \R^2$, there exists a
constant $C$ such that
\[
\int_{\R_+^2 \cap K} |\nabla u_{a, d}^*|^p \, dx \le C
\]
for all $a \in A_N$ with $a_{n + 1} - a_n \ge R$ for $n = 1, \dotsc, N - 1$ and for all $d \in \{\pm 1\}^N$.

In contrast to \eqref{def_rho}, we now define
\be
\label{def_rho_unconf}
\rho(a) = \frac{1}{2} \min\{a_2 - a_1, \dotsc, a_N - a_{N - 1}\}
\ee
for $a \in A_N$. Moreover, we define
\[
\Omega_r(a) = \R_+^2 \setminus \bigcup_{n = 1}^N B_r(a_n, 0)
\]
again for $r > 0$ and consider
\be
\label{w1_unconfined}
W_1(a, d) = \frac{1}{2} \lim_{r \searrow 0} \left(\int_{\Omega_r(a)} |\nabla u_{a, d}^*|^2 \, dx + \int_{-\infty}^\infty \left(\mu_{a, d}^*\right)^2 \, dx_1 - \pi \log \frac{1}{r} \sum_{n = 1}^N \gamma_n^2\right).
\ee
We will prove that this limit indeed exists and is finite as well. This is an important quantity that will
appear in the renormalised energy. It is the counterpart of the contribution of the tail-tail interaction in the confined case,
cf.\ \eqref{minusW}. The difference consists in the additional term involving $\mu_{a, d}^*$, which comes
from the anisotropy. We will prove the following.

\begin{proposition} \label{prop:tail-tail_interaction}
 If $I_0$ is defined as in \eqref{eqn:Euler-Mascheroni}, then
\[
W_1(a, d) = \frac\pi 2I_0 \sum_{n = 1}^N \gamma_n^2 + \frac{\pi}{2} \sum_{n = 1}^N \sum_{k \not= n} \gamma_k \gamma_n I(|a_k - a_n|).
\]
\end{proposition}

As we prove in Lemma \ref{lem:energy_near_wall} below, the first term above (independent of $a$) represents the (intrinsic) renormalised energy induced by the limiting stray field and anisotropy energy of every N\'eel wall. We observe the same phenomenon as in the confined case: $W_1(a,d)=-W(a,d)$, meaning that the contribution of the interactions between the tails of one
N\'eel walls and the core of another is twice the size in absolute value but of
opposite sign.

We decompose the proof of this statement into several lemmas.
As before, we consider the functions $v$ and $u$ given
by the formulas \eqref{eqn:integral_rep_v} and \eqref{eqn:integral_rep_u}.
We first need some pointwise estimates. The next proposition states that like in the confined case,
$u$ behaves analogously to the phase of a vortex of degree $+1$ close to the origin. But in contrast to the confined case, the error is logarithmic here.

\begin{proposition} \label{prop:pointwise_estimates}
Let
\[
w_0(x) = - \arctan \frac{x_1}{x_2}
\]
for $x \in \R_+^2$.
Then
\begin{equation} \label{eqn:pointwise}
\sup_{x \in \R_+^2} \frac{|\nabla u(x) - \nabla w_0(x)|}{\log \left(\frac{1}{|x|} + 2\right)} < \infty.
\end{equation}
As a consequence, 
$$
\sup_{x \in \R_+^2} \frac{|u(x) - w_0(x)|}{|x|\log \left(\frac{1}{|x|}+2\right)} < \infty \quad \textrm{and} \quad \sup_{x \in \R_+^2} \frac{\left|v(x) - \log \left(\frac{1}{|x|}\right)-I_0\right|}{|x|\log \left(\frac{1}{|x|}+2\right)} < \infty.
$$
Therefore, for any $a \in A_N$, there exists $C > 0$ such that
\[
\left|\nabla u_{a, d}^*(x) - \gamma_n \nabla w_0(x_1 - a_n, x_2)\right| \le C \log \left(\frac1{(x_1 - a_n)^2 + x_2^2}+2\right) \quad \text{in $B_{\rho(a)}^+(a_n,0)$}
\]
for $n = 1, \dotsc, N$.  Moreover,\footnote{The term $O(1)$ depends on $a$.}
\be
\label{estim_u*}
\int_{\Omega_\delta(a)} |\nabla u_{a, d}^*|^2\, dx=\pi |\log \delta|\sum_{n = 1}^N \gamma_n^2+O(1) \quad \textrm{as}\quad \delta \searrow 0.
\ee
\end{proposition}

\begin{proof}
We divide the proof in several steps.

\setcounter{stepno}{0}
\begin{subproof}{Step \step: proof of \eqref{eqn:pointwise}}
We compute, for $x\neq 0$:\footnote{Recall that for $z=x_1+ix_2\in \C$ with $x_1 > 0$, we have $\int_0^\infty e^{-tz}\, dt=\frac1z$.}
\[
\begin{split}
\dd{u}{x_1}(x) & = \int_0^\infty e^{-t|x_1|} \frac{t\cos(tx_2) - t^2\sin(tx_2)}{1 + t^2} \, dt \\
& = \int_0^\infty e^{-t|x_1|} \left(\frac{t\cos(tx_2) + \sin(tx_2)}{1 + t^2} - \sin(tx_2)\right) \, dt \\
& = \int_0^\infty e^{-t|x_1|} \frac{t\cos(tx_2) + \sin(tx_2)}{1 + t^2} \, dt - \frac{x_2}{x_1^2 + x_2^2}
\end{split}
\]
and
\[
\begin{split}
\dd{u}{x_2}(x) & = \frac{x_1}{|x_1|} \int_0^\infty e^{-t|x_1|} \frac{t\sin(tx_2) + t^2\cos(tx_2)}{1 + t^2} \, dt \\
& = \frac{x_1}{|x_1|} \int_0^\infty e^{-t|x_1|} \left(\frac{t\sin(tx_2) - \cos(tx_2)}{1 + t^2} + \cos(tx_2)\right) \, dt \\
& = \frac{x_1}{|x_1|} \int_0^\infty e^{-t|x_1|} \frac{t\sin(tx_2) - \cos(tx_2)}{1 + t^2} \, dt + \frac{x_1}{x_1^2 + x_2^2}.
\end{split}
\]
Hence
\[
\dd{u}{x_1}(x) - \dd{w_0}{x_1}(x) = \int_0^\infty e^{-t|x_1|} \frac{t\cos(tx_2) + \sin(tx_2)}{1 + t^2} \, dt
\]
and
\[
\dd{u}{x_2}(x) - \dd{w_0}{x_2}(x) = \frac{x_1}{|x_1|} \int_0^\infty e^{-t|x_1|} \frac{t\sin(tx_2) - \cos(tx_2)}{1 + t^2} \, dt.
\]
Clearly
\[
\left|\int_0^\infty e^{-t|x_1|} \frac{\sin(tx_2)}{1 + t^2} \, dt\right| \le \int_0^\infty \frac{dt}{1 + t^2} = \frac{\pi}{2},
\]
and similarly,
\[
\left|\int_0^\infty e^{-t|x_1|} \frac{\cos(tx_2)}{1 + t^2} \, dt\right| \le \frac{\pi}{2}.
\]
For the estimate of
\[
\int_0^\infty e^{-t|x_1|} \frac{t\cos(tx_2)}{1 + t^2} \, dt \quad \text{and} \quad \int_0^\infty e^{-t|x_1|} \frac{t\sin(tx_2)}{1 + t^2} \, dt,
\]
it suffices to consider $x_1 > 0$ by the symmetry. We distinguish the following three cases. 

\setcounter{caseno}{0}
\begin{subproof}{Case \case: $x_1 \ge x_2$}
Note that
\[
\left|\int_0^\infty e^{-t|x_1|} \frac{t\cos(tx_2)}{1 + t^2} \, dt \right| \le \int_0^\infty \frac{t e^{-t|x_1|}}{1 + t^2} \, dt = \int_0^\infty \frac{s e^{-s}}{x_1^2 + s^2} \, ds = I(|x_1|),
\]
and similarly
\[
\left|\int_0^\infty e^{-t|x_1|} \frac{t\sin(tx_2)}{1 + t^2} \, dt\right| \le I(|x_1|).
\]
By Lemma \ref{lem:I_is_logarithmic}, there exists a universal constant $C$ such that
\[
I(x_1) \leq \min{\left\{\log \left(\frac{1}{x_1}\right)+I_0+\frac{\pi x_1}{2}, \frac1{x_1^2}\right\}} \le C\log \left(\frac{1}{x_1} + 2\right) \quad \textrm{for all $x_1 > 0$.}
\]
Together with the previous inequalities, this implies \eqref{eqn:pointwise} for all points $(x_1, x_2) \in \R_+^2$ with $x_1 \ge x_2$.
\end{subproof}

\begin{subproof}{Case \case: $0 < x_1 < x_2 \le 1$}
A substitution gives
\[
\int_0^\infty e^{-t|x_1|} \frac{t \cos(tx_2)}{1 + t^2} \, dt = \int_0^\infty e^{-s|x_1|/x_2} \frac{s \cos s}{x_2^2 + s^2} \, ds.
\]
For $k \in \N$, let
\[
s_k = \int_{\pi k - \pi/2}^{\pi k + \pi/2} e^{-s|x_1|/x_2} \frac{s \cos s}{x_2^2 + s^2} \, ds.
\]
Then $s_k > 0 > s_{k + 1}$ whenever $k$ is even. Moreover, since the function $s \mapsto e^{-s|x_1|/x_2} s/(x_2^2 + s^2)$
is decreasing for $s \ge x_2$, we also conclude that $|s_{k + 1}| < |s_k|$ for all $k \in \N$ (because $x_2\leq 1$). Therefore,
\[
\sum_{k = 1}^\infty s_k = \sum_{k = 1}^\infty (s_{2k - 1} + s_{2k}) < 0
\]
and
\[
\sum_{k = 2}^\infty s_k = \sum_{k = 1}^\infty (s_{2k} + s_{2k + 1}) > 0.
\]
These series converge by the Leibniz criterium for alternating series. Moreover, they 
correspond to certain integrals, and therefore, we obtain
\[
\begin{split}
\int_0^\infty e^{-s|x_1|/x_2} \frac{s \cos s}{x_2^2 + s^2} \, ds & = \int_0^{\pi/2} e^{-s|x_1|/x_2} \frac{s \cos s}{x_2^2 + s^2} \, ds + \sum_{k = 1}^\infty s_k \\
& < \int_0^{\pi/2} e^{-s|x_1|/x_2} \frac{s \cos s}{x_2^2 + s^2} \, ds
\le \int_0^{\pi/2} \frac{s}{x_2^2 + s^2} \, ds
= \log \sqrt{\frac{\pi^2}{4x_2^2} + 1}
\end{split}
\]
and
\[
\begin{split}
\int_0^\infty e^{-s|x_1|/x_2} \frac{s \cos s}{x_2^2 + s^2} \, ds & = \int_0^{3\pi/2} e^{-s|x_1|/x_2} \frac{s \cos s}{x_2^2 + s^2} \, ds + \sum_{k = 2}^\infty s_k \\
& > \int_0^{3\pi/2} e^{-s|x_1|/x_2} \frac{s \cos s}{x_2^2 + s^2} \, ds
\ge -\int_0^{3\pi/2} \frac{s}{x_2^2 + s^2} \, ds
= -\log \sqrt{\frac{9\pi^2}{4x_2^2} + 1}.
\end{split}
\]
The same kind of estimate holds for $\int_0^\infty e^{-s|x_1|/x_2} \frac{s \sin s}{x_2^2 + s^2} \, ds$.
Again we have suitable estimates for all the above integrals, which means that \eqref{eqn:pointwise} is proved
in this case.
\end{subproof}

\begin{subproof}{Case \case: $0 < x_1 < x_2$ and $x_2 > 1$}
Here we proceed differently. We apply the mean value formula for the harmonic functions $\frac{\partial u}{\partial x_j}$, $j=1,2$,
in the ball
$B_{1/2}(x)$. We combine the resulting formula with the estimate for the Dirichlet
energy in Proposition \ref{prop:conjugate_harmonic_functions}, which yields a uniform bound for
$|\nabla u|$ in $\R \times [1, \infty)$. Thus \eqref{eqn:pointwise} is proved
in this case as well.
\end{subproof}
\end{subproof}

\begin{subproof}{Step \step: behaviour of $u$ and $v$ near the origin} 
By \eqref{eqn:integral_rep_u}, the dominated convergence theorem implies that $u(x_1,0)\to -{\sgn}(x_1) \int_0^\infty \frac{dt}{1+t^2}=w_0(0\pm)$ as $x_1\searrow 0$ or $x_1 \nearrow 0$. By Lemma \ref{lem:I_is_logarithmic} and \eqref{v_x1}, we also know
that $v(x_1,0)-\log\frac1{|x_1|}-I_0\to 0$ as $x_1\to 0$. Then the fundamental theorem of calculus, combined with \eqref{eqn:pointwise}
and $\nabla v=\nabla^\perp u$, gives the second statement of the proposition.
\end{subproof}

\begin{subproof}{Step \step: estimates on $\nabla u_{a, d}^*$}
The desired pointwise estimate for $\nabla u_{a, d}^*$  is an obvious consequence of \eqref{eqn:pointwise}, while \eqref{estim_u*} follows from Proposition 
\ref{prop:conjugate_harmonic_functions}.
\end{subproof}
\end{proof}

The contribution of the tail of a single N\'eel wall to the renormalised energy is computed in the next lemma. As mentioned previously, 
in the unconfined model, this contribution depends both on the stray-field energy and the anisotropy energy.
This also improves the estimate of Proposition~\ref{prop:conjugate_harmonic_functions}.

\begin{lemma} \label{lem:energy_near_wall}
The following limit exists and is given by
\[
\lim_{r \searrow 0} \left(\int_{\R_+^2 \setminus B_r(0)} |\nabla v|^2 \, dx +\int_{-\infty}^\infty v^2(x_1,0)\, dx_1- \pi \log \frac{1}{r}\right)=\pi I_0,
\]
where $I_0$ is defined by \eqref{eqn:Euler-Mascheroni}.
\end{lemma}

\begin{proof}
Let $r\in (0,1)$. Denote $\partial^+ B_r(0)=\partial B_r(0)\cap \R^2_+$ and write $\h^1$
for the $1$-dimensional Hausdorff measure.
Furthermore, we write $\dd{v}{r} = \frac{x}{|x|} \cdot \nabla v$.
Since $v$ is the harmonic conjugate of $u$, we have
\[
\int_{\R^2_+\setminus B_r(0)} |\nabla u|^2\, dx=\int_{\R^2_+\setminus B_r(0)} |\nabla v|^2\, dx=-\int_{\partial^+ B_r(0)} v \dd{v}{r} \, d\h^1-\int_{\R\setminus (-r,r)} v \dd{v}{x_2} \, dx_1.
\]
For the last integral, since $v\in L^2(\R)$ (see \eqref{planch_v^2}), we know by \eqref{eqn:boundary_condition_for_v_unconfined} that
$$\int_{\R\setminus (-r,r)} v \dd{v}{x_2} \, dx_1=\int_{\R\setminus (-r,r)} v^2 \, dx_1\to \int_{-\infty}^\infty v^2(x_1, 0) \, dx_1$$
as $r\to 0$. For the remaining integral, denoting $$v_0(x)=-\log |x|, \quad x\in \R^2\setminus \{0\},$$
we compute:
\[
\int_{\partial^+ B_r(0)} v \dd{v}{r} \, d\h^1=\int_{\partial^+ B_r(0)} v \left(\dd{v}{r}-\dd{v_0}{r}\right) \, d\h^1 +\int_{\partial^+ B_r(0)} (v-v_0) \dd{v_0}{r} \, d\h^1+\int_{\partial^+ B_r(0)} v_0\dd{v_0}{r} \, d\h^1.
\]
Note that
$$\int_{\partial^+ B_r(0)} v_0\dd{v_0}{r} \, d\h^1=\pi \log r.$$ 
By Proposition \ref{prop:pointwise_estimates}, we know that 
$$\lim_{x\to 0} (v(x)-v_0(x))=I_0.$$
Consequently,
$$\int_{\partial^+ B_r(0)} (v-v_0) \dd{v_0}{r} \, d\h^1=-\frac1r \int_{\partial^+ B_r(0)} (v-v_0) \, d\h^1\to -\pi I_0$$
as $r\to 0$. Finally, by Proposition \ref{prop:pointwise_estimates} again,
\begin{align*}
\biggl| \int_{\partial^+ B_r(0)} v \left(\dd{v}{r}-\dd{v_0}{r}\right) \, d\h^1\biggr|&\leq C \int_{\partial^+ B_r(0)} 
|v| \log\left(\frac1{r}+2\right) \, d\h^1\\
&\leq C \int_{\partial^+ B_r(0)} 
\left(\log \frac1r +C'\right) \log\left(\frac1{r}+2\right) \, d\h^1\to 0,
\end{align*}
for some universal constants $C$ and $C'$, as $r\to 0$. The claim now follows.
\end{proof}

The contribution of the tail-tail interaction between two N\'eel walls is computed as follows.
\begin{lemma} \label{lem:cross-terms}
Let $b, c \in \R$ with $b \not= c$. If $v_b$ and $v_c$ are defined as in \eqref{def_ub_vb_unconfined}, then
\[
\int_{\R_+^2} \nabla v_b \cdot \nabla v_c \, dx + \int_{-\infty}^\infty v_b(x_1, 0) v_c(x_1, 0) \, dx_1 = \pi I(|b-c|).
\]
\end{lemma}

\begin{proof}
Clearly is suffices to prove this identity for $c = 0$ and $b > 0$.

Recalling \eqref{eqn:Fourier_rep_v}, we obtain a similar formula for the Fourier transform of $v_b$:
\[
\hat{v}_b(\xi, x_2) = \frac{\pi e^{-ib\xi - |\xi|x_2}}{1 + |\xi|}.
\]
Thus for $r > 0$, Plancherel's theorem and Fubini's theorem yield
\[
\begin{split}
&\int_{\R \times (r, \infty)} \dd{v_b}{x_1} \dd{v}{x_1} \, dx = \frac{\pi}{2} \int_r^\infty \int_{-\infty}^\infty \frac{\xi^2 e^{-ib\xi - 2|\xi| x_2}}{(1 + |\xi|)^2} \, d\xi \, dx_2 \\
& = \frac{\pi}{2} \int_{-\infty}^\infty \frac{\xi^2 e^{-ib\xi}}{(1 + |\xi|)^2} \int_r^\infty e^{-2|\xi|x_2} \, dx_2 \, d\xi 
= \frac{\pi}{4} \int_{-\infty}^\infty  \frac{|\xi| e^{-ib\xi - 2|\xi|r}}{(1 + |\xi|)^2} \, d\xi \\
& = \frac{\pi}{4} \left(\int_0^\infty \frac{\xi e^{-ib\xi - 2\xi r}}{(1 + \xi)^2} \, d\xi + \int_0^\infty \frac{\xi e^{ib\xi - 2\xi r}}{(1 + \xi)^2} \, d\xi\right)  = \frac{\pi}{2} \int_0^\infty \frac{\xi \cos(b\xi)}{(1 + \xi)^2} e^{-2\xi r} \, d\xi \\
& = \frac{\pi}{2} \int_0^\infty \frac{t \cos t}{(b + t)^2} e^{-2tr/b} \, dt.
\end{split}
\]
We wish to consider the limit $r \searrow 0$, but since the function $t \mapsto \frac{t \cos t}{(b + t)^2}$ does
not belong to $L^1(0, \infty)$, some care is required here.

\begin{claim}
For any $b > 0$,
\[
\lim_{r \searrow 0} \int_0^\infty \frac{t \cos t}{(b + t)^2} e^{-2tr/b} \, dt=\int_0^\infty \frac{t \cos t}{(b + t)^2} \, dt.
\]
\end{claim}

\begin{subproof}{Proof of the claim}
For $k \in \N$ and $r \ge 0$, define $s_k(r)=s_k^1(r)+s_k^2(r)$, where
\[
s^1_k(r) = \int_{2k \pi - \pi/2}^{2k \pi + \pi/2} \frac{t \cos t}{(b + t)^2} e^{-2tr/b} \, dt, \quad
s^2_k(r) = \int_{2k \pi + \pi/2}^{2k \pi + 3\pi/2} \frac{t \cos t}{(b + t)^2} e^{-2tr/b} \, dt,
\]
and define $\sigma_k=\sigma_k^1+\sigma_k^2$, where
\[
\sigma^1_k = \int_{2k \pi - \pi/2}^{2k \pi + \pi/2} \frac{\cos t}{b + t} \, dt, \quad
\sigma^2_k = \int_{2k \pi +\pi/2}^{2k \pi + 3\pi/2} \frac{\cos t}{b + t} \, dt.
\]
Since the functions $t \mapsto te^{-2tr/b}/(b + t)^2$ are decreasing for $t > b$, there exists
$k_0 \in \N$ such that $s^1_k(r) > |s^2_k(r)|$ for $k \ge k_0$. Hence $s_k(r) > 0$ for every $k \ge k_0$ and every $r \ge 0$. Similarly,
$\sigma_k > 0$ for all $k \in \N$. We have the formulas
\[
\int_{3\pi/2}^\infty \frac{t \cos t}{(b + t)^2} e^{-2tr/b} \, dt = \sum_{k = 1}^\infty s_k(r)<\infty \quad \text{for } r>0,
\]
and\footnote{The convergence of the integral was also proved via the contour integral in Lemma \ref{lem:alternative_rep_of_I}.}
\[
\int_{3\pi/2}^\infty \frac{\cos t}{b + t} \, dt = \sum_{k = 1}^\infty \sigma_k < \infty.
\]
The convergence of these series follows from the Leibniz criterion for alternating series, since $s^1_k(r)>0>s^2_k(r)$ and 
$|s_k^i(r)|\searrow 0$ as $k\to \infty$ for $i=1,2$, and similarly, 
$\sigma_k^1>0>\sigma_k^2$ and $|\sigma_k^i|\searrow 0$ as $k\to \infty$ for $i=1,2$. We now consider, for $r > 0$, the function
\[
f_r(t) = \frac{1}{b + t} - \frac{te^{-2rt/b}}{(b + t)^2} \quad \text{for } t>0
\]
and claim that $f_r$ is decreasing in $t$. Indeed, we compute
$$
f_r'(t) = \frac{t}{(b + t)^3} \left(\frac{2rt}{b} e^{-2rt/b} - 1 + e^{-2rt/b}\right) 
+ \frac{b}{(b + t)^3} \left(\frac{2rt}{b} e^{-2rt/b} - 1 - e^{-2rt/b}\right).
$$
For any $R \ge 0$, the inequality $Re^{-R} \le 1 - e^{-R}$ is easy to see. Using it for $R = 2rt/b$,
we find that $f_r'(t) < 0$. In particular, we have $f_r\geq \lim_{t\to \infty} f_r(t)=0$.

Now we observe that
\[
\sigma_k - s_k(r) = \int_{2k \pi - \pi/2}^{2k \pi + 3\pi/2} f_r(t) \cos t \, dt.
\]
By the monotonicity of $f_r$, 
as we integrate over a full period of $\cos$ beginning with its positive part, we
conclude that $\sigma_k \ge s_k(r)$ for all $r > 0$ and all $k \in \N$.
Also, we clearly have the convergence
\[
s_k(0) = \lim_{r \searrow 0} s_k(r).
\]
As $\sum_k \sigma_k$ converges, Lebesgue's dominated convergence theorem, applied to the counting measure,
now implies that
\[
\sum_{k = 1}^\infty s_k(0) = \lim_{r \searrow 0} \sum_{k = 1}^\infty s_k(r).
\]
This completes the proof of the claim.
\end{subproof}

Finally, Propositions \ref{prop:conjugate_harmonic_functions} and \ref{prop:pointwise_estimates} imply
that $\dd{v_b}{x_1} \dd{v}{x_1} \in L^1(\R_+^2)$. Therefore, the above claim  implies
\[
\int_{\R_+^2} \dd{v_b}{x_1} \dd{v}{x_1} \, dx = \lim_{r \searrow 0} \int_{\R \times (r, \infty)} \dd{v_b}{x_1} \dd{v}{x_1} \, dx = \frac{\pi}{2} \int_0^\infty \frac{t \cos t}{(b + t)^2} \, dt.
\]
Moreover, the same computations give
\[
\int_{\R_+^2} \dd{v_b}{x_2} \dd{v}{x_2} \, dx = \frac{\pi}{2} \int_0^\infty \frac{t \cos t}{(b + t)^2} \, dt
\]
as well. Thus
\[
\int_{\R_+^2} \nabla v_b \cdot \nabla v \, dx = \pi \int_0^\infty \frac{t \cos t}{(b + t)^2} \, dt.
\]

Set $\mu(x_1) = v(x_1, 0)$ and $\mu_b(x_1) = v(x_1 - b, 0)$. Then
\[
\hat{\mu}(\xi) = \frac{\pi}{1 + |\xi|}
\quad \textrm{and} \quad 
\hat{\mu}_b(\xi) = \frac{\pi e^{-ib\xi}}{1 + |\xi|}.
\]
Hence
\[
\int_{-\infty}^\infty \mu_b \mu \, dx_1 = \frac{\pi}{2} \int_{-\infty}^\infty \frac{e^{-ib\xi}}{(1 + |\xi|)^2} \, d\xi = \pi \int_0^\infty \frac{\cos(b\xi)}{(1 + \xi)^2} \, d\xi = b\pi \int_0^\infty \frac{\cos t}{(b + t)^2} \, dt.
\]
If we combine this with the previous identity, we obtain, by Lemma \ref{lem:alternative_rep_of_I},
\[
\int_{\R_+^2} \nabla v_b \cdot \nabla v \, dx + \int_{-\infty}^\infty \mu_b \mu \, dx_1 = \pi \int_0^\infty \frac{\cos t}{b + t} \, dt=\pi I(b),
\]
which is the desired identity.
\end{proof}

\begin{proof}[Proof of Proposition \ref{prop:tail-tail_interaction}]
This is now an easy consequence of Lemmas \ref{lem:energy_near_wall} and \ref{lem:cross-terms}. 
\end{proof}

\subsection{Preliminary estimates}

The rest of the arguments for the proof of Theorem \ref{thm:renormalised_unconfined}
follow the strategy of the proof of Theorem \ref{thm:renormalised_confined}, given in our previous
paper \cite{Ignat-Moser:16}. Because these arguments are lengthy, we do not repeat
the details and merely highlight the modifications where necessary.

In the first step, we prove the following result. This has a direct counterpart for
the confined problem from Section \ref{sect:confined}, see \cite[Theorem 28]{Ignat-Moser:16}.

\begin{theorem}[Stray field energy estimate] \label{thm:stray_field} 
Let $N \in \N$, $R > 0$, and $C_0 > 0$. Then there exists $C_1 > 0$ such that
for any $a \in A_N$ with $\rho(a) \ge R$ and any $d \in \{\pm 1\}^N$, the
following holds true. Let
\be
\label{def_gam}
\Gamma = \sum_{n = 1}^N (d_n - \cos \alpha)^2.
\ee
Suppose that $\epsilon \in (0, \frac{1}{2}]$ with $\delta \le R$ and $m \in M(a, d)$ with
\[
E_\epsilon(m) \le \frac{\pi \Gamma}{2\log \frac{1}{\delta}} + \frac{C_0}{(\log \delta)^2}.
\]
Let $u \in \dot{H}^1(\R_+^2)$ be the solution of \eqref{eqn:stray_field_potential} and \eqref{eqn:boundary_condition}.
Then
\[
\eps \int_{\R} |m'|^2\, dx_1+  \int_{\R} (m_1-\cos \alpha)^2\, dx_1+\int_{\Omega_\delta(a)} \left|\nabla u - \frac{\nabla u_{a, d}^*}{\log \frac{1}{\delta}}\right|^2 \, dx \le \frac{C_1}{(\log \delta)^2}
\]
and
\[
\int_{\Omega_\delta(a)} |\nabla u|^2 \, dx \ge \frac{\pi \Gamma}{\log \frac{1}{\delta}} - \frac{C_1}{(\log \delta)^2}.
\]
\end{theorem}

\begin{proof} We first note that it suffices to prove these inequalities for $\epsilon$ small. Consider the function $u_{a, d}^*/|\log \delta|$ and modify it near the points $(a_n, 0)$, $n = 1, \dotsc, N$,
similarly to the proof of \cite[Theorem 28]{Ignat-Moser:16}. To this end, note first that
by Proposition \ref{prop:pointwise_estimates}, we know that
$u_{a, d}^*(\blank, 0)$ is continuous on $(a_{n - 1}, a_n)$ for $n = 2, \dotsc, N$
and on $(-\infty, a_1)$ and $(a_N, \infty)$, and the one-sided limits
\[
\tau_n^+ = \lim_{x_1 \searrow a_n} u_{a, d}^*(x_1, 0) \quad \text{and} \quad \tau_n^- = \lim_{x_1 \nearrow a_n} u_{a, d}^*(x_1, 0)
\]
exist.
Then $\tau_n^+ - \tau_n^- = -\gamma_n \pi$ for $n = 1, \dotsc, N$ (by Proposition \ref{prop:pointwise_estimates} again).
For a given number $s \in (0, R]$, we now define $\xi_s \in \dot{H}^1(\R_+^2)$ such that
$\xi_s = u_{a, d}^*/|\log \delta|$ in $\Omega_s(a)$ and
\be
\label{def_xis}
\xi_s(a_n + r\cos \theta, r\sin \theta) = \frac{ru_{a, d}^*(a_n + s\cos \theta, s\sin \theta)}{s \log \frac{1}{\delta}} + \left(1 - \frac{r}{s}\right)\frac{\tau_n^+ + \tau_n^-}{2\log \frac{1}{\delta}}
\ee
for $0 < r \le s$, $0 < \theta < \pi$, and $n = 1, \dotsc, N$.
Then the inequalities of Propositions \ref{prop:conjugate_harmonic_functions} and \ref{prop:pointwise_estimates}
imply
\be
\label{eq77}
\int_{\bigcup_{n=1}^N B_s(a_n,0)} |\nabla \xi_s|^2 \, dx \le \frac{C_2}{(\log \delta)^2} \quad \textrm{and} \quad 
\int_{\R_+^2} |\nabla \xi_s|^2 \, dx \le \frac{\pi \Gamma \log \frac{1}{s} + C_2}{(\log \delta)^2}
\ee
for some constant $C_2 = C_2(N, R)$. 
Also define
\[
\zeta(x_1) = \frac{v_{a, d}^*(x_1, 0)}{\log \frac{1}{\delta}}.
\]
Using \eqref{planch_v^2}, we easily find a constant $C_3 = C_3(N)$ such that
\begin{equation} \label{eqn:L^2-estimate_for_v*}
\|\zeta\|_{L^2(\R)} \le \frac{C_3}{\log \frac{1}{\delta}}.
\end{equation}

Since $m_1(a_n) = d_n$ for $n = 1, \dotsc, N$ and $\lim_{x_1 \to \pm \infty} m_1(x_1) = \cos \alpha$,
the fundamental theorem of calculus gives the identity
\[
\int_{a_{n - 1}}^{a_n} \frac{d}{dx_1} \left((m_1(x_1) - \cos \alpha) u_{a, d}^*(x_1, 0)\right) \, dx_1 = \gamma_n \tau_n^- - \gamma_{n - 1} \tau_{n - 1}^+
\]
for $n = 2, \dots, N$, while
\[
\int_{-\infty}^{a_1} \frac{d}{dx_1} \left((m_1(x_1) - \cos \alpha) u_{a, d}^*(x_1, 0)\right) \, dx_1 = \gamma_1 \tau_1^-
\]
and
\[
\int_{a_N}^\infty \frac{d}{dx_1} \left((m_1(x_1) - \cos \alpha) u_{a, d}^*(x_1, 0)\right) \, dx_1 = -\gamma_N \tau_N^+.
\]
Observing that $\dd{}{x_1} u_{a, d}^* = \dd{}{x_2} v_{a, d}^* = v_{a, d}^*$ on
$(\R \setminus \{a_1, \dotsc, a_N\}) \times \{0\}$, we find that
$$
\frac{d}{dx_1} \left((m_1(x_1) - \cos \alpha) u_{a, d}^*(x_1, 0)\right) = m_1'(x_1) u_{a, d}^*(x_1, 0) 
+ (m_1(x_1) - \cos \alpha) v_{a, d}^*(x_1)
$$
away from $a_1, \dotsc, a_N$. Therefore,
\[
\begin{split}
\int_{-\infty}^\infty \left(m_1'(x_1) u_{a, d}^*(x_1, 0) + (m_1(x_1) - \cos \alpha) v_{a, d}^*(x_1)\right) \, dx_1 & = \sum_{n = 1}^N \gamma_n (\tau_n^- - \tau_n^+)  = \pi \Gamma.
\end{split}
\]
It follows from  \eqref{eqn:stray_field_potential},
\eqref{eqn:boundary_condition}, the definition of $\xi_s$, and \cite[Lemma 29]{Ignat-Moser:16} that
\begin{align*}
\frac{\pi \Gamma}{\log \frac{1}{\delta}} & = \int_{-\infty}^\infty \xi_s(x_1, 0) m_1'(x_1) \, dx_1 + \int_{\R} (m_1 - \cos \alpha) \zeta \, dx_1 
- \int_{\R} \left(\xi_s(x_1, 0) - \frac{u_{a, d}^*(x_1, 0)}{\log \frac{1}{\delta}}\right) m_1'(x_1) \, dx_1\\
&\leq \int_{\R_+^2} \nabla \xi_s \cdot \nabla u \, dx+\int_{-\infty}^\infty (m_1 - \cos \alpha) \zeta \, dx_1 +
\frac{C_4 s}{\log \frac{1}{\delta}} \int_{-\infty}^\infty |m'|^2 \, dx_1
\end{align*}
for a constant $C_4 = C_4(\alpha, N, R)>0$.
Hence
\begin{align}
\nonumber
\frac{\pi \Gamma}{\log \frac{1}{\delta}} & \le \frac{C_4 s}{\epsilon \log \frac{1}{\delta}}\left(2E_\epsilon(m) -\|\nabla u\|_{L^2(\R_+^2)}^2 - \|m_1 - \cos \alpha\|_{L^2(\R)}^2\right) \\
& \label{staru}\quad + \int_{\R_+^2} \nabla \xi_s \cdot \nabla u \, dx + \int_{-\infty}^\infty (m_1 - \cos \alpha) \zeta \, dx_1 \\
&\nonumber \le \frac{C_4 s}{\epsilon \log \frac{1}{\delta}}\left(2E_\epsilon(m) -\|\nabla u\|_{L^2(\R_+^2)}^2 - \|m_1 - \cos \alpha\|_{L^2(\R)}^2\right) \\
&\nonumber \quad + \sqrt{\|\nabla \xi_s\|_{L^2(\R_+^2)}^2 + \|\zeta\|_{L^2(\R)}^2} \sqrt{\|\nabla u\|_{L^2(\R_+^2)}^2 + \|m_1 - \cos \alpha\|_{L^2(\R)}^2} \\
& \nonumber\le \frac{C_4 s}{\epsilon \log \frac{1}{\delta}}\left(2E_\epsilon(m) -\|\nabla u\|_{L^2(\R_+^2)}^2 - \|m_1 - \cos \alpha\|_{L^2(\R)}^2\right) \\
&\nonumber \quad + \frac{\sqrt{\pi \Gamma \log \frac{1}{s} + C_5}}{\log \frac{1}{\delta}} \sqrt{\|\nabla u\|_{L^2(\R_+^2)}^2 + \|m_1 - \cos \alpha\|_{L^2(\R)}^2},
\end{align}
where $C_5 = C_2 + C_3^2$. If we choose $s$ such that
\[
\|\nabla u\|_{L^2(\R_+^2)}^2 + \|m_1 - \cos \alpha\|_{L^2(\R)}^2 = \frac{\pi\Gamma}{\log \frac{1}{s}} - \frac{2C_5}{(\log s)^2},
\]
then we obtain essentially the same inequality as in the proof of \cite[Theorem 28]{Ignat-Moser:16}, except that we
must now consider the quantities
$\nabla u$ and $m_1 - \cos \alpha$ jointly. The argument in the proof of \cite[Theorem 28]{Ignat-Moser:16}
still applies and yields $s\geq \tilde C \delta$ for some $\tilde C=\tilde C(\alpha, R, N, C_0)$.
We keep following the reasoning of \cite[Theorem 28]{Ignat-Moser:16} and obtain a constant $C_6 = C_6(\alpha, N, R)$ such that
\begin{equation} \label{eqn:lower_estimate_for_stray_field_energy}
\int_{\R_+^2} |\nabla u|^2 \, dx + \int_{-\infty}^\infty (m_1 - \cos \alpha)^2 \, dx_1 \ge \frac{\pi \Gamma}{\log \frac{1}{\delta}} - \frac{C_6}{(\log \delta)^2}.
\end{equation}

Next we use the first inequality in \eqref{staru} again, but with $s = \delta$.
Combining it with \eqref{eqn:lower_estimate_for_stray_field_energy}, we obtain a constant
$C_7 = C_7(\alpha, N, R)$ such that
$$\int_{\R_+^2} \nabla \xi_\delta \cdot \nabla u \, dx + \int_{-\infty}^\infty (m_1 - \cos \alpha) \zeta \, dx_1\geq 
\frac{\pi \Gamma}{\log \frac{1}{\delta}} - \frac{C_7}{(\log \delta)^2}.$$
Furthermore, the arguments in the proof of \cite[Theorem 28]{Ignat-Moser:16} give
\begin{equation} \label{eqn:estimate_of_u-xi_s}
\int_{\R^2_+} \left|\nabla u - \nabla \xi_{\delta} \right|^2 \, dx + \int_{-\infty}^\infty (m_1 - \cos \alpha - \zeta)^2 \, dx_1 \le \frac{C_8}{(\log \delta)^2}
\end{equation}
for some constant $C_8 = C_8(\alpha, N, R)$. Since $\xi_{\delta} = u_{a, d}^*/|\log \delta|$ in $\Omega_{\delta}(a)$,  we deduce that
\[
\int_{\Omega_\delta(a)} \left|\nabla u - \frac{\nabla u_{a, d}^*}{\log \frac{1}{\delta}}\right|^2 \, dx + \int_{-\infty}^\infty (m_1 - \cos \alpha - \zeta)^2 \, dx_1 \le \frac{C_8}{(\log \delta)^2}.
\]
Now inequalities \eqref{eq77}, \eqref{eqn:L^2-estimate_for_v*}, and \eqref{eqn:estimate_of_u-xi_s} imply that
$$\int_{\bigcup_{n=1}^N B_\delta(a_n, 0) } |\nabla u|^2\, dx=O\left(\frac{1}{(\log \delta)^2}\right)=\int_{-\infty}^\infty (m_1 - \cos \alpha)^2 \, dx_1.$$
We may combine this with \eqref{eqn:lower_estimate_for_stray_field_energy}. 
We conclude that
\[
\int_{\Omega_\delta(a)} |\nabla u|^2 \, dx \ge \frac{\pi \Gamma}{\log \frac{1}{\delta}} - O\left(\frac{1}{(\log \delta)^2}\right).
\]
Finally, this estimate, combined with the bound on $E_\eps(m)$, gives rise to the remaining inequality.
\end{proof}

We will also need some estimates for higher derivatives of critical points of $E_\epsilon$.
These will satisfy an Euler-Lagrange equation that is most easily stated in terms of
the lifting $\phi$ of a map $m \colon \R \to \mathbb{S}^1$. Suppose that $m = (\cos \phi, \sin \phi)$
is a critical point of $E_\epsilon$ and let $u \in \dot{H}^1(\R_+^2)$ be a solution of \eqref{eqn:stray_field_potential},
\eqref{eqn:boundary_condition}. Then
\begin{equation} \label{eqn:Euler-Lagrange_unconfined}
\epsilon \phi'' = (\cos \alpha - \cos \phi + u') \sin \phi \quad \text{in $\R$}.
\end{equation}
Here we use the shorthand notation $u'$ for $\dd{u}{x_1}(\blank, 0)$.
We refer to our previous paper \cite{Ignat-Moser:17} for a derivation.
Equation \eqref{eqn:boundary_condition} can be expressed in terms of $\phi$, too, yielding
\[
\dd{u}{x_2} = \phi' \sin \phi.
\]

For the functional without anisotropy, estimates for higher derivatives were obtained in \cite[Lemma 11]{Ignat-Moser:16}.
For the case $\epsilon = 1$ (but with anisotropy term), the same arguments were used
in \cite[Lemma 3.3]{Ignat-Moser:17}. Examining both proofs, it is easy to see that
the following statement is true. (We do not repeat the arguments here.)

\begin{lemma} \label{lem:higher_estimates}
Let $0 \le r < r' < R' < R$. Then there exists a constant $C$ (depending only on $r' - r$ and $R - R'$) such that
the following holds true. Let $\epsilon >0$ and $J = (-R, -r) \cup (r, R)$. Suppose that
$\phi \in H^1(J)$ and $u \in H^1(B_R^+(0) \setminus B_r(0))$ solve the system
\begin{alignat*}{2}
\Delta u & = 0 & \quad & \text{in $B_R^+(0) \setminus B_r(0)$}, \\
\dd{u}{x_2} & = \phi' \sin \phi && \text{on $J \times \{0\}$}, \\
\epsilon \phi'' & = (\cos \alpha - \cos \phi + u') \sin \phi && \text{in $J$}.
\end{alignat*}
Further suppose that $\sin \phi \not= 0$ in $J$. Then
\begin{multline*}
\int_{B_{R'}^+(0) \setminus B_{r'}(0)} |\nabla^2 u|^2 \, dx 
+ \int_{(-R', -r') \cup (r', R')} \left(\epsilon (\phi'')^2 + \epsilon (\phi')^4(1 + \cot^2 \phi) + (\phi')^2 \sin^2 \phi\right) \, dx_1 \\
\le C\epsilon \int_J (\phi')^2 \, dx_1 + C\int_{B_R^+(0) \setminus B_r(0)} |\nabla u|^2 \, dx.
\end{multline*}
\end{lemma}

The following statement relies on the Euler-Lagrange equation as well,
but applies to minimisers of the energy in $M(a, d)$. The arguments here are
similar to \cite[Lemma 3.1]{Ignat-Moser:17}. The result is useful above all
in view of the condition $\sin \phi \not= 0$ in the preceding lemma.

\begin{lemma} \label{lem:no_unaccounted_walls}
Suppose that $m \in M(a, d)$ satisfies
\[
E_\epsilon(m) = \inf_{M(a, d)} E_\epsilon.
\]
Then $m_1(x_1) \not= \pm 1$ for all $x_1 \in \R \setminus \{a_1, \dotsc, a_N\}$.
\end{lemma}

\begin{proof}
Choose $\phi \colon \R \to \R$ such that $m = (\cos \phi, \sin \phi)$.
Then $\phi$ satisfies the Euler-Lagrange equation \eqref{eqn:Euler-Lagrange_unconfined}
away from $a_1, \dotsc, a_N$. Assume, by way of contradiction, that we have
$b \in \R \setminus \{a_1, \dotsc, a_N\}$
with $\sin \phi(b) = 0$. Consider the
initial value problem
\[
\epsilon \psi'' = (\cos \alpha - \cos \phi + u') \sin \psi, \quad \psi(b) = \phi(b), \quad \psi'(b) = 0.
\]
Then $\psi \equiv \phi(b)$ is the unique solution.
The function $\phi$ also satisfies the differential equation and the first initial condition.
But clearly it cannot be constant, so we conclude that $\phi'(b) \neq 0$.\footnote{This kind of argument was 
also used by Capella-Melcher-Otto \cite{CMO}.} Adding a multiple of $2\pi$ if necessary,
we may assume that $\phi(b) = 0$ or $\phi(b) = \pi$. In the first case, we define
$\tilde{\phi} \colon \R \to \R$ by
\[
\tilde{\phi}(x_1) = \begin{cases} \phi(x_1) & \text{if $x_1 \le b$}, \\ -\phi(x_1) & \text{if $x_1 > b$}. \end{cases}
\]
Then $\tilde{m} = (\cos \tilde{\phi}, \sin \tilde{\phi})$ belongs to $M(a, d)$ as well, and
$E_\epsilon(\tilde{\phi}) = E_\epsilon(\phi)$. Hence $\tilde{m}$ minimises the
energy in $M(a, d)$ and $\tilde{\phi}$ is a solution of the Euler-Lagrange equation away
from $a_1, \dotsc, a_N$. We can show, however, that solutions of this equation are
necessarily smooth (see e.g. \cite[Theorem 1.1]{Ignat-Knuepfer:10}), whereas $\tilde{\phi}$ is clearly not smooth at $b$.
In the case $\phi(b) = \pi$, we can use a similar construction and obtain the same
contradiction. Hence there is  no point $b \in \R \setminus \{a_1, \dotsc, a_N\}$ with $\sin \phi(b) = 0$.
\end{proof}

\subsection{Proof of Theorem \ref{thm:renormalised_unconfined}}

As mentioned previously, we follow the arguments from \cite[Section 6]{Ignat-Moser:16}
in the proof of Theorem \ref{thm:renormalised_unconfined} without repeating all of the
details. In order to help the reader follow these arguments, we mimic the presentation
of the proof as well.

We fix $a \in A_N$ and $d \in \{\pm 1\}^N$. Set $\gamma_n = d_n - \cos \alpha$
for $n = 1, \dots, N$ and recall the definition 
$
\Gamma = \sum_{n = 1}^N \gamma_n^2
$
in \eqref{def_gam}.
Throughout the following arguments, we use the symbol
$C$ to indiscriminately denote various constants depending only on $\alpha$, $N$, and $a$,
and occasionally on the exponent of an $L^p$-space used in the context.

Part of the proof requires a construction where we glue a `tail' profile together with
a number of `core' profiles (one for every $n = 1, \dotsc, N$). While the tail profile comes from the
previously constructed function $u_{a, d}^*$ (and is different from \cite{Ignat-Moser:16}),
we use exactly the same core profile as in the previous paper. This may seem somewhat
inconsistent, as we neglect the anisotropy there, but in the limit the difference will
be invisible. The core profiles are minimisers of an auxiliary functional $E_\epsilon^{\gamma_\pm}$,
see \cite[Section 4]{Ignat-Moser:16}.
Hence the quantities $\gamma_\pm$ and $E_\epsilon^{\gamma_\pm}$ are the same as in the other paper. Furthermore, we now use the symbol $e$ for the function $e \colon \{\pm 1\} \to \R$ defined in our previous
paper \cite[Definition 26]{Ignat-Moser:16} (and also mentioned in Theorem \ref{thm:renormalised_confined})
and called the `core energy'.
To avoid confusion, we do not use the exponential function any more.

\subsubsection{Preparation} \label{subsubsection:preparation}

Recall the function
\[
w_0(x) = -\arctan \frac{x_1}{x_2}
\]
defined in Proposition \ref{prop:pointwise_estimates}.
(The same function is defined by a different formula in \cite{Ignat-Moser:16}.)
Recall that
\[
u_{a, d}^* = \sum_{n = 1}^N \gamma_n u_{a_n} \quad \text{and} \quad v_{a, d}^* = \sum_{n = 1}^N \gamma_n v_{a_n}
\]
for the functions constructed in Section \ref{sect:limiting_stray_field} (see \eqref{def_ub_vb_unconfined}
and the subsequent formulas). 
We also set $\mu_{a, d}^*(x_1) = v_{a, d}^*(x_1, 0)$. By \eqref{v_x1},
\[
\mu_{a, d}^*(x_1) = \sum_{n = 1}^N \gamma_n I(|x_1 - a_n|).
\]
Let $r \in (0, 1/2]$ with $r \le \rho(a)$ (where $\rho(a)$ is defined as in \eqref{def_rho_unconf}). For $n = 1, \dotsc, N$, let
\[
\lambda_n = \sum_{k \not= n} \gamma_k I(|a_k - a_n|).
\]
As $I$ is locally Lipschitz continuous in $(0, \infty)$, we find that
\begin{equation} \label{eqn:new71}
\left|\mu_{a, d}^*(x_1) - \lambda_n  - \gamma_n I(|x_1 - a_n|)\right| \le Cr
\end{equation}
for $x_1 \in [a_n - r, a_n + r]$. Also define
\[
\omega_n = \sum_{k \not= n} \gamma_k u_{a_k}(a_n, 0).
\]
Then by Proposition \ref{prop:pointwise_estimates}, 
\begin{equation} \label{eqn:new72}
|u_{a, d}^*(x) - \omega_n - \gamma_n w_0(x_1 - a_n, x_2)| \le Cr \log \frac{1}{r} \quad \text{ in $B_r^+(a_n, 0)$}
\end{equation}
and
\begin{equation} \label{eqn:new73}
|\nabla u_{a, d}^* - \gamma_n \nabla w_0(x_1 - a_n, x_2)| \le C \log \frac{1}{r}  \quad \text{in $B_r^+(a_n, 0)$}.
\end{equation}
We define $W_1(a, d)$ as in Section \ref{sect:further_properties} and
\[
W_2(a, d) = -\pi \sum_{n = 1}^N \gamma_n \lambda_n = -\pi \sum_{n = 1}^N \sum_{k \not= n} \gamma_k \gamma_n I(|a_k - a_n|).
\]
If $W$ is the function defined in Section \ref{subsect:unconfined}, then by Proposition \ref{prop:tail-tail_interaction}:
\be
\label{def_sum_W}
W(a, d) = W_1(a, d) + W_2(a, d) - \pi I_0 \sum_{n = 1}^N \gamma_n^2 = -W_1(a, d).
\ee

\subsubsection{A lower bound for the interaction energy}

We first want to prove that
\[
\liminf_{\epsilon \searrow 0} \left((\log \delta)^2 \inf_{M(a, d)} E_\epsilon - \frac{\pi \Gamma}{2} \log \frac{1}{\delta}\right) \ge \sum_{n = 1}^N e(d_n) + W(a, d).
\]

\paragraph{First step: use minimisers} Similarly to \cite[Proposition 1]{Ignat-Moser:18}, we conclude that $E_\epsilon$ has minimisers $m_\epsilon$
in $M(a, d)$. 
Clearly it suffices to consider these minimisers. 
We claim that
\[
\limsup_{\epsilon \searrow 0} \left((\log \delta)^2 E_\epsilon(m_\epsilon) - \frac{\pi\Gamma}{2} \log \frac{1}{\delta}\right) < \infty.
\]
In order to prove this, we first consider the case where
$a_n \in (-1, 1)$ for $n = 1, \dotsc, N$. Then we can apply the results from
our previous paper \cite{Ignat-Moser:16}, in particular Proposition 27, which states
that there exist a constant $C_0$ and $\hat{m}_\epsilon \in M(a, d)$ with $\hat{m}_{\epsilon 1} \equiv \cos \alpha$ outside
of $(-1, 1)$ such that
\[
\frac{\epsilon}{2} \int_{-1}^1 |\hat{m}_\epsilon'|^2 \, dx_1 + \frac{1}{2} \int_{\R_+^2} |\nabla \hat{u}_\epsilon|^2 \, dx \le \frac{\pi\Gamma}{2 \log \frac{1}{\delta}} + \frac{C_0}{(\log \delta)^2},
\]
where $\hat{u}_\epsilon \in \dot{H}^1(\R_+^2)$ solves \eqref{eqn:stray_field_potential},
\eqref{eqn:boundary_condition} for $\hat{m}_\epsilon$ instead of $m$.
Let $\hat{v}_\epsilon \in \dot{H}^1(\R_+^2)$ be the harmonic extension of $\hat{m}_{\epsilon 1} - \cos \alpha$
to $\R_+^2$ (so that $\hat{u}_\epsilon$ and $\hat{v}_\epsilon$ are conjugate harmonic
functions). Theorem 28 and Remark 31 in  \cite{Ignat-Moser:16}
imply (by the dyadic decomposition
argument of Struwe \cite{Struwe}) that for any $p \in [1, 2)$, the
following inequality holds true:
\[
\frac{\epsilon}{2} \int_{-1}^1 |\hat{m}_\epsilon'|^2 \, dx_1 + \|\nabla \hat{v}_\epsilon\|_{L^p(B_2^+(0))}^2 \le \frac{C}{(\log \delta)^2}.
\]
Standard trace theorems for Sobolev spaces then imply that
\[
\int_{-1}^1 (\hat{m}_{\epsilon 1} - \cos \alpha)^2 \, dx_1 \le \frac{C}{(\log \delta)^2}.
\]
Hence
\[
E_\epsilon(m_\epsilon) \le E_\epsilon(\hat{m}_\epsilon) \le \frac{\pi\Gamma}{2 \log \frac{1}{\delta}} + \frac{C}{(\log \delta)^2}.
\]
For any other $a \in A_N$, we may scale the domain such that we are in the above
situation and observe that the stray field energy does not change under such scaling. The exchange energy and the anisotropy energy will change
(one of them will decrease and the other increase), but only by a factor
depending on $a$. As both of them are of order $1/(\log \delta)^2$,
we still obtain an inequality of the same form, albeit for a different constant $C$.

\paragraph{Second step: prove convergence away from the walls}
By Theorem \ref{thm:stray_field} and Proposition \ref{prop:conjugate_harmonic_functions}, we have a sequence $\epsilon_k \searrow 0$ such that
the functions $w_k = u_{\epsilon_k} \log \frac{1}{\delta_k}$ satisfy $w_k \rightharpoonup w$ weakly
in $\dot{H}^1(\Omega_s(a)) \cap W_{\loc}^{1,p}(\R^2_+)$ for all $s > 0$ and $p\in [1,2)$, where $w \in u_{a, d}^* + \dot{H}^1(\R_+^2)$
and $\delta_k = \epsilon_k \log \frac{1}{\epsilon_k}$.
Moreover, by Theorem \ref{thm:stray_field}, we may choose this subsequence such that
$\mu_k = (m_{\epsilon_k 1} - \cos \alpha) \log \frac{1}{\delta_k}  \rightharpoonup \mu$
weakly in $L^2(\R)$ for some $\mu \in L^2(\R)$.

But we have in fact better convergence:
Lemma \ref{lem:higher_estimates} and Theorem \ref{thm:stray_field}
imply that $w_k \rightharpoonup w$ weakly in $H^2(B_R(0) \cap \Omega_s(a))$
for any $R, s > 0$ and $\mu_k \to \mu$ in the strong $L^\infty$ and weak $H^1$ sense in $(-R, R) \setminus \bigcup_{n = 1}^N (a_n - s, a_n + s)$
for any $s > 0$. (We can use Lemma \ref{lem:higher_estimates} here because of
Lemma \ref{lem:no_unaccounted_walls}.) Furthermore,
\[
\limsup_{k \to \infty} \left((\log \delta_k)^2 \epsilon_k \int_{(-R, R) \setminus \bigcup_{n = 1}^N (a_n - s, a_n + s)} (\phi_{\epsilon_k}'')^2 \, dx_1 \right) < \infty.
\]
By the Euler-Lagrange equation,
\[
u_\epsilon' + \cos \alpha - \cos \phi_\epsilon = \frac{\epsilon \phi_\epsilon''}{\sin \phi_\epsilon} \quad \text{ in $\R \setminus \{a_1, \dotsc, a_N\}$}.
\]
It follows that $w'_k(\blank, 0) - \mu_k \rightharpoonup 0$ in the distribution sense in $\R \setminus \{a_1, \dotsc, a_N\}$. 
Since by the trace theorem, $w'_k(\blank, 0) \rightharpoonup w'(\blank, 0)$ as distributions in $\R$,  and since $\mu_k \rightharpoonup \mu$ weakly in $L^2(\R)$,
we conclude that $w'(\blank, 0) = \mu$ in $\R \setminus \{a_1, \dotsc, a_N\}$. 
Clearly $\Delta w = 0$ in $\R_+^2$. If we write $\dd{w}{x_2}$ for the distribution on $\R$
such that for any $\eta \in C_0^\infty(\overline{\R_+^2})$,
\[
\int_{-\infty}^\infty \dd{w}{x_2} \eta \, dx_1 = - \int_{\R_+^2} \nabla w \cdot \nabla \eta \, dx,
\]
then the equations $\dd{w_k}{x_2} = - \mu_k'$ and the weak $W_{\loc}^{1, p}$-convergence for $p < 2$ imply that
$\dd{w}{x_2} = -\mu'$ in the sense of distributions. 

We claim that only the function $u_{a, d}^*$ has these properties. In order to prove this claim,
set $f = w - u_{a, d}^*$ and $h = \mu - \mu_{a, d}^*$. Then $f \in \dot{H}^1(\R_+^2)$, $h\in L^2(\R)$, and $\Delta f = 0$
in $\R_+^2$. Moreover, $f' = h$ in $\R \setminus \{a_1, \dotsc, a_N\}$. But since
$f(\blank, 0) \in \dot{H}^{1/2}(\R)$ cannot have any jumps, this implies that $f' = h$ in all of $\R$.
We also know that $\dd{f}{x_2} = -h'$ on $\R \times \{0\}$
in the sense that
\[
\int_{-\infty}^\infty h \eta' \, dx_1 = - \int_{\R_+^2} \nabla f \cdot \nabla \eta \, dx
\]
for any $\eta \in C_0^\infty(\overline{\R_+^2})$. If $g$ denotes the conjugate harmonic
function to $f$ such that $\nabla g = \nabla^\perp f$, then
\[
\begin{split}
\int_{\R \times \{0\}} h \eta' \, dx_1 & = \int_{\R_+^2} \nabla^\perp g \cdot \nabla \eta \, dx = \lim_{s \searrow 0} \int_{\R \times (s, \infty)} \nabla^\perp g \cdot \nabla \eta \, dx = -\lim_{s \searrow 0} \int_{-\infty}^\infty \dd{g}{x_1}(x_1, s) \eta(x_1, s) \, dx_1 \\
& = \lim_{s \searrow 0} \int_{-\infty}^\infty g(x_1, s) \dd{\eta}{x_1}(x_1, s) \, dx_1 = \int_{\R \times \{0\}} g \eta' \, dx_1.
\end{split}
\]
Thus after adding a suitable
constant, we find that $g = h$ and $\dd{}{x_2} g = g$ on $\R \times \{0\}$. As $g$ is harmonic,
the last identity, combined with an integration by parts, implies that
$\int_{\R^2_+} |\nabla g|^2\, dx=-\int_{\R} g^2\, dx_1$.  That is, $g = 0$ in $\R^2_+$, and hence
$w = u_{a, d}^*$ and $\mu = \mu_{a, d}^*$. 
As we have thus identified unique limits, the above convergence holds in fact
not just for the sequences $w_k$ and $\mu_k$, but for the full families
$u_\epsilon |\log \delta|$ and $(m_{\epsilon 1} - \cos \alpha)|\log \delta|$, in the same sense, as $\epsilon \searrow 0$.

We conclude in particular that
\[
\int_{\Omega_r(a)} |\nabla u_{a, d}^*|^2 \, dx \le \liminf_{\epsilon \searrow 0} \left((\log \delta)^2 \int_{\Omega_r(a)} |\nabla u_\epsilon|^2 \, dx\right)
\]
and
\[
\int_{-\infty}^\infty \left(\mu_{a, d}^*\right)^2 \, dx_1 \le \liminf_{\epsilon \searrow 0} \left((\log \delta)^2 \int_{-\infty}^\infty (m_{\epsilon 1} - \cos \alpha)^2 \, dx_1\right).
\]
Note also that for every $r \in (0, 1/2]$ with $r \le \rho(a)$,
\[
\left|m_{\epsilon 1}(a_n \pm r) - \cos \alpha - \frac{\gamma_n I(r) + \lambda_n}{\log \frac{1}{\delta}}\right| \le \frac{Cr + o(1)}{\log \frac{1}{\delta}}
\]
by \eqref{eqn:new71} and the above convergence.

\paragraph{Third step: rescale the cores}
Fix $n\in \{1, \dots, N\}$, $r\in (0, \rho(a)]$ and consider the rescaled functions
$\tilde{m}_\epsilon: \R \to \mathbb{S}^1$ and $\tilde{u}_\epsilon:\R^2_+\to \R$ defined by
\begin{align*}
\tilde{m}_\epsilon(x_1)  = m_\epsilon(rx_1 + a_n), \quad \tilde{u}_\epsilon(x) & = u_\epsilon(r x_1 + a_n, r x_2), \quad
\tilde{\epsilon}  = \frac{\epsilon}{r}.
\end{align*}
Then the arguments from the first half of p.\ 475 in \cite{Ignat-Moser:16} apply, the
only difference being that some of the right-hand sides of the estimates
need to be multiplied by $\log \frac{1}{r}$, due to the appearance of this
factor in \eqref{eqn:new73}. We then apply \cite[Corollary 21]{Ignat-Moser:16} to $\mu = d_n \tilde{m}_\epsilon$
with
\[
\gamma = d_n \gamma_n, \quad \eta =C r \log \frac{1}{r} + o(1), \quad \text{and} \quad \zeta = d_n(\lambda_n + \gamma_n I(r)) + Cr + o(1).
\]
This eventually gives the inequality
\[
\begin{split}
\lefteqn{(\log \delta)^2 \left(\epsilon \int_{a_n - r}^{a_n + r} |m_\epsilon'|^2 \, dx_1 + \int_{B_r^+(a_n, 0)} |\nabla u_\epsilon|^2 \, dx\right) - \pi \gamma_n^2 \left(\log \frac{1}{\delta} - \log \frac{1}{r}\right)} \hspace{2cm} \\
& \ge 2 e(d_n) - 2\pi \gamma_n \lambda_n - 2\pi \gamma_n^2 I(r) + 2\pi \gamma_n^2 \log \frac{1}{r} - Cr \log \frac{1}{r} - o(1) \\
& \ge 2e(d_n) - 2\pi \gamma_n \lambda_n - 2\pi \gamma_n^2 I_0 - Cr \log \frac{1}{r} - o(1).
\end{split}
\]
(The last inequality comes from Lemma \ref{lem:I_is_logarithmic}.)

\paragraph{Fourth step: combine the estimates}
 We can now adapt the arguments on p.\ 476 in \cite{Ignat-Moser:16} to our unconfined model,
 recalling identity \eqref{def_sum_W}. The desired lower bound follows.

\subsubsection{An upper bound for the interaction energy}
\label{sec:upper}

We now want to prove the inequality
\[
\limsup_{\epsilon \searrow 0} \left((\log \delta)^2 \inf_{M(a, d)} E_\epsilon - \frac{\pi \Gamma}{2} \log \frac{1}{\delta}\right) \le \sum_{n = 1}^N e(d_n) + W(a, d).
\]

\paragraph{First step: glue energy minimising cores into the tail profile}
Define
\[
\kappa_n^r = \frac{\lambda_n + \gamma_nI(r)}{\gamma_n \log \frac{1}{\delta}}.
\]
Then we define certain profiles $m_\epsilon$ and approximate stray field potentials
$\tilde{u}_\epsilon$ with the same formulas as on p.\ 477 of \cite{Ignat-Moser:16}.
The functions $\mu_{a, d}^*$ and $u_{a, d}^*$ appearing in this construction
are as in Section \ref{subsubsection:preparation}, while $\mu_{\epsilon/r}^n$ and $u_{\epsilon/r}^n$
are as in the other paper \cite{Ignat-Moser:16}.
Furthermore, the function $u_\epsilon \colon \R_+^2 \to \R$ is the solution of
\eqref{eqn:stray_field_potential}, \eqref{eqn:boundary_condition} for $m_\epsilon$
instead of $m$.
In the rest of the proof, we estimate the energy of $m_\epsilon$, showing that
it provides the desired bounds.

\paragraph{Second step: estimate the magnetostatic energy in terms of $\tilde{u}_\epsilon$}
The Laplacian $\Delta \tilde{u}_\epsilon$ satisfies the same formula as in our previous paper \cite{Ignat-Moser:16}.
But in the subsequent estimates, the inequalities (72) and (73) of \cite{Ignat-Moser:16} have to replaced by \eqref{eqn:new72}
and \eqref{eqn:new73}. This means that the additional factor $\log \frac1r$ appears in some of the estimates. 
This eventually leads to the conclusion that
\be
\label{89}
\|\Delta \tilde{u}_\epsilon\|_{L^p(\R_+^2)} \le \frac{C r^{2/p - 1} \log \frac{1}{r} + o(1)}{\log \frac{1}{\delta}}
\ee
for any fixed $p < \infty$.
We can still use \cite[Theorem 22]{Ignat-Moser:16} to conclude that
\[
\left|\mu_\epsilon^n(x_1) - \cos \alpha + \frac{\gamma_n \log |x_1|}{\log \frac{1}{\delta}}\right| \le \frac{o(1)}{\log \frac{1}{\delta}}
\]
for any $x_1 \in [- \frac{3}{4}, - \frac{1}{2}] \cup [\frac{1}{2}, \frac{3}{4}]$.
We now combine this estimate with \eqref{eqn:new71} and Lemma \ref{lem:I_is_logarithmic}.
For $x_1 \in [a_n - \frac{3r}{4}, a_n - \frac{r}{2}] \cup [a_n + \frac{r}{2}, a_n + \frac{3r}{4}]$, this gives rise to:
\begin{equation} \label{eqn:new86}
\begin{split}
\lefteqn{\left|(1 - \kappa_n^r)\mu_{\epsilon/r}^n\left(\frac{x_1 - a_n}{r}\right) + d_n \kappa_n^r - \cos \alpha - \frac{\mu_{a, d}^*(x_1)}{\log \frac{1}{\delta}}\right|} \\
& = \left|(1 - \kappa_n^r)\left(\cos \alpha - \frac{\gamma_n \log \frac{|x_1 - a_n|}{r}}{\log \frac{1}{\delta}}\right) + d_n \kappa_n^r - \cos \alpha - \frac{\lambda_n + \gamma_n I(|x_1 - a_n|)}{\log \frac{1}{\delta}}\right|+ \frac{Cr + o(1)}{\log \frac{1}{\delta}} \\
& = \left|\kappa_n^r \gamma_n - \frac{\gamma_n \log \frac{1}{r}}{\log \frac{1}{\delta}} - \frac{\lambda_n + \gamma_n I_0}{\log \frac{1}{\delta}}\right| + \frac{Cr + o(1)}{\log \frac{1}{\delta}}  = \frac{Cr + o(1)}{\log \frac{1}{\delta}}.
\end{split}
\end{equation}

The next arguments are similar to our previous paper \cite{Ignat-Moser:16} again,
but with two adaptations: first, we will need to multiply some of the terms
in the estimates with $\log \frac{1}{r}$, because we use \eqref{eqn:new72}
and \eqref{eqn:new73} instead of (72) and (73) in  \cite{Ignat-Moser:16}. Second, when we restrict
the functions $u_\epsilon$ and $\tilde{u}_\epsilon$ to a half-disk in $\R_+^2$,
we need to make sure that all the N\'eel walls are included.

Therefore, we fix $R \ge 1$ such that $a_1, \dots, a_N \in B_{R/2}(0)$. We now
have the inequality
\begin{equation} \label{eqn:estimate_for_Laplacian}
\|\Delta(u_\epsilon - \tilde{u}_\epsilon)\|_{L^p(\R_+^2)} \le \frac{C r^{2/p - 1} \log \frac{1}{r} + o(1)}{\log \frac{1}{\delta}}
\end{equation}
for an arbitrary fixed $p \in (1, 2)$, and still
\begin{equation} \label{eqn:88}
\left\|\dd{}{x_2}(u_\epsilon - \tilde{u}_\epsilon)\right\|_{L^\infty(\R)} \le \frac{C + o(1)}{\log \frac{1}{\delta}}.
\end{equation}
The support of $\Delta(u_\epsilon - \tilde{u}_\epsilon)$ is contained in $\bigcup_{n = 1}^N B_r^+(a_n, 0)$
and the support of $\dd{}{x_2}(u_\epsilon - \tilde{u}_\epsilon)$ is contained in $\bigcup_{n = 1}^N (a_n - r, a_n + r)$.
Thus we conclude that
\[
\|\Delta(u_\epsilon - \tilde{u}_\epsilon)\|_{\mathcal{M}(\R^2)} \le \frac{Cr \log \frac{1}{r} + o(1)}{\log \frac{1}{\delta}},
\]
where $\mathcal{M}(\R^2)$ is the space of Radon measures on $\R^2$. Hence
by well-known estimates for the Poisson equation with source term in $\mathcal{M}(\R^2)$, we find the estimate
\[
\|\nabla(u_\epsilon - \tilde{u}_\epsilon)\|_{L^q(B_R^+(0))} \le R^{2/q - 1} \frac{C \log \frac{1}{r} + o(1)}{\log \frac{1}{\delta}}
\]
for an arbitrary fixed $q \in [1, 2)$.
By the arguments in \cite{Ignat-Moser:16}, we have the inequality
\[
\|\nabla \tilde{u}_\epsilon\|_{L^q(B_R^+(0))} \le R^{2/q - 1} \frac{C \log \frac{1}{r} + o(1)}{\log \frac{1}{\delta}}.
\]
Hence
\begin{equation} \label{eqn:estimate_for_gradient}
\|\nabla u_\epsilon\|_{L^q(B_R^+(0))} \le R^{2/q - 1} \frac{C \log \frac{1}{r} + o(1)}{\log \frac{1}{\delta}}.
\end{equation}
Setting
\[
\bar{u}_\epsilon = \fint_{B_R^+(0)} u_\epsilon \, dx,
\]
we now compute
\begin{equation} \label{eqn:estimate_of_nabla_u_epsilon}
\int_{\R_+^2} |\nabla u_\epsilon|^2 \, dx = \int_{\R_+^2} \nabla \tilde{u}_\epsilon \cdot \nabla u_\epsilon \, dx - \int_{(-R, R) \times \{0\}} (u_\epsilon - \bar{u}_\epsilon) \dd{}{x_2} (u_\epsilon - \tilde{u}_\epsilon) \, dx_1 
+ \int_{\R_+^2} (u_\epsilon - \bar{u}_\epsilon) \Delta \tilde{u}_\epsilon \, dx.
\end{equation}
By the H\"older inequality, the trace theorem, and \eqref{eqn:estimate_for_gradient},  we obtain
\[
\begin{split}
\|u_\epsilon - \bar{u}_\epsilon\|_{L^1(\bigcup_{n = 1}^N (a_n - r, a_n + r))} & \le Cr^{2 - 2/q} \|u_\epsilon - \bar{u}_\epsilon\|_{L^{q/(2 - q)}(-R, R)} \\
& \le C r^{2 - 2/q} \|\nabla u_\epsilon\|_{L^q(B_R^+)}  \le Cr^{2 - 2/q} R^{2/q - 1} \frac{C \log \frac{1}{r} + o(1)}{\log \frac{1}{\delta}},
\end{split}
\]
and since we have inequality \eqref{eqn:88}, we obtain
\[
- \int_{(-R, R) \times \{0\}} (u_\epsilon  - \bar{u}_\epsilon) \dd{}{x_2} (u_\epsilon - \tilde{u}_\epsilon) \, dx_1  \le R^{2/q - 1} \frac{C r^{2 - 2/q} \log \frac{1}{r} + o(1)}{(\log \delta)^2}.
\]
Moreover, as the support of $\Delta \tilde{u}_\epsilon$ is included in $\bigcup_{n = 1}^N B_r^+(a_n, 0)\subset B_R(0)$,
the H\"older inequality and Sobolev embedding theorem, combined with \eqref{89} and \eqref{eqn:estimate_for_gradient}, yield
\[
\int_{\R_+^2} (u_\epsilon - \bar{u}_\epsilon) \Delta \tilde{u}_\epsilon \, dx \le R^{2 - 2/p} \frac{C r^{2/p - 1} (\log r)^2 + o(1)}{(\log \delta)^2}.
\]
For $p = \frac{4}{3}$ and $q = \frac{2p}{3p - 2} = \frac{4}{3}$,
identity \eqref{eqn:estimate_of_nabla_u_epsilon} and the above inequalities imply that
\[
\int_{\R_+^2} |\nabla u_\epsilon|^2 \, dx \le \int_{\R_+^2} |\nabla \tilde{u}_\epsilon|^2 \, dx + \sqrt{R} \frac{C \sqrt{r} (\log r)^2 + o(1)}{(\log \delta)^2}.
\]
Hence
\begin{equation} \label{eqn:new90}
E_\epsilon(m_\epsilon) \le \frac{1}{2} \int_{-\infty}^\infty \left(\epsilon |m_\epsilon'|^2 + \left(m_{\epsilon 1} - \cos \alpha\right)^2\right) \, dx_1 + \frac{1}{2} \int_{\R_+^2} |\nabla \tilde{u}_\epsilon|^2 \, dx 
+ \sqrt{R} \frac{C \sqrt{r} (\log r)^2 + o(1)}{(\log \delta)^2}.
\end{equation}

\paragraph{Third step: estimate $\|\nabla \tilde{u}_\epsilon\|_{L^2(\R_+^2)}$}
Identity \cite[(91)]{Ignat-Moser:16} remains true here. The subsequent calculations in \cite{Ignat-Moser:16} still work,
but instead of \cite[(92)]{Ignat-Moser:16} we now obtain
\[
\|\nabla \tilde{u}_\epsilon\|_{L^2(B_r^+(a_n, 0))}^2 \le (1 - \kappa_n^r)^2 \|\nabla u_{\epsilon/r}^n\|_{L^2(B_1^+(0))}^2 + \frac{Cr \log \frac{1}{r} + o(1)}{(\log \delta)^2}.
\]
By the above definition of $\kappa_n^r$, the same calculations as on p.\ 481 of \cite{Ignat-Moser:16}
now yield the estimate
\begin{multline} \label{eqn:new93}
\int_{\R_+^2} |\nabla \tilde{u}_\epsilon|^2 \, dx \le \frac{1}{(\log \delta)^2} \int_{\Omega_r(a)} |\nabla u_{a, d}^*|^2 \, dx + \sum_{n = 1}^N \int_{B_1^+(0)} |\nabla u_{\epsilon/r}^n|^2 \, dx \\
- \sum_{n = 1}^N \frac{2\pi \gamma_n (\lambda_n + \gamma_n I(r))}{(\log \delta)^2} + \frac{Cr \log \frac{1}{r} + o(1)}{(\log \delta)^2}. \end{multline}

\paragraph{Fourth step: estimate the exchange energy}
The calculations here are the same as in \cite{Ignat-Moser:16}, but as inequality (73) there
is replaced by \eqref{eqn:new73} in the current article, the first inequality on p.\ 482 in \cite{Ignat-Moser:16} becomes
\[
\left\|\frac{1 - \kappa_n^r}{r} \frac{d\mu_{\epsilon/r}^n}{dx_1} \left(\frac{x_1 - a_n}{r}\right) - \frac{\frac{d}{dx_1}\mu_{a, d}^*(x_1)}{\log \frac{1}{\delta}}\right\|_{L^2(T_n^r)} \le \frac{C \sqrt{r} \log \frac{1}{r} + o(1)}{\log \frac{1}{\delta}}.
\]
(Incidentally, the set $T_n^r$ should be defined as $T_n^r = (a_n - \frac{3r}{4}, a_n - \frac{r}{2}) \cup (a_n + \frac{r}{2}, a_n + \frac{3r}{4})$, not as stated in \cite{Ignat-Moser:16}.) 
Since the core profiles $\mu_{\epsilon/r}^n$ used here are exactly the same as in \cite{Ignat-Moser:16},
we can use \cite[Theorem 17]{Ignat-Moser:16} to estimate them. Furthermore, we have
inequality \eqref{eqn:new86}. Thus we find that
\begin{equation} \label{eqn:new94}
\frac{\epsilon}{2} \int_{-\infty}^\infty |m_\epsilon'|^2 \, dx_1 = \frac{\epsilon}{2r} \sum_{n = 1}^N \int_{-1}^1 \frac{\left(\frac{d}{dx_1} \mu_{\epsilon/r}^n\right)^2}{1 - (\mu_{\epsilon/r}^n)^2} \, dx_1 + \frac{o(1)}{(\log \delta)^2}.
\end{equation}

\paragraph{Extra step: estimate the anisotropy energy}
For the problem studied here, we need to estimate the anisotropy energy as well, of course.
Fortunately, this is rather straightforward. Owing to inequality \eqref{eqn:new86},
we can immediately conclude that
\begin{equation} \label{eqn:extra_estimate}
\int_{-\infty}^\infty (m_{\epsilon 1} - \cos \alpha)^2 \, dx_1 = \frac{1}{(\log \delta)^2} \int_{-\infty}^\infty (\mu_{a, d}^*)^2 \, dx_1 + \frac{Cr + o(1)}{(\log \delta)^2}.
\end{equation}
Here we also use the fact that $\varphi_n$ (in the definition of $m_{\epsilon 1}$ in \cite{Ignat-Moser:16}) is supported in the set 
$\bigcup_{n=1}^N T_n^r$ of
measure $Nr/2$ and
$\|\mu_{a, d}^*\|_{L^1(\R)} \le C$ by Lemma \ref{lem:I_is_logarithmic}. 

Combining \eqref{eqn:new90}, \eqref{eqn:new93}, \eqref{eqn:new94}, and \eqref{eqn:extra_estimate}, we obtain
(using the notation of \cite{Ignat-Moser:16}): 
\begin{multline*}
E_\epsilon(m_\epsilon) \le \frac{1}{ 2(\log \delta)^2} \int_{\Omega_r(a)} |\nabla u_{a, d}^*|^2 \, dx + \sum_{n = 1}^N \inf_{M_{|\gamma_n|}} E_{\epsilon/r}^{|\gamma_n|} \\
- \sum_{n = 1}^N \frac{\pi \gamma_n (\lambda_n + \gamma_n I(r))}{(\log \delta)^2} + \frac{1}{2(\log \delta)^2} \int_{-\infty}^\infty (\mu_{a, d}^*)^2 \, dx_1 + \sqrt{R}\frac{C \sqrt{r} (\log r)^2 + o(1)}{(\log \delta)^2}.
\end{multline*}

\paragraph{Fifth step: estimate the core energy}
This is exactly the same as on p.\ 483 of \cite{Ignat-Moser:16}.

\paragraph{Sixth step: combine the estimates}
It follows that
\[
\begin{split}
(\log \delta)^2 E_\epsilon(m_\epsilon) & \le \frac{1}{2} \int_{\Omega_r(a)} |\nabla u_{a, d}^*|^2 \, dx + \frac{1}{2} \int_{-\infty}^\infty (\mu_{a, d}^*)^2 \, dx_1 \\
& \quad + \frac{\pi \Gamma}{2} \log \frac{1}{\delta} - \frac{\pi \Gamma}{2} \log \frac{1}{r} + \sum_{n = 1}^N e(d_n) + W_2(a, d)  - \pi I_0 \sum_{n = 1}^N \gamma_n^2 \\
& \quad + C \sqrt{R r} (\log r)^2 + o(1).
\end{split}
\]
Thus by \eqref{w1_unconfined} and \eqref{def_sum_W}, 
$$
(\log \delta)^2 E_\epsilon(m_\epsilon) \le W(a, d) + \frac{\pi \Gamma}{2} \log \frac{1}{\delta} + \sum_{n = 1}^N e(d_n) + C \sqrt{R r} (\log r)^2 + o(1).
$$
When we pass to the limit $r\to 0$, this proves the upper bound and completes the proof of Theorem~\ref{thm:renormalised_unconfined}.

\subsection{Renormalised energy, separation and $\Gamma$-convergence}
\label{sec:last}

Now that the proof of Theorem \ref{thm:renormalised_unconfined} is complete, the theory
for the unconfined model is at the same stage as it has been developed for the confined model in our previous paper
\cite{Ignat-Moser:16}. As discussed in the introduction, the model also permits counterparts to
the further results from Subsection~\ref{subsect:confined}. Fortunately, the proofs require few
fundamentally new ideas, and therefore, we can keep this discussion relatively short.

We also prove Proposition \ref{pro:min_W_unconf} here. For this purpose, we analyse 
the renormalised energy in the unconfined case. First, we need the following result, which is  similar to Lemma \ref{lem:blow_up_of_logarithmic_NEW}. For $a \in A_N$, recall the quantity $\rho(a)$, defined in \eqref{def_rho_unconf}, as it will appear in this statement.

\begin{lemma} \label{lem:blow_up_of_logarithmic_function} 
Let $N\geq 2$ and $A_{k\ell} \in \R$ for $1 \le k < \ell \le N$, and let $f \colon A_N \to \R$ be defined by 
\[
f(a) = -\sum_{1\leq k < \ell\leq N} A_{k\ell} I (a_\ell - a_k), \quad a\in A_N.
\]
\begin{enumerate}
\item \label{item:divergence} 
If $\sum_{K \le k < \ell \le L} A_{k\ell} < 0$ for all $K, L \in \{1, \dotsc, N\}$ with $K < L$, then $\inf_{A_N} f>-\infty$ and $f(a) \to \infty$ as $ \rho(a) \to 0$. 

\item \label{item:negative_sums} If there exists $K<L$ such that $\sum_{K \le k < \ell \le L} A_{k\ell} > 0$, then there exists a sequence $(a^{(i)})_{i \in \N}$ in $A_N$ such that $ \rho(a^{(i)}) \to 0$ and $f(a^{(i)}) \to -\infty$ as $i\to \infty$. \end{enumerate}
\end{lemma}

\begin{proof} We follow the arguments in the proof of Lemma \ref{lem:blow_up_of_logarithmic_NEW}, pointing out only the differences (which are concerned with the possibility that $\rho(a)$ may be unbounded).

For statement \ref{item:divergence}, we proceed, as before, by induction. For $N=2$,
the statement follows from the properties of $I$ proved in Lemma \ref{lem:I_is_logarithmic}. 
In the induction step, we consider an arbitrary sequence $(a^{(i)})_{i \in \N}$ in $A_N$ (not necessarily satisfying
$ \rho(a^{(i)}) \to 0$, because we also make a statement about the infimum
of $f$ over $A_N$) and we distinguish two cases.

\setcounter{caseno}{0}
\begin{subproof}{Case \case:  $\limsup_{i \to \infty} \big(a_N^{(i)}-a_1^{(i)}) = \infty$}
We partition $\{1, \dotsc, N\}$ into $\Lambda_1^{(i)}$ and $\Lambda_2^{(i)}$ with
$\Lambda_1^{(i)}=\{1, \dotsc, n^{(i)}\}$, $\Lambda_2^{(i)}=\{n^{(i)}+1, \dotsc, N\}$ such that
$a^{(i)}_{n^{(i)}+1}-a^{(i)}_{n^{(i)}}\to \infty$ as $i\to \infty$. 
We may assume that the inequalities $2 \le n^{(i)} \le N - 2$
hold for all values of $i$ or for none; if necessary, we pass to a subsequence with this property.
In the first case, we conclude that 
$|\Lambda_1^{(i)}| \ge 2$ and $|\Lambda_2^{(i)}|\ge 2$.
Otherwise, they will satisfy $|\Lambda_1^{(i)}| = 1$ or $|\Lambda_2^{(i)}| = 1$.

\setcounter{subcaseno}{0}
\begin{subproof}{Case \subcase: \label{Subcasei1} $|\Lambda_1^{(i)}|, |\Lambda_2^{(i)}|\geq 2$}  Then
$$f(a^{(i)})=-\sum_{1\leq k<\ell\leq n^{(i)}} A_{k\ell} I(a^{(i)}_\ell-a^{(i)}_k)
-\sum_{n^{(i)}< k<\ell\leq N} A_{k\ell} I(a^{(i)}_\ell-a^{(i)}_k)-\sum_{k\in \Lambda_1^{(i)}, \ell\in \Lambda_2^{(i)}} A_{k\ell} I(a^{(i)}_\ell-a^{(i)}_k).$$
Since $2\leq |\Lambda_1^{(i)}|, |\Lambda_2^{(i)}|<N$, by induction, we know that the first two terms are uniformly bounded below. 
The last term converges to $0$ by Lemma \ref{lem:I_is_logarithmic}. Moreover, if 
$ \rho(a^{(i)}) \to 0$, then at least one of $\Lambda_1^{(i)}$ or $\Lambda_2^{(i)}$ has
 points that collide as $i\to \infty$. Thus, by induction, $f(a^{(i)})\to \infty$ as $i\to \infty$.
\end{subproof}

\begin{subproof}{Case \subcase: $|\Lambda_1^{(i)}|=1$ or $|\Lambda_2^{(i)}|=1$}
Without loss of generality, we may assume that  $|\Lambda_1^{(i)}|=1$ and $N>|\Lambda_2^{(i)}|\geq 2$. Then
$$f(a^{(i)})=
-\sum_{2\leq k<\ell\leq N} A_{k\ell} I(a^{(i)}_\ell-a^{(i)}_k)-\sum_{\ell=2}^N A_{k\ell} I(a^{(i)}_\ell-a^{(i)}_1).$$
The conclusion follows as in Case \ref{Subcasei1}.
\end{subproof}
\end{subproof}

\begin{subproof}{Case \case: $\limsup_{i \to \infty} \big(a_N^{(i)}-a_1^{(i)}) <\infty$}
As $f$ is translation invariant, we may assume that
$a_1^{(i)} = 0$ for all $i \in \N$. Then we use the arguments in Lemma \ref{lem:blow_up_of_logarithmic_NEW},
since $I(t)$ behaves like $\log \frac1t$ as $t\to 0$.
\end{subproof}

For statement \ref{item:negative_sums}, we use the same arguments as in Lemma \ref{lem:blow_up_of_logarithmic_NEW}.
\end{proof}

For the proof of Proposition \ref{pro:min_W_unconf}, we require the following lemma.

\begin{lemma} \label{lem:I'(2t)/I'(t)}
The function $(0, \infty) \to (\frac{1}{8}, \frac{1}{2})$, $t \mapsto \frac{I'(2t)}{I'(t)}$,
is well-defined and bijective.
\end{lemma}

\begin{proof}
By \eqref{zez}, a change of variables implies
\begin{align*}
2 I'(2t)-I'(t)&=\int_0^\infty \frac{s^2}{s^2+1}(e^{-st}-2e^{-2st})\, ds=\int_0^\infty \bigg(\frac{s^2}{s^2+1}-\frac{s^2}{s^2+4}\bigg)e^{-st}\, ds>0,\\
8 I'(2t)-I'(t)&=\int_0^\infty \frac{s^2}{s^2+1}(e^{-st}-8e^{-2st})\, ds=\int_0^\infty \bigg(\frac{s^2}{s^2+1}-\frac{4s^2}{s^2+4}\bigg)e^{-st}\, ds<0.
\end{align*}
Since $I'<0$ in $(0, \infty)$, then
$\frac18<\frac{I'(2t)}{I'(t)}<\frac12$ in $(0, \infty)$. By \eqref{zez}, another change of variables gives
\begin{align*}
t^3 I'(t)&=-t^3\int_0^\infty \frac{s^2}{s^2+1}e^{-st} \, ds=- \int_0^\infty \frac{t^2s^2}{s^2+t^2}e^{-s} \, ds\to - \int_0^\infty s^2 e^{-s} \, ds=-2 \quad \textrm{as $t\to \infty$},\\
t I'(t)&=-t\int_0^\infty \frac{s^2}{s^2+1}e^{-st} \, ds=- \int_0^\infty \frac{s^2}{s^2+t^2}e^{-s} \, ds\to - \int_0^\infty e^{-s} \, ds=-1 \quad \textrm{as $t\to 0$},
\end{align*}
 by the dominated convergence theorem. In particular, we have
$\lim_{t\to 0} \frac{I'(2t)}{I'(t)}=\frac12$ and $\lim_{t\to \infty} \frac{I'(2t)}{I'(t)}=\frac18$.
By the intermediate value theorem, the function $t \mapsto I'(2t)/I'(t)$ attains every value in $(\frac{1}{8}, \frac{1}{2})$.

In order to prove injectivity, we consider the function $f(t) = I'(2t) - qI'(t)$ for a fixed $q \in (\frac{1}{8}, \frac{1}{2})$.
We also consider $t_0 > 0$ with $f(t_0) = 0$. By  computations similar to the above,
we may write
\[
f(t_0) = \int_0^\infty g(s) e^{-st_0} \, ds,
\]
where
\[
g(s) =\frac{qs^2}{s^2 + 1} - \frac{s^2}{2(s^2 + 4)} = \frac{s^2}{2(s^2 + 1)(s^2 + 4)} \left((8q - 1) - (1 - 2q)s^2\right).
\]
We note that $8q - 1 > 0$ and $1 - 2q > 0$. Hence there exists $\sigma > 0$ such that
$g(s) > 0$ for $s < \sigma$ and $g(s) < 0$ for $s > \sigma$. Therefore $(\sigma - s) g(s) > 0$ for all $s \neq \sigma$.
Next we compute
\[
f'(t_0) = -\int_0^\infty sg(s) e^{-st_0} \, ds > -\sigma \int_0^\infty g(s) e^{-st_0} \, ds = -\sigma f(t_0)=0.
\]
To summarise, if $f(t_0) = 0$, then $f'(t_0) > 0$. Of course it follows that $f$ can have at most one zero.
Thus for any $q \in (\frac{1}{8}, \frac{1}{2})$, there exists exactly one $t_0 > 0$ such that
$I'(2t_0)/I'(t_0) = q$.
\end{proof}

\begin{proof}[Proof of Proposition \ref{pro:min_W_unconf}]
The first two statements are consequences of Lemmas \ref{lem:coefficients_for_alternating_signs} and \ref{lem:blow_up_of_logarithmic_function}.

For statement \ref{item:N=3}, if $N=3$, assume that $a\in A_3$ is a 
critical point of $W(\cdot, d)$. Then $a$ satisfies two equations: $I'(a_2-a_1)=I'(a_3-a_2)$ and $I'(a_3-a_1)=-\frac{\gamma_2}{\gamma_1}I'(a_3-a_2)$. As $I'$ is increasing (see Lemma \ref{lem:I_is_logarithmic}), we deduce that $a_2-a_1=a_3-a_2 \eqqcolon t$ (i.e., the points $a$ are equidistributed). Thus, the existence of $a$ is equivalent to the existence of a solution of $I'(2t)=cI'(t)$ with $c=-\frac{\gamma_2}{\gamma_1}$.

Suppose that $d_1 = 1$. By Lemma \ref{lem:I'(2t)/I'(t)}, the equation $I'(2t)=cI'(t)$ has a solution if, and only if,
$\frac{1}{8} < c < \frac{1}{2}$. Moreover, this solution is unique. Hence $W(\blank, d)$ has a critical point exactly
under these conditions, and the critical point is unique.
Since $c =-\frac{\gamma_2}{\gamma_1}=\frac{1+\cos \alpha}{1-\cos \alpha}$,
the above inequalities hold true if, and only if, $-\frac{7}{9} < \cos \alpha < -\frac{1}{3}$,
which corresponds to the range of angles $\alpha$ given in Proposition~\ref{pro:min_W_unconf}.
For $d_1 = -1$, the arguments are analogous.
Moreover, these critical points are \emph{not} minimisers. Indeed, $W(\blank, d)$ does not have any minimiser, since
statement \ref{item:inf_-infty} implies that $\inf_{A_3} W(\cdot, d)=-\infty$ for $d_1 \cos \alpha<-\frac13$.

For statement \ref{item:small_wall_outermost}, we first consider the case $N=2$ and $\alpha\in (0, \pi)$. By Lemma \ref{lem:I_is_logarithmic}, we know that $I'<0$, so $I$ does not have any critical point and neither does $W(\blank, d)$.
In the general case, we only consider the case where $d_1\cos \alpha \geq 0$,
as the arguments are the same in the other case. This is equivalent to $0<|\gamma_1|\leq |\gamma_2|$. 

\setcounter{caseno}{0}
\begin{subproof}{Case \case: $N$ is odd} Then, given $a\in A_N$, we compute
$$\frac1{\pi \gamma_1} \frac{\partial W(\blank, d)}{\partial a_1}(a)=\big(\gamma_2 I'(a_2-a_1)+\gamma_1I'(a_3-a_1)\big)+\dots +\big(\gamma_2 I'(a_{N-1}-a_1)+\gamma_1I'(a_N-a_1)\big).$$
According to Lemma \ref{lem:I_is_logarithmic}, we know that 
$I'(a_n-a_1)<I'(a_{n+1}-a_1)<0$ for $2\leq n\leq N-1$.
The assumption $|\gamma_1|\leq |\gamma_2|$  then implies that either all terms of the form
$\gamma_2 I'(a_n-a_1)+\gamma_1I'(a_{n + 1}-a_1)$ (for $n$ even) give a positive contribution to the above sum,
or all give a negative contribution. In particular, $\frac{\partial W(a, d)}{\partial a_1}\neq 0$ for every $a\in A_N$, and $W(\blank, d)$ can have no critical points.
\end{subproof}
\begin{subproof}{Case \case: $N$ is even} In this case, we have one additional term at the end of the above sum, which is  $\gamma_2 I'(a_N-a_1)$. It has the same sign as the sum of  the other terms. Thus, the conclusion holds.
\end{subproof}
\end{proof}

We now turn to the separation and $\Gamma$-convergence result in the unconfined case
 (the counterparts of Theorem \ref{thm:compactness_and_separation_confined} and Corollary \ref{cor:gamma_confined}).
Writing $m = (\cos \phi, \sin \phi):\R\to \mathbb{S}^1$, we recall that we sometimes write $E_\epsilon(\phi)$ instead of
$E_\epsilon(m)$. Note that a finite energy configuration $E_\eps(m)$ has the property that $m_1-\cos \alpha\in H^1(\R)$.
In particular, $m_1 \to \cos \alpha$ as $x_1\to \pm \infty$.
With some abuse of notation, we express this fact as $m_1(\pm \infty)=\cos \alpha$.
This implies that $m_2$ also has a limit at $\pm \infty$,  which belongs to $\{\pm \sin \alpha\}$.
Therefore, any continuous lifting $\phi$ of $m$ has limits at $\pm \infty$, belonging to $2\pi \Z\pm \alpha$.

When $\epsilon \searrow 0$, we expect to find limit configurations given by 
piecewise constant functions $\phi_0:\R\to 2\pi \Z \pm \alpha$ with finite total variation.
A lot of the terminology from Subsection \ref{subsect:confined} is adapted to this situation in an obvious manner.
We therefore use it here as well, including in particular the notation
$\iota(\phi_0)$ and $\eta(\phi_0)$ for a function $\phi_0$ as above.

\begin{theorem} 
\label{thm:unconfined_gamma_conv}
 The following holds true.

\begin{enumerate}

\item \textbf{Compactness and Separation.}
Suppose that $(\phi_\epsilon)_{\epsilon > 0}$ is a family of continuous functions in $\R$ such that
$\phi_\epsilon(-\infty) = \pm \alpha$ and $\phi_\epsilon(+\infty) \in 2\pi \Z \pm \alpha$ for all $\epsilon > 0$.
Suppose that
\[
\limsup_{\epsilon \searrow 0} |\log \epsilon| E_\epsilon(\phi_\epsilon) < \infty.
\]
Then there exist a sequence of numbers $\epsilon_k \searrow 0$
and a function $\phi_0:\R\to 2\pi \Z \pm \alpha$ of finite total variation such that $\phi_{\epsilon_k} \to \phi_0$ in 
$L_{\loc}^2(\R)$. If $N = \iota(\phi_0)$ and $\theta_N < \alpha < \pi - \theta_N$,
and if inequality \eqref{eqn:just_enough_energy} is satisfied, then $\phi_0$ is simple.

\item \textbf{Lower bound.} Let $(\phi_\epsilon)_{\epsilon > 0}$ be a family of continuous functions in $\R$ such that 
$\phi_\epsilon(-\infty) = \pm \alpha$, $\phi_\epsilon(+\infty) \in 2\pi \Z \pm \alpha$ for all $\epsilon > 0$ and $\phi_\eps\to \phi_0$ in $L_{\loc}^2(\R)$ for some simple, piecewise constant limit $\phi_0:\R\to 2\pi \Z \pm \alpha$. Suppose that $(a, d)$ is the transition profile of $\phi_0$ and $N = \iota(\phi_0)$. Then 
\[
E_\epsilon(\phi_\epsilon) \ge \frac{\eta(\phi_0) }{\log \frac{1}{\delta}} + \frac{\sum_{n = 1}^N e(d_n) + W(a, d)}{\left(\log \frac{1}{\delta}\right)^2} - o\left(\frac{1}{\left(\log \frac{1}{\delta}\right)^2}\right).
\]

\item \textbf{Upper bound.} If $\phi_0:\R\to 2\pi \Z \pm \alpha$ is simple with transition profile $(a,d)$,
and if $N = \iota(\phi_0)$, then there exists a family $(\phi_\epsilon)_{\epsilon > 0}$ of continuous functions in $\R$ such that
$\phi_\epsilon(-\infty) = \pm \alpha$, $\phi_\epsilon(+\infty) \in 2\pi \Z \pm \alpha$,  
$\phi_\eps\to \phi_0$ in $L^2(\R)$ and 
$$E_\epsilon(\phi_\epsilon)  \leq \frac{\eta(\phi_0) }{\log \frac{1}{\delta}} + \frac{\sum_{n = 1}^N e(d_n) + W(a, d)}{\left(\log \frac{1}{\delta}\right)^2} + o\left(\frac{1}{\left(\log \frac{1}{\delta}\right)^2}\right).$$

\end{enumerate}
\end{theorem}

The proof follows the same arguments as the proofs of Theorem \ref{thm:compactness_and_separation_confined} and Corollary \ref{cor:gamma_confined}. Since these arguments are based on Propositions \ref{prop:blow-up_of_renormalised_energy} and \ref{prop:lower_energy_bound_confined}, we need to discuss why they are, \emph{mutatis mutandis}, still
valid in the unconfined case.

\begin{itemize}
\item A version of Proposition \ref{prop:blow-up_of_renormalised_energy} for the unconfined model 
can be proved with the same arguments. The proof is now based on Lemmas \ref{lem:coefficients_for_alternating_signs} and \ref{lem:blow_up_of_logarithmic_function} and the fact that $I(|a_\ell-a_k|)\leq 1$ whenever $|a_\ell-a_k|\geq 1$.

\item  The proof of Proposition \ref{prop:lower_energy_bound_confined} relies mostly on
theory from \cite{Ignat-Moser:16}, the counterpart of which for the unconfined model has been developed above.
The additional ingredient is the energy estimate of Lemma \ref{lem:energy_in_annulus} (which is used in
Lemma \ref{lem:modification_of_u_{a,d}^*}, which in turn feeds into Lemma \ref{claimW} and finally Proposition
\ref{prop:lower_energy_bound_confined}). The inequality of Lemma \ref{lem:energy_in_annulus}, however, is easier
in the unconfined case and follows immediately from Proposition \ref{prop:conjugate_harmonic_functions}.

There is one result in \cite{Ignat-Moser:16} that has not explicitly been adapted to the unconfined case
in this paper, but is used in the proof of Proposition \ref{prop:lower_energy_bound_confined} (via Lemma \ref{claimW})
as well. This is \cite[Lemma 13]{Ignat-Moser:16}. It is readily checked, however, that the arguments for this result
carry over to the unconfined model.
\end{itemize}

Corollary \ref{cor:prescribed_winding_number_confined}, on the other hand, cannot be generalised directly
to the unconfined model. This becomes evident when we compare statement \ref{item:minimisers_topological}
with the results of Proposition \ref{pro:min_W_unconf}. Even if $\alpha \in \Theta$, the function $W(\blank, d_N^\pm)$
does not have a minimiser in general (or perhaps it never does). Hence a result such as statement
\ref{item:minimisers_topological} in Corollary \ref{cor:prescribed_winding_number_confined} is inconceivable
in the unconfined model.

Of course, Theorem \ref{thm:unconfined_gamma_conv} still gives some information about topological N\'eel walls in
the unconfined model, but this information will be less useful than in the confined case, because
we have to expect that part of the topology will disappear at $\pm \infty$ in the limit.

\bibliographystyle{plain}
\bibliography{references}

\end{document}